\documentclass{amsart}

\usepackage{amsfonts,amsmath,amssymb,amsthm}

\usepackage{float}

\usepackage{tikz}
\usepackage{fullpage}

\newcommand\R{\mathbb{R}}

\newcommand{\St}{{\mathcal St}}

        \newcommand{\HH}{{\mathcal H}}
        \renewcommand{\H}{\HH^1}
        
        \newcommand{\degre}{\ensuremath{^\circ}}
        \newcommand{\defeq}{:=}
        \newcommand{\forget}[1]{}

        \def\sup{\mathrm{sup}\,}    
        
        \def\deg{\mathrm{deg}\,}
        \def\Ntw{\mathrm{Ntw}\,}
        \def\dist{\mathrm{dist}\,}
        
        \def\diam{\mathrm{diam}\,}

        \def\X{\mathrm{X}}
        
        \def\S{\mathcal S}

        \def\diam{\mathrm{diam}\,}

        \def\turn{\mathrm{turn}}

        \def\Int{\mathrm{Int}}

\usetikzlibrary{%
  arrows,%
  calc,%
  patterns,%
  intersections,%
  shapes,
  positioning
}

\tikzset{arcnode/.style={
            decoration={
                        markings, raise = 2mm,
                        mark=at position 0.5 with {
                        \node[inner sep=0] {#1};
                        }
            },
            postaction={decorate}
      }
}

\tikzstyle{dotnode} = [draw, fill, inner sep = 0pt, minimum size = 1.1mm, circle]
\tikzstyle{termnode} = [draw, fill = white, inner sep = 0pt, minimum size = 1.5mm, circle]

\newtheorem{theorem}{Theorem}[section]
\newtheorem{definition}[theorem]{Definition}

\newtheorem{lemma}[theorem]{Lemma}

\newtheorem{corollary}[theorem]{Corollary}

\theoremstyle{remark}
\newtheorem{remark}[theorem]{Remark}
\newtheorem{example}[theorem]{Example}

\numberwithin{equation}{section}

\title{On the horseshoe conjecture for maximal distance minimizers}
\author{Danila Cherkashin \and Yana Teplitskaya}

\begin{document}
\maketitle

\begin{abstract}
We study the properties of sets $\Sigma$ having the minimal length (one-dimensional Hausdorff measure) over the class		
of closed connected sets $\Sigma \subset \mathbb{R}^2$ satisfying the inequality $\max_{y \in M} \dist(y,\Sigma) \leq r$ for a given compact
set $M \subset \mathbb{R}^2$ and some given $r > 0$. Such sets play the role of shortest possible pipelines arriving at a distance at most $r$ 
to every point of $M$, where $M$ is the set of customers of the pipeline.
	
We prove a conjecture of Miranda, Paolini and Stepanov that describes the set of minimizers for $M$ a circumference of radius $R>0$ for the case when $r < R/4.98$.
Moreover we show that when $M$ is the boundary of a smooth convex set with minimal radius of curvature $R$, then every minimizer $\Sigma$ has similar structure for $r < R/5$.
Additionaly, we prove a similar statement for local minimizers.
\end{abstract}



	\section{Preliminaries}

\subsection{Introduction}
For a given compact set $M \subset \mathbb{R}^2$ consider the functional
\[
	F_{M}(\Sigma)\defeq \sup _{y\in M}\dist (y, \Sigma),
\]
where $\Sigma$ is a subset of $\R^2$ and $\dist(y, \Sigma)$ stands for the Euclidean distance between $y$ and $\Sigma$ (naturally, $F_{M} (\emptyset) \defeq +\infty$). The quantity $F_M (\Sigma)$ will be called the \emph{energy} of $\Sigma$.
Consider the class of closed connected sets $\Sigma \subset \R^2$ satisfying $F_M(\Sigma) \leq r$ for some $r > 0$. We are interested in the properties of
sets of minimal length (one-dimensional Hausdorff measure) $\H(\Sigma)$ over the mentioned class. Such sets will be further called \emph{minimizers}. They can be viewed as shortest possible
pipelines arriving at a distance at most $r$ to every point of $M$ which in this case is considered as the
set of customers of the pipeline.

It is proven (in fact, even in the general $n$-dimensional case $M\subset \R^n$; see~\cite{mir} for the rigorous statement and details) that the set $OPT^*_\infty(M)$ of minimizers (for all $r>0$) is nonempty and coincides with the set $OPT_\infty(M)$ of solutions of the dual problem: minimize $F_M$ over all compact connected sets  $\Sigma \subset \mathbb{R}^2$ with prescribed bound on the total length $\H(\Sigma) \leq l$.
The latter minimizing problem is quite similar to many other problems of minimizing other functionals over closed connected sets, 
for instance the average distance with respect
to some finite Borel measure (see~\cite{BOS, BS1, BS2, L, XS}) or similar urban planning problems (see~\cite{BPSS}). 
If one minimizes maximum or average distance functional over discrete sets with an a priori restriction on the number of connected components (rather than over connected one-dimensional sets) one gets another class of closely related problems
known as $k$-center problem and $k$-median problem (see e.g.~\cite{GL, SD, SO} as well as~\cite{BouchJimMahad11, BrabutSamSte09} and references therein).

Some basic properties of minimizers for the above mentioned problem in $n$-dimensional case (like the absence of loops and Ahlfors regularity) have been proven in~\cite{PaoSte04max}. Further, in~\cite{mir} the following characterization of minimizers has been studied.
Let $B_\rho (x)$ be the open ball of radius $\rho$ centered at a point $x$, and let $B_\rho(M)$ be the open $\rho$-neighborhood of $M$ i.e.\
	\[
        		B_\rho(M) \defeq \bigcup_{x\in M} B_\rho(x).
	\]
Further, we introduce

\begin{definition}
A point $x \in \Sigma$ is called \emph{energetic}, if for all $\rho>0$ one has
 	\[
		F_{M}(\Sigma \setminus B_{\rho}(x)) > F_{M}(\Sigma).
	\]
\end{definition}
Denote the set of all energetic points of $\Sigma$ by $G_\Sigma$.

Let us consider a minimizer $\Sigma$ with energy $F_{M}(\Sigma) = r$ (the subset of  $OPT_{\infty}^{*}(M)$ of minimizers with energy $r$ will be further denoted by $OPT_{\infty}^{*}(M, r)$). Then the set $\Sigma$ can be split into three disjoint subsets:
        \[
        \Sigma=E_{\Sigma}\sqcup\X_{\Sigma}\sqcup\S_{\Sigma},
        \]
        where $X_{\Sigma}\subset G_\Sigma$ 
is the set of isolated energetic points 
(i.e.\ every $x\in X_\Sigma$ is energetic and there is a $\rho>0$ possibly depending on $x$ such that $B_{\rho}(x)\cap G_{\Sigma}=\{x\}$), $E_{\Sigma} := G_{\Sigma}\setminus X_{\Sigma}$ is the set of non isolated energetic points and $S_\Sigma \defeq \Sigma \setminus G_\Sigma$ is the set of non energetic points also called the Steiner part of $\Sigma$.
        In~\cite{mir} the following assertions have been proven:
        \begin{itemize}
        \item[(a)] For every point $x \in G_\Sigma$ there exists a point $Q_x \in M$ (possibly non unique) such that $\dist (x, Q_x) = r$ and $B_{r}(Q_x)\cap \Sigma=\emptyset$. If $X_\Sigma$ is not finite, the limit points of $X_\Sigma$ belong to $E_\Sigma$.
        \item[(b)] For all $x \in S_\Sigma$ there exists an $\varepsilon>0$ such that $S_\Sigma \cap B_{\varepsilon}(x)$ is either a line segment or a regular tripod, i.e.\ the union of three line segments with an endpoint in $x$ and relative angles of $2\pi/3$. If a point $x \in S_\Sigma$ is a center of a regular tripod, then it called a \textit{Steiner point} (or a \textit{branching point}) of $\Sigma$.
        \end{itemize}

\noindent Note that the finiteness of $\H (\Sigma)$ implies that $\Sigma$ is path-connected (see, for example,~\cite{EiSaHa}).
By the absence of loops the path in $\Sigma$ between every couple of points of $\Sigma$ is unique. Also the condition on curvature of $M_r$ and convexity of $N$ imply $C^{1,1}$ smoothness of $M_r$.

This paper is organized as follows: Subsection~\ref{stp} contains a very brief survey on the Steiner problem, Subsection~\ref{not} describes basic notations, Section~\ref{mr} is devoted to the statement of the main result with Subsection~\ref{op} containing a sketch of the proof of the main theorem (Theorem~\ref{theo1}), Section~\ref{pr} contains complete proofs of lemmas stated in Section~\ref{mr}.

	\subsection{Steiner problem} \label{stp}

The Steiner problem which has several different but more or less equivalent formulations, is that of finding a set $S$ with minimal length (one-dimensional Hausdorff measure $\H(S)$) such that $S \cup A$ is connected, where $A$ is a given compact subset of a given complete metric space $X$.

Namely, if we define
\[
\Ntw(A) \defeq \{S \subset X: S \cup A \mbox{ is connected} \}
\]
then the Steiner problem is to find an element of $\Ntw(A)$ with minimal $\H$-length.
 
If $S$ is a solution to the Steiner problem for a given set $A$ (in the case when $X$ is proper and connected a solution exists~\cite{PaoSte13Steiner}), then the set $\Sigma \defeq A \cup S$ is called a Steiner tree for the set $A$ (or a Steiner tree connecting the set $A$, or just a Steiner set). 
It has been proven in~\cite{PaoSte13Steiner} that in the case $\H(S) < +\infty$ the following properties hold: 

\begin{enumerate}
 \item $\overline{S}$ contains no loops (homeomorphic images of $S^1$);
 \item $S\setminus A$ has at most countably many connected components, and each of the latter has strictly positive length;
 \item the closure of every connected component of $S \setminus A$ is a topological tree with endpoints on $A$ and with at most one endpoint belonging to each connected component of $A$. 
\end{enumerate}

From now on we will consider the Steiner problem in the case when the ambient space $X$ is the Euclidean plane $\R^2$. Then

\begin{enumerate}
\addtocounter{enumi}{3}
 \item $\Sigma \setminus A$ consists of line segments (this follows from the result of ~\cite{PaoSte13Steiner} stating that away from the data every Steiner tree is an embedded graph consisting of geodesic segments); 
 \item the angle between two segments adjacent to the same vertex is greater or equal to $2\pi/3$~\cite{ivanov1994minimal}.
 \item Let us call \textit{a Steiner (or branching) point} such a point of $\Sigma$ that does not belong to $A$ and which is not an interior point of a segment of $\Sigma$.
 The degree (in the graph theoreric sense) of a Steiner point $x$ is equal to $3$. In this case the angle between any pair of segments adjacent to $x$ is equal to $2\pi/3$ (see~\cite{ivanov1994minimal}). Such a set is called \textit{regular tripod}.
 \item It is well-known that for $A = \{x_1, x_2, x_3\} \subset \R^2$ there is the unique solution of the Steiner problem. We denote it by $\St (x_1, x_2, x_3)$. 
\end{enumerate}

We will say that a set $S \in \Ntw(A)$ is a \textit{locally minimal network} for the given set $A$ if for an arbitrary point $x\in S$ there exists a neighbourhood $U \ni x$ such that $S \cap \overline{U}$ is a Steiner tree for $S \cap \partial U$. If a neighbourhood of a point $x \in S$ is a regular tripod then it is still called a \textit{Steiner point}. 

A locally minimal network satisfies all the properties of a Steiner tree mentioned above except the first one (see~\cite{PaoSte13Steiner, ivanov1994minimal}).

In this paper we will use the locally minimal networks for a set $A$ that consists of at most four points. It is well known that there are only $7$ possible combinatorial types of such networks which one can find on Figures~\ref{lmn23} and~\ref{lmn4} (see~\cite{PaoSte13Steiner, ivanov1994minimal}).

\subsection{Notation} \label{not}

We introduce the following notation.

\begin{itemize}
    \item For a given set $X \subset \R^2$ we denote by $\overline{X}$ its closure, by $\Int (X)$ its interior and by $\partial X$ its topological boundary.
    \item For given points $B$, $C$ we use the notation $[BC]$, $[BC)$ and $(BC)$ for the corresponding (closed) line segment, ray and line respectively.
    We denote by $]BC]$ and $]BC[$ the corresponding semiopen and open segments, and by $|BC|$ the length of these segments.
    \item By a \textit{closed convex curve} we mean a \textit{boundary of a convex compact set}.
    \item We call a \textit{chord} of a closed convex curve $Z$ a line segment connecting two points of $Z$. 
    \item A subset of a planar curve $Z$ is called an \textit{arc} of $Z$ if it is a continuous injective image of an interval (possibly degenerate). We say that an arc of $Z$ is \emph{closed}, if it is a relatively closed subset of $Z$. The images of the endpoints of the interval will be called \textit{ends} of the arc; the images of internal points of the interval will be called \textit{internal} points of the arc. Whenever there is no confusion the closed arc with ends $B$, $C$ will be denoted by $[\breve {BC}]$ and its length by $|\breve{BC}|$ (not to be confused with the length of the segment connecting $B$ and $C$ which is denoted by $|BC|$).  
    
    \item For a convex closed set $N \subset \R^2$ we define the \emph{minimal radius of curvature} 
of its boundary by the formula
    \[
R (\partial N) \defeq \inf_{x \in \partial N} \sup \{\rho \colon 
\overline{B_\rho(O)} \cap \partial N  = x \text{ for some } O \in N\}.
\]
    \item For a convex closed set $N \subset \R^2$ we define the \emph{inner set} $N_r$ to be the
set of all points of $N$ lying at a distance of at least $r$ from the boundary, namely, $N_r \defeq N \setminus B_r(\partial N)$.

    \item From now on we define $N \defeq \text{conv}\,(M)$, where $\textit{conv}$ stands for the closed convex hull,
    and $M_r \defeq \partial N_r$. 
    Note that $N$, $N_r$, $M$ and $M_r$ are closed sets. 
\end{itemize}

From now $M$ is a convex closed curve with minimal radius of curvature $R > r$.
Clearly, $M_r$ is convex closed curve and has minimal radius of curvature at least $R-r$.

\begin{itemize}
    
    \item If a line segment $[BC]$ is an arc of $M_r$ and a chord of $M_r$ simultaneously (this happens if $M_r$ is not strictly convex) we will work with $[BC]$ as with the arc in the case when $]BC[$ has an energetic point, and as with the chord otherwise. In Section 2 it will be explained in details.

    \item We say that an arc $\breve{[BC]} \subset \Sigma$ of $M_r$ \textit{is continued by a chord} in the set $\Sigma$ if for some $J \in \{B, C\}$ there is a chord $[JD]$ of $M_r$ such that $[JD] \subset \Sigma$.
    \item For a point $x \in \Sigma \cap M_r$ let $Q_x \in M$ be such a point that $\dist (x, Q_x) = r$ (in this case $Q_x$ is unique because of the condition on curvature of $M$). 
    Also, in this case $[xQ_x]$ is a part of the normal to $M_r$ at $x$ and of the normal to $M$ at $Q_x$.
    \item For an energetic point $x \in G_\Sigma$ let $Q_x$ be a point mentioned in the property (a) of the set of energetic points. We may consider this choice of $Q_x \in M$ as a canonical choice of a point on $M$ at the distance $r$ from $x$.
    \item We say that a set $Z \subset \R^2$ \textit{covers} a subset $Q \subset M$ if $Q \subset \overline {B_r (Z)}$. 
    Usually we use the latter notion for the case when $Q$ is an arc of $M$. 
    \item For a set $Z \subset \mathbb{R}^2$ we define the \textit{diameter} of $Z$ as $\sup \{\dist (x,y)| x,y \in Z \}$, and denote it by $\diam (Z)$.
    \item We fix the clockwise orientation of the plane.
    \item For rays $[BC)$, $[CD)$ let $\angle([BC), [CD))$ stand for the \textit{directed angle} from $[BC)$ to $[CD)$ with respect to the clockwise orientation.
    \item 
    When using the asymptotic expressions $o(\cdot)$, $O(\cdot)$, we will always be silently assuming that the respective variable tends to some limit; both the variable and the limit will be usually clear from the context (if it necessary to avoid confusion, the variable name will be indicated in the lower index of the asymptotic symbols).
    
\end{itemize}
	\section {Main results} \label{mr}

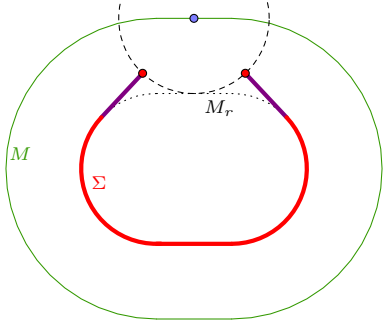
\begin{figure}
	\begin{center}
		\definecolor{yqqqyq}{rgb}{0.5019607843137255,0.,0.5019607843137255}
		\definecolor{xdxdff}{rgb}{0.49019607843137253,0.49019607843137253,1.}
		\definecolor{ffqqqq}{rgb}{1.,0.,0.}
    \definecolor{ttzzqq}{rgb}{0.2,0.6,0.}
    \begin{tikzpicture}[line cap=round,line join=round,>=triangle 45,x=1.0cm,y=1.0cm]
    \clip(-5.754809259201689,-2.15198488708635626) rectangle (0.492320066152345,2.2268250892462824);
    \draw [color=ttzzqq] (-3.,2.)-- (-2.,2.);
    \draw [color=ttzzqq] (-2.,-2.)-- (-3.,-2.);
    \draw [line width=1.6pt,color=ffqqqq] (-3.,-1.)-- (-2.,-1.);
    \draw [dotted] (-3.,1.)-- (-2.,1.);
    \draw [shift={(-3.,0.)},color=ttzzqq]  plot[domain=1.5707963267948966:4.71238898038469,variable=\t]({1.*2.*cos(\t r)+0.*2.*sin(\t r)},{0.*2.*cos(\t r)+1.*2.*sin(\t r)});
    \draw [shift={(-3.,0.)},dotted]  plot[domain=1.5707963267948966:4.71238898038469,variable=\t]({1.*1.*cos(\t r)+0.*1.*sin(\t r)},{0.*1.*cos(\t r)+1.*1.*sin(\t r)});
    \draw [shift={(-2.,0.)},dotted]  plot[domain=-1.5707963267948966:1.5707963267948966,variable=\t]({1.*1.*cos(\t r)+0.*1.*sin(\t r)},{0.*1.*cos(\t r)+1.*1.*sin(\t r)});
    \draw [shift={(-2.,0.)},color=ttzzqq]  plot[domain=-1.5707963267948966:1.5707963267948966,variable=\t]({1.*2.*cos(\t r)+0.*2.*sin(\t r)},{0.*2.*cos(\t r)+1.*2.*sin(\t r)});
    \draw [dash pattern=on 2pt off 2pt] (-2.5008100281931047,2.) circle (1.cm);
    \draw [line width=1.6pt,color=yqqqyq] (-3.7309420990521054,0.6824394829091469)-- (-3.1832495111022516,1.2690579009478955);
    \draw [line width=1.6pt,color=yqqqyq] (-1.8178907146013357,1.2695061868000874)-- (-1.2695061868000874,0.6829193135917684);
    \draw [shift={(-2.,0.)},line width=1.6pt,color=ffqqqq]  plot[domain=-1.5707963267948966:0.7517515639553677,variable=\t]({1.*1.*cos(\t r)+0.*1.*sin(\t r)},{0.*1.*cos(\t r)+1.*1.*sin(\t r)});
    \draw [shift={(-3.,0.)},line width=1.6pt,color=ffqqqq]  plot[domain=2.390497746093128:4.771687465439039,variable=\t]({1.*1.*cos(\t r)+0.*1.*sin(\t r)},{0.*1.*cos(\t r)+1.*1.*sin(\t r)});
    \begin{scriptsize}
    \draw[color=ttzzqq] (-4.8089822121844977,0.2025724503290166) node {$M$};
    \draw [fill=xdxdff] (-2.5008100281931047,2.) circle (1.5pt);
    \draw [fill=ffqqqq] (-3.1832495111022516,1.2690579009478955) circle (1.5pt);
    \draw [fill=ffqqqq] (-1.8178907146013357,1.2695061868000874) circle (1.5pt);
    \draw[color=ffqqqq] (-3.763965847825129,-0.1818662773850787) node {$\Sigma$};
    \draw[color=black] (-2.159620164509451,0.8028001329511188) node {$M_r$};
    \end{scriptsize}
    \end{tikzpicture}
    \caption{A horseshoe.}
    \label{horseFig}
	\end{center}
    \end{figure}

    \begin{definition}
    Let  $M$ be a closed convex curve with minimal radius of curvature $R > r$.
    Then the connected curve $\Sigma$ is called a \emph{horseshoe}, 
if $F_M (\Sigma) = r$ and $\Sigma$ is a union of an arc $q$ of $M_r$ with two 
non degenerate tangent segments to $M_r$ at the different ends of $q$ ending with energetic points (as shown in Fig.~\ref{horseFig}).
    \end{definition}

    The following theorem proves the particular case of the conjecture of Miranda, Paolini and Stepanov 
from~\cite{mir} about the set $OPT_{\infty}^{*}(M, r)$ of minimizers for $M := \partial B_R(O)$ if $r < R/4.98$.
    It shows even more: namely, that \emph{every} closed convex curve $M$ has minimizers 
of the same structure, if minimal radius of curvature of $M$ is at least $5r$.
Note however that it does not prove the whole conjecture (which has been formulated in~\cite{mir} for circumference $M = \partial B_R(O)$ and every $r< R$).

    \begin{theorem}\label{theo1}	
    For every closed convex curve $M$ with minimal radius of curvature $R$ and for every 
$r < R/5$ the set of minimizers $OPT_{\infty}^{*}(M, r)$ contains only horseshoes.
    For the circumference $M = \partial B_R(O)$ the claim is true for $r < R/4.98$.
    \end{theorem}


    \begin{definition}
    Let $M \subset \mathbb{R}^2$ be a planar compact set. A connected set $\Sigma$ is called \emph{local minimizer} if it covers $M$
    and there is an $\varepsilon > 0$ such that for every connected $\Sigma'$ covering $M$ and satisfying
    $\diam (\Sigma \Delta \Sigma') \leq \varepsilon$ one has
    $\H (\Sigma) \leq \H (\Sigma')$.
    \label{local}
    \end{definition}

    \begin{corollary}
    Let $\hat \Sigma$ be a local minimizer for some closed convex curve $M$ with minimal radius of curvature $R > 5r$.
    Then if $\hat \Sigma$ is not a horseshoe, one has $\H (\hat \Sigma) - \H (\Sigma) \geq (R-5r)/2$, 
    where $\Sigma \in OPT_{\infty}^{*}(M, r)$ is an arbitrary (global) minimizer. 
    \label{coroll}
    \end{corollary}
    
    \begin{figure}
\begin{center}
\definecolor{ffqqqq}{rgb}{1.,0.,0.}
\definecolor{uuuuuu}{rgb}{0.26666666666666666,0.26666666666666666,0.26666666666666666}
\definecolor{wwzzqq}{rgb}{0.4,0.6,0.}
\begin{tikzpicture}[line cap=round,line join=round,>=triangle 45,x=1.0cm,y=1.0cm]
\clip(-2.0746840397001493,-2.2681302319585696) rectangle (14.457332526577204,2.56210111714977153);
\draw [line width=1.6pt,color=wwzzqq] (11.88,2.)-- (0.,2.);
\draw [line width=1.6pt,color=wwzzqq] (0.,-2.)-- (11.88,-2.);

\draw [shift={(0.,0.)},line width=1.6pt,color=wwzzqq]  plot[domain=1.5707963267948966:4.71238898038469,variable=\t]({1.*2.*cos(\t r)+0.*2.*sin(\t r)},{0.*2.*cos(\t r)+1.*2.*sin(\t r)});
\draw [shift={(11.88,0.)},line width=1.6pt,color=wwzzqq]  plot[domain=-1.5707963267948966:1.5707963267948966,variable=\t]({1.*2.*cos(\t r)+0.*2.*sin(\t r)},{0.*2.*cos(\t r)+1.*2.*sin(\t r)});

\draw (0.,-0.2)-- (11.88,-0.2);

\draw (3.3050673596736084,1.3940269438694433) -- (0, 0.2);
\draw [shift={(0.,0.)}] plot[domain=1.5707963267948966:4.71238898038469,variable=\t]({1.*0.2*cos(\t r)+0.*0.2*sin(\t r)},{0.*0.2*cos(\t r)+1.*0.2*sin(\t r)});
\draw [shift={(11.88,0.)}] plot[domain=-1.5707963267948966:1.5707963267948966,variable=\t]({1.*0.2*cos(\t r)+0.*0.2*sin(\t r)},{0.*0.2*cos(\t r)+1.*0.2*sin(\t r)});
\d
raw (0.,0.2)-- (3.3050673596736084,1.3940269438694433);
\draw (6.6949326403263925,1.3940269438694481)-- (11.88,0.2);
\draw [dash pattern=on 3pt off 3pt] (0.,0.2)-- (11.88,0.2);

\draw [line width=1.6pt,color=ffqqqq] (1.98 ,0.8568464670045404)-- (0.,2.);
\draw [line width=1.6pt,color=ffqqqq] (1.98,0.8568464670045404)-- (3.96,2.);
\draw [line width=1.6pt,color=ffqqqq] (1.98,0.8568464670045404)-- (1.98,-2.);

\draw [line width=1.6pt,color=ffqqqq] (5.94,0.8568464670045406)-- (5.94,-2.);
\draw [line width=1.6pt,color=ffqqqq] (5.94,0.8568464670045406)-- (7.92,2.);
\draw [line width=1.6pt,color=ffqqqq] (5.94,0.8568464670045406)-- (3.96,2.);

\draw [line width=1.6pt,color=ffqqqq] (9.90,0.8568464670045406)-- (9.90,-2.);
\draw [line width=1.6pt,color=ffqqqq] (9.90,0.8568464670045406)-- (11.88,2.);
\draw [line width=1.6pt,color=ffqqqq] (9.90,0.8568464670045406)-- (7.92,2.);

\draw [line width=1.6pt,color=ffqqqq] (13.88,0.)-- (11.88,2.);
\draw [line width=1.6pt,color=ffqqqq] (13.88,0.)-- (11.88,-2.);

\draw [line width=1.6pt,color=ffqqqq] (-2.,0.)-- (0.,2.);
\draw [line width=1.6pt,color=ffqqqq] (-2.,0.)-- (0.,-2.);

\draw (1.5436991891648444,2.680758446633948) node[anchor=north west] {$|BC| = 2r$};
\begin{scriptsize}
\draw[color=uuuuuu] (5.065720988074994,-0.485192722912954) node {$\Sigma_1$};
\draw[color=ffqqqq] (2.445720988074994,-1.185192722912954) node {$\Sigma'$};

\draw[color=wwzzqq] (5.065720988074994,-2.185192722912954) node {$M$};
\draw [fill=uuuuuu] (3.3050673596736084,1.3940269438694433) circle (1.5pt);
\draw [fill=uuuuuu] (6.6949326403263925,1.3940269438694481) circle (1.5pt);
\draw [fill=uuuuuu] (0.,2.) circle (1.5pt);
\draw [fill=uuuuuu] (11.88,2.) circle (1.5pt);
\draw [fill=uuuuuu] (7.92,2.) circle (1.5pt);
\draw [fill=uuuuuu] (9.90,-2.) circle (1.5pt);
\draw [fill=ffqqqq] (9.90,0.8568464670045406) circle (1.5pt);

\draw [fill=uuuuuu] (13.88,0.) circle (1.5pt);
\draw [fill=uuuuuu] (11.88,-2.) circle (1.5pt);
\draw [fill=uuuuuu] (0.,-2.) circle (1.5pt);
\draw [fill=uuuuuu] (-2., 0.) circle (1.5pt);

\draw [fill=uuuuuu] (11.88,0.) circle (1.5pt);
\draw[color=uuuuuu] (11.3,0.) node {$(t, 0)$};

\draw [fill=uuuuuu] (0.,0.) circle (1.5pt);
\draw[color=uuuuuu] (0.5,0.) node {$(0, 0)$};

\draw[color=uuuuuu] (0.1534274264371537,2.333190505952026) node {$B$};
\draw [fill=uuuuuu] (7.92,2.) circle (1.5pt);
\draw [fill=uuuuuu] (1.98,-2.) circle (1.5pt);
\draw [fill=uuuuuu] (5.94,-2.) circle (1.5pt);
\draw [fill=ffqqqq] (1.98,0.8568464670045404) circle (1.5pt);
\draw [fill=ffqqqq] (5.94,0.8568464670045406) circle (1.5pt);
\draw [fill=ffqqqq] (5.94,0.8568464670045406) circle (1.5pt);
\draw [fill=ffqqqq] (5.94,0.8568464670045406) circle (1.5pt);
\draw [fill=uuuuuu] (3.96,2.) circle (1.5pt);
\draw[color=uuuuuu] (4.115701950211072,2.333190505952026) node {$C$};
\end{scriptsize}
\end{tikzpicture}

\caption{The horseshoe $\Sigma_1$ and the better competitor $\Sigma'$ in the Example~\ref{stadium}.}

\label{stad}
\end{center}
\end{figure}
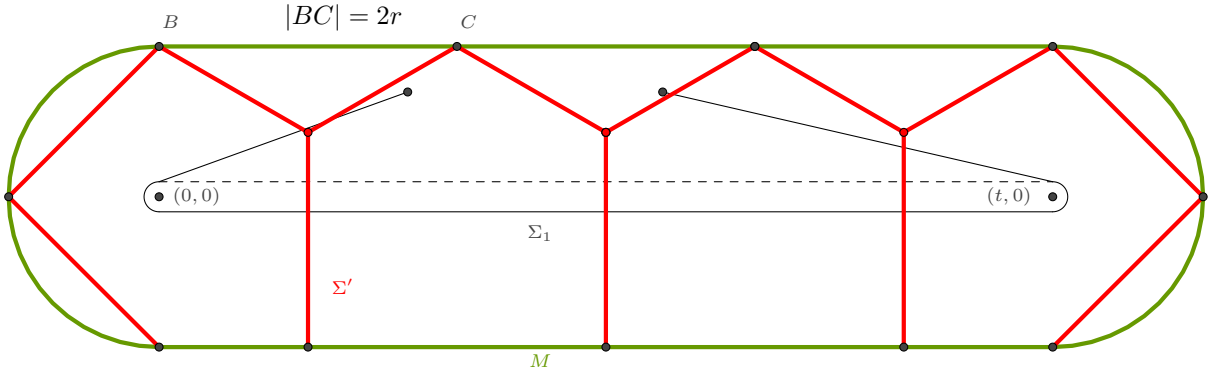
    
    It is worth mentioning that the claim of Theorem~\ref{theo1} does not hold without some assumptions on the dependence of $R$ on $r$, as the following example shows.
    \begin{example}[see Fig.~\ref{stad}] \label{stadium}
    Consider the stadium (see Fig.~\ref{stad}) $$N^t := \overline{\cup_{x \in [0,t] \times \{0\}} B_1 (x)},\ M^t := \partial N^t.$$
    Note that $M = M^t$ is the border of the stadium and has minimal radius of curvature $1$ for every $t$. 
    Let us choose $1 > r > 1 - \varepsilon$ for sufficiently small $\varepsilon$, and a sufficiently large $t$ such that $t/r \in 2\mathbb{N}$. Clearly, any horseshoe has length $2t - O(1)$ as $t \to \infty$. 
Consider the points $x_{2k} = (2kr, 1), x_{2k + 1} = ((2k+1)r, -1) \in M^t$ for $k = 0, 1, \dots, t/(2r)$. Let $X$ be the union of the sets $\St (x_{2k}, x_{2k+1}, x_{2k+2})$ for $k = 0, 1, \dots, t/(2r) - 1$. Every such tree is a tripod; its length tends to $2 + \sqrt{3}$ when $\varepsilon \to 0+$. 
Note that 
$$\Sigma' := X \cup [(0, 1), (-1, 0)] \cup [(-1, 0), (0, -1)] \cup [(t, 1), (t + 1, 0)] \cup [(t + 1, 0), (t, -1)]$$
is connected and $M^t \subset \overline{B_r(\Sigma')}$
(it is true, because $\dist (x_i, x_{i+2}) = 2r$, so the set $\{x_i\}$ covers horizontal lines of $M^t$; four additional segments cover semicircles of $M^t$).
The length of $X$ is $(2 + \sqrt{3} + o_\varepsilon(1))\frac{t}{2r} + 4\sqrt{2} \leq (1 + \frac{\sqrt{3}}{2} + o_\varepsilon(1)) \frac{t}{r} < 2t - O_t(1)$, and therefore one can choose $\varepsilon$ sufficiently small and $t$ sufficiently large such that $\Sigma'$ is a better competitor than a horseshoe, hence a horseshoe is not a global minimizer for $M^t$.
\end{example}

\subsection{The outline of the proof} \label{op}

Here is the sketch of the proof of Theorem~\ref{theo1}. 
Let us recall that in~\cite{PaoSte04max} the following statement is proven.

\begin{lemma}
Let $M \subset \mathbb{R}^2$ be a compact set, $\Sigma \in OPT_{\infty}^{*}(M)$.
Then $\Sigma$ has no loops.    
\label{cycle}
\end{lemma}

In the sequel the union of the closures of all connected components of $\Sigma \cap \Int (N_r)$ is denoted by $\Sigma_r$.

    \begin{lemma}  \label{global_lm}
    Let $M$ be a convex closed curve with minimal radius of curvature $R$ 
and $\Sigma \in OPT_{\infty}^{*}(M, r)$ be a minimizer with the energy $r < R$. Then the following assertions hold.
    \begin{itemize}
    \item[(i)] The closure of every connected component of $\Sigma \cap \Int(N_{r})$ is a solution of the Steiner problem for some set of points belonging to $M_{r}$, and in particular consists only of line segments of positive length.
    \item[(ii)] $\Sigma$ consists of arcs of $M_r$ (possibly degenerate) and line segments of positive length with disjoint interiors. 
    \item[(iii)] The length of each line segment in $\Sigma_r$
does not exceed $a_M(r)$ for some $a_M(r)\leq 2r$. For the circumference $\partial B_R(O)$ one can take $a_{\partial B_R(O)} (r) = 2r\sqrt{1 - \frac{r^2}{4R^2}}$.  
    \end{itemize}
    \end{lemma}

Note that we do not show in this Lemma that the number of line segments in $\Sigma$ is finite.

A statement similar to Lemma~\ref{global_lm} may be proven for $M$ being a boundary of a not necessarily convex set, but we restrict the statement to the convex case to avoid excessive technicalities.

    \begin{lemma}\label{hex}
    Let $M$ be a closed convex curve with minimal radius of curvature
$R > 2a_M(r) + r$, where $a_M$ is such that the length of each line segment in $\Sigma_r$
does not exceed $a_M(r)$ (in particular one can take $a_M$ as in Lemma~\ref{global_lm}),
$\Sigma \in OPT_{\infty}^{*}(M, r)$. 
Then $\Sigma$ has no Steiner point in $\Int (N_r) \cup (S_\Sigma \cap M_r)$.
Thus $\Sigma \cap \Int(N_r)$ consists of disjoint interiors of chords of $M_r$.
    \end{lemma}

Let us consider the set of the closures of connected components of $\Sigma \setminus N_r$. Denote it by $V_C(G)$ (further it will be associated with a subset of the vertex set of a graph). Note that $\Sigma$ is connected (and does not reduce to a single point), so every $S \in V_C(G)$ has positive length. In our setting $M$ is compact, thus every $\Sigma \in OPT_{\infty}^{*}(M, r)$ has finite length, hence the set $V_C(G)$ is at most countable. 

Consider an arbitrary $S \in V_C(G)$. Note that by connectedness of $S$ the set $\overline {B_r (S)} \cap M$ is always a closed arc. We denote it by $q_S$. 

Consider the set of all maximal arcs of $M_r$ in the set $\Sigma$, which are not contained in the closure of a connected component of $\Sigma \setminus N_r$. 
Let us denote by $V_A(G)$ the subset of such arcs having an energetic point in their interior.  
Note that if $M$ is not strictly convex, then an arc $\breve{[BC]}$ of $M_r$ can be a chord of $M_r$.
In this situation if $\breve{]BC[}$ has no energetic point then we will consider it as a chord of $M_r$:  note that if $\Sigma \setminus \breve{]BC[}$ does not cover $Q_x \in M$ for some $x \in ]BC[$, then $x$ is energetic; thus if $\breve{]BC]}$ has no energetic point then $[BC] = \breve{[BC]}$ has all the properties of a standard chord of $M_r$.

Obviously, an arc $\breve{[BC]} \in V_A(G)$ of $M_r$ covers an arc $q_{\breve{[BC]}} \defeq \breve{[Q_B Q_C]}$ of $M$, where $Q_B, Q_C \in M$ are the unique points such that $\dist (B, Q_B) = \dist (C, Q_C) = r$. 

\begin{definition}
Let $M$ be a closed convex curve with minimal radius of curvature
$R > r$, $\Sigma \in OPT_{\infty}^{*}(M, r)$. Let $S \in V_C(G)$, a closure of a connected component of $\Sigma \setminus N_r$.

\begin{itemize}
\item[(i)] Denote by $n(S)$ the number of energetic points in $S$.   

\item[(ii)] A point $x \in S \cap M_r$ is called an \emph{entering point}.
Denote the number of entering points of $S$ by $m(S)$.
\end{itemize}
\end{definition}

The following lemma says in particular that $n(S)$, $m(S)$ are finite.

\begin {lemma}
Let $M$ be a closed convex curve with minimal radius of curvature $R > 2a_M(r) + r$, $\Sigma \in OPT_{\infty}^{*}(M, r)$. 
Let $S$ be the closure of a connected component of $\Sigma \setminus N_r$. Then $n(S) \leq 2$, $m(S) \leq 2$. 
Further, $S$ is a locally minimal network connecting the set of entering points of $S$ and energetic points of $S \setminus M_r$.
\label{compcon}
\end {lemma}

By the previous Lemma, $S$ is a locally minimal network for at most $n(S) + m(S) \leq 4$ points. 
All the possible combinatorial types of such networks are listed in Figures~\ref{lmn23} and~\ref{lmn4}.

\begin{lemma}
Under conditions of Theorem~\ref{theo1} if $S \in V(G) \defeq V_C(G) \sqcup V_A(G)$ does not reduce to a point, then 
$$q_S \not \subset \bigcup_{S' \in V(G) \setminus \{S\}} q_{S'}.$$
Moreover, every set $S \in V(G)$ has an energetic point.
\label{dugi}
\end{lemma}

\begin{lemma}
Under conditions of Theorem~\ref{theo1} the set $V(G) = V_C(G) \sqcup V_A(G)$ is finite.
\label{accum}
\end{lemma}

Note that a singleton of $\Sigma \cap M_r$ (a maximal arc $\xi \subset \Sigma \cap M_r$ of zero length not contained in the closure of a connected component of $\Sigma \setminus N_r$) cannot be energetic (by the previous Lemma the union of $q_S$ over $S \in V(G) \setminus \xi$ is closed as a finite union of closed sets, hence it coincides with $M$ because $q_\xi = \{ Q_\xi \}$), so a neighbourhood of $\xi$ is a segment or a tripod (the latter is impossible by Lemma~\ref{hex}).
Summing up, every point of $\Sigma \cap M_r$ is contained in a maximal arc of $M_r$ of positive length or in the closure of a connected component of $\Sigma \setminus N_r$. Also by Lemma~\ref{compcon} every connected component of $\Sigma \setminus N_r$ contains at most 5 segments, thus $\Sigma$ consists of a finite number of segments and arcs of $M_r$.

\begin{lemma}
Under conditions of Theorem~\ref{theo1} let $[BI] \subset \Sigma$ be a chord of $M_r$.
Then $I \in S_\Sigma$ and moreover there exists such an $\varepsilon>0$ that $\overline{B_\varepsilon(I)} \cap \Sigma = [I_1I_2]$, for some $I_1, I_2 \in \partial B_\varepsilon (I)$. 
\label{bpMr}
\end{lemma}

\begin{lemma}
    Under conditions of Theorem~\ref{theo1} every maximal arc $\breve{[BC]} \in V_A(G)$ is continued by segments lying on tangent lines to $M_r$ in the sense that
    there exists such an open $U \supset \breve{[BC]}$ that $\Sigma \cap \overline{U} = [B'B] \cup \breve{[BC]} \cup [CC']$, where $[B'B]$ and $[CC']$ are subsets of tangent lines to $M_r$ at points $B$, $C$ respectively. 
    \label{27}
\end{lemma}


\begin{lemma}
Under conditions of Theorem~\ref{theo1} let $C \in M_r \cap \Sigma$. Then $\Sigma$ has the tangent line at $C$, in particular for every $\varepsilon > 0$ there is a $\delta > 0$ such that for every couple of points $B,D \in \Sigma \cap B_\delta (C) \setminus C$, holds $\min (|\angle BCD - \pi|, |\angle BCD|) < \varepsilon$.
\label{debil}    
\end{lemma}

\begin{definition}
Under conditions of Theorem~\ref{theo1} consider the following abstract graph $G = (V(G), E(G))$ (recall that the set of vertices $V (G) = V_C(G) \sqcup V_A(G)$; by Lemma~\ref{accum} it is finite), 
where the set of edges $E(G)$ is defined as follows:

\begin{itemize}
\item in the case $S_1$, $S_2 \in V_C(G)$ there is an edge between them if they are connected in $\Sigma$ by a chord of $M_r$ or if $S_1 \cap S_2 \neq \emptyset$;
\item in the case $S_1 \in V_C(G)$, $\breve{[BC]} \in V_A(G)$ there is an edge between $S_1$ and $\breve{[BC]}$, if $S_1 \cap \breve{[BC]} \neq \emptyset$; 
\item and finally in the case $\breve{[B_1C_1]}$, $\breve{[B_2C_2]} \in V_A(G)$ there is no edge between them.
\end{itemize}
\end{definition}

\begin{corollary}\label{co_graph1}
Under conditions of Theorem~\ref{theo1} graph $G$ has no cycles; it has exactly two vertices of degree $1$ and all the other vertices have degree $2$. In other words $G$ is a path with at least one edge.
\end{corollary}

\begin{proof}
First, by Lemma~\ref{accum} the graph is finite. 
By Lemma~\ref{bpMr} every chord of $M_r$ in $\Sigma$ connects exactly two vertices in $V(G)$.
Thus, the inequality $m(S) \leq 2$ (Lemma~\ref{compcon}) implies $\deg (v) \leq 2$ for $v \in V_C(G)$; for $v \in V_A(G)$ the inequality $\deg (v) \leq 2$ holds by Lemma~\ref{27}.

Note that if $(S_1, S_2) \in E(G)$ then there is a path between $S_1$ and $S_2$ in $\Sigma$  not intersecting other sets $S \in V(G)$, $S \notin \{S_1$, $S_2\}$. It means that if $G$ has a cycle $C$ then so has $\Sigma$, contradicting Lemma~\ref{cycle}.
Moreover, the path between two points in $\Sigma$ belonging to two different vertices of $V(G)$ naturally induces a path in $G$ (in fact, if a path in $\Sigma$ connects two different vertices $S_1, S_2 \in V(G)$ wihtout touching other vertices, then $(S_1, S_2) \in E(G)$; therefore for a generic path in $\Sigma$ connecting two different vertices of $G$ it is enough to split it in a finite number of paths connecting different vertices in $G$ and not passing throw other vertices).
Therefore, connectedness of $\Sigma$ gives us that $G$ is connected. We conclude that $G$ is a path.

Now we have to show that $\# V(G) > 1$. Suppose the contrary, i.e.\ $V(G) = \{v\}$. If $v \in V_C(G)$, then $m(v) = 0$, so $v$ is a segment that is impossible.
Otherwise $v$ is an arc, but $q_v = M$, so $v = M_r$ contains a loop. We got again a contradiction with Lemma~\ref{cycle}.
\end{proof}

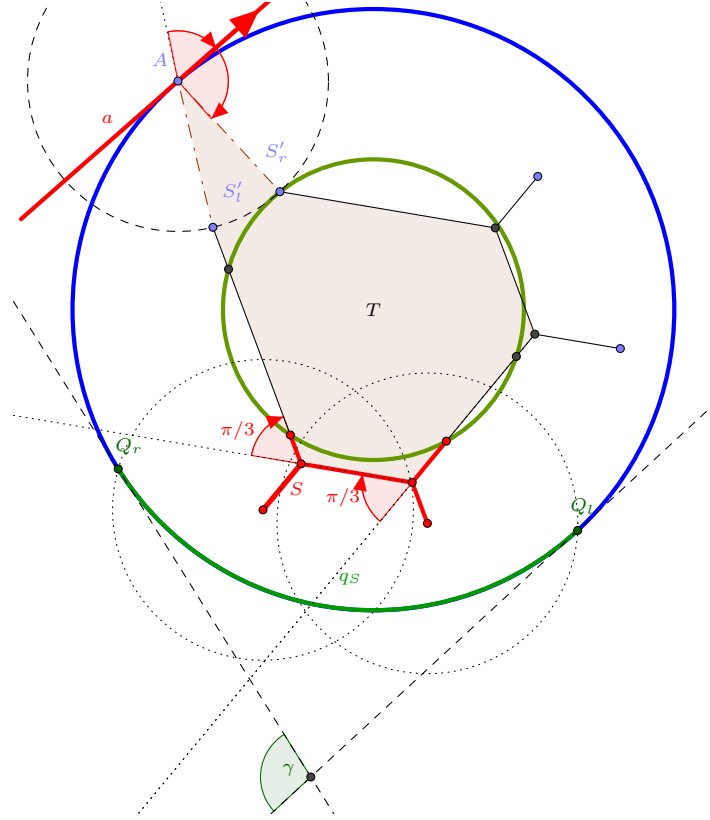
\begin{figure}
\begin{center}
	
\definecolor{zzttqq}{rgb}{0.6,0.2,0.}
\definecolor{ffffqq}{rgb}{1.,1.,0.}
\definecolor{qqzzqq}{rgb}{0.,0.6,0.}
\definecolor{qqwuqq}{rgb}{0.,0.39215686274509803,0.}
\definecolor{uuuuuu}{rgb}{0.26666666666666666,0.26666666666666666,0.26666666666666666}
\definecolor{ffqqqq}{rgb}{1.,0.,0.}
\definecolor{xdxdff}{rgb}{0.49019607843137253,0.49019607843137253,1.}
\definecolor{wwzzqq}{rgb}{0.4,0.6,0.}
\definecolor{qqqqff}{rgb}{0.,0.,1.}
\begin{tikzpicture}[line cap=round,line join=round,>=triangle 45,x=1.0cm,y=1.0cm]
\clip(-4.885214540250301,-6.699362511691325) rectangle (4.552546451354907,4.087549585829807);
\draw [shift={(-0.9614792438033364,-2.0505127981990108)},color=ffqqqq,fill=ffqqqq,fill opacity=0.1] (0,0) -- (110.43520263855103:0.6685747726686306) arc (110.43520263855103:170.435202638551:0.6685747726686306) -- cycle;
\fill[dash pattern=on 1pt off 3pt on 5pt off 4pt,color=zzttqq,fill=zzttqq,fill opacity=0.1] (-1.2431150100124204,1.570294293230513) -- (-2.5979910478844244,3.041454013315406) -- (-2.1338765110458207,1.0960497631156978) -- (-0.9614792438033364,-2.0505127981990108) -- (0.5153217823737517,-2.299361785598255) -- (2.146087477179613,-0.3256368878341969) -- (1.61939171892526,1.0879463747951914) -- cycle;
\draw [shift={(-0.8324307982627177,-6.217526912569525)},color=qqwuqq,fill=qqwuqq,fill opacity=0.1] (0,0) -- (121.99135999094486:0.6685747726686306) arc (121.99135999094486:222.75729619800924:0.6685747726686306) -- cycle;
\draw [shift={(0.5153217823737517,-2.299361785598255)},color=ffqqqq,fill=ffqqqq,fill opacity=0.1] (0,0) -- (170.435202638551:0.6685747726686306) arc (170.435202638551:230.43520263855098:0.6685747726686306) -- cycle;
\draw [line width=1.6pt,color=qqqqff] (0.,0.) circle (4.cm);
\draw [line width=1.6pt,color=wwzzqq] (0.,0.) circle (2.cm);
\draw [dash pattern=on 3pt off 3pt] (-2.5979910478844244,3.041454013315406) circle (2.cm);
\draw (-2.1338765110458207,1.0960497631156978)-- (-0.9614792438033364,-2.0505127981990108);
\draw [line width=1.6pt,color=ffqqqq] (0.5153217823737517,-2.299361785598255)-- (-0.9614792438033364,-2.0505127981990108);
\draw [line width=1.6pt,color=ffqqqq] (0.7177952270671499,-2.842774311824234)-- (0.5153217823737517,-2.299361785598255);
\draw [line width=2.pt,color=ffqqqq] (-1.4680431708373316,-2.663609969840924)-- (-0.9614792438033364,-2.0505127981990108);
\draw (0.5153217823737517,-2.299361785598255)-- (2.146087477179613,-0.32563688783419686);
\draw (2.146087477179613,-0.32563688783419686)-- (3.282526144404209,-0.5171329706289529);
\draw (2.146087477179613,-0.32563688783419686)-- (1.6193917189252605,1.0879463747951916);
\draw (1.6193917189252605,1.0879463747951916)-- (-1.2431150100124204,1.570294293230513);
\draw (1.6193917189252605,1.0879463747951916)-- (2.185436222663575,1.7730332231568222);
\draw [dotted,domain=-4.785214540250301:0.5153217823737517] plot(\x,{(--3.267462540795403--0.24884898739924433*\x)/-1.476801026177088});
\draw [dotted,domain=-4.785214540250301:0.9700516680913267] plot(\x,{(--1.3292021207872184-0.5503621397340159*\x)/-0.454729885717575});
\draw [shift={(-0.9614792438033364, -2.0505127981990108)}, <-,color=ffqqqq] (110.43520263855103:0.6685747726686306) arc (110.43520263855103:170.435202638551:0.6685747726686306);
\draw [dotted] (-1.4680431708373316,-2.663609969840924) circle (2.cm);
\draw [dotted] (0.7177952270671499,-2.842774311824234) circle (2.cm);
\draw [shift={(0.,0.)},line width=1.6pt,color=qqzzqq]  plot[domain=3.699947217622564:5.4586445782904045,variable=\t]({1.*4.*cos(\t r)+0.*4.*sin(\t r)},{0.*4.*cos(\t r)+1.*4.*sin(\t r)});
\draw [dash pattern=on 1pt off 3pt on 5pt off 4pt,color=zzttqq] (-1.2431150100124204,1.570294293230513)-- (-2.5979910478844244,3.041454013315406);
\draw [dotted, ,domain=-4:-2.5979910478844244] plot(\x,{-4.58*(\x + 2.5979910478844244) + 3.041454013315406});

\draw [shift={(-2.5979910478844244,3.041454013315406)}, <-,color=ffqqqq] (40.43520263855103:0.6685747726686306) arc (40.43520263855103:101.435202638551:0.6685747726686306);
\draw [shift={(-2.5979910478844244,3.041454013315406)}, <-,color=ffqqqq] (-48.435202638551:0.6685747726686306) arc (-48.435202638551:40.43520263855103:0.6685747726686306);
\draw [shift={(-2.5979910478844244,3.041454013315406)},color=ffqqqq,fill=ffqqqq,fill opacity=0.1] (0,0) -- (40.43520263855103:0.6685747726686306) arc (40.43520263855103:101.435202638551:0.6685747726686306) -- cycle;
\draw [shift={(-2.5979910478844244,3.041454013315406)},color=ffqqqq,fill=ffqqqq,fill opacity=0.1] (0,0) --  (-48.435202638551:0.6685747726686306) arc (-48.435202638551:40.43520263855103:0.6685747726686306) -- cycle;

\draw [dash pattern=on 1pt off 3pt on 5pt off 4pt,color=zzttqq] (-2.5979910478844244,3.041454013315406)-- (-2.1338765110458207,1.0960497631156978);
\draw [dash pattern=on 3pt off 3pt,domain=-4.785214540250301:4.552546451354907] plot(\x,{(--16.--3.392511986160677*\x)/-2.119165501737922});
\draw [dash pattern=on 3pt off 3pt,domain=-4.785214540250301:4.552546451354907] plot(\x,{(-16.--2.7155770024955532*\x)/2.9369442527084617});
\draw [dotted,domain=-4.785214540250301:2.146087477179613] plot(\x,{(--4.766823752132409-1.9737248977640582*\x)/-1.630765694805861});
\draw [shift={(0.5153217823737517, -2.299361785598255)}, <-,color=ffqqqq] (170.435202638551:0.6685747726686306) arc (170.435202638551:230.43520263855098:0.6685747726686306);
\draw [line width=1.6pt,color=ffqqqq] (0.5153217823737517,-2.299361785598255)-- (0.9700516680913267,-1.7489996458642392);
\draw [line width=1.6pt,color=ffqqqq] (-0.9614792438033364,-2.0505127981990108)-- (-1.1041569436210037,-1.6675843138664757);
\draw [line width=1.6pt, ->, color=ffqqqq,domain=-4.685214540250301:-1.5153217823737517] plot(\x,{0.88*(\x + 2.5979910478844244) + 3.041454013315406});
\draw [line width=1.6pt, color=ffqqqq,domain=-2.5153217823737517:2] plot(\x,{0.88*(\x + 2.5979910478844244) + 3.041454013315406});

\begin{scriptsize}

\draw [fill=xdxdff] (-2.5979910478844244,3.041454013315406) circle (1.5pt);
\draw[color=xdxdff] (-2.834346138932578,3.330403941939283) node {$A$};
\draw[color=ffqqqq] (-3.534346138932578,2.530403941939283) node {$a$};
\draw [fill=xdxdff] (-2.1338765110458207,1.0960497631156978) circle (1.5pt);
\draw[color=xdxdff] (-1.8657713662639475,1.5692512754446393) node {$S_l'$};
\draw [fill=ffqqqq] (-0.9614792438033364,-2.0505127981990108) circle (1.5pt);
\draw [fill=ffqqqq] (0.5153217823737517,-2.299361785598255) circle (1.5pt);
\draw [fill=ffqqqq] (0.7177952270671499,-2.842774311824234) circle (1.5pt);
\draw [fill=xdxdff] (3.282526144404209,-0.5171329706289529) circle (1.5pt);
\draw [fill=xdxdff] (2.185436222663575,1.7730332231568222) circle (1.5pt);
\draw [fill=xdxdff] (-1.2431150100124204,1.570294293230513) circle (1.5pt);
\draw[color=xdxdff] (-1.2863398966178008,2.104111093579542) node {$S_r'$};
\draw [fill=ffqqqq] (-1.4680431708373316,-2.663609969840924) circle (1.5pt);
\draw [fill=uuuuuu] (-1.9261384709982619,0.5385077441787424) circle (1.5pt);
\draw [fill=ffqqqq] (-1.1041569436210037,-1.6675843138664757) circle (1.5pt);
\draw [fill=ffqqqq] (0.9700516680913267,-1.7489996458642392) circle (1.5pt);
\draw [fill=uuuuuu] (1.900556349843245,-0.6228045930069245) circle (1.5pt);
\draw[color=ffqqqq] (-1.8089138889970844,-1.6036225878985044) node {$\pi/3$};
\draw [fill=qqwuqq] (2.7155770024955532,-2.9369442527084617) circle (1.5pt);
\draw[color=qqwuqq] (2.7811095196833305,-2.6204839666121007) node {$Q_l$};
\draw [fill=qqwuqq] (-3.392511986160677,-2.119165501737922) circle (1.5pt);
\draw[color=qqwuqq] (-3.2474925631124507,-1.8181942394097461) node {$Q_r$};
\draw[color=qqzzqq] (-0.30576356337047583,-3.5787744741038017) node {$q_S$};
\draw [fill=uuuuuu] (2.146087477179613,-0.3256368878341969) circle (1.5pt);
\draw [fill=uuuuuu] (1.61939171892526,1.0879463747951914) circle (1.5pt);
\draw [fill=uuuuuu] (-0.8324307982627177,-6.217526912569525) circle (1.5pt);
\draw[color=qqwuqq] (-1.12747695732550768,-6.120503473845768) node {$\gamma$};
\draw[color=ffqqqq] (-1.0248127698768492,-2.37533988330027) node {$S$};
\draw[color=ffqqqq] (-0.4048127698768492,-2.49533988330027) node {$\pi/3$};
\draw[color=black] (0,0) node {$T$};

\end{scriptsize}
\end{tikzpicture}
\caption{Figure to the construction of $T$.}
\label{defT}
\end{center}
\end{figure}

Thus under conditions of Theorem~\ref{theo1} there are two connected components of $\Sigma \setminus N_r$ with one entering point; 
these components correspond to the leaves of our graph. 
We call them \textit{ending components} and denote by $S_l$ and $S_r$
(calling them \emph{left} and \emph{right} respectively); the 
other components will be called \textit{middle components}.

By Lemma~\ref{dugi} every point of $M$ is covered by at most two sets from $V(G)$.
By Corollary~\ref{co_graph1} graph $G$ is a path, so if $S_1$, $S_2$ are connected by an edge in $G$, then $q_{S_1} \cap q_{S_2} \neq \emptyset$. 
Moreover, the same reasoning gives $q_{S_l} \cap q_{S_r} \neq \emptyset$, because otherwise there would be some part of $M$ not covered by $\Sigma$.

\begin{lemma}
The arcs $q_{S_l}$ and $q_{S_r}$ have disjoint interiors.
\label{disj}
\end{lemma}

Denote by $A$ an arbitrary point of the intersection of $q_{S_l}$ and $q_{S_r}$ (see Fig.~\ref{defT});
by Lemma~\ref{disj} there are at most 2 such points.
Consider the set $\hat \Sigma := \Sigma \cup [AS_l'] \cup [AS_r']$, where 
$[AS_l']$ and $[AS_r']$ are segments of length $r$ connecting $A$ with $S_l$ and $S_r$ respectively. 
In view of Lemma~\ref{cycle} and the fact that $B_r(A) \cap \Sigma = \emptyset$, the set $\hat \Sigma$ bounds the unique region which we further denote by $T$ (see Fig.\ref{defT}). 

Previous Lemmas give us the following corollary.
\begin{corollary}
The boundary of $T$ is a closed curve consisting of a finite number of arcs of $M_r$ and a finite number of line segments.
\label{col27}
\end{corollary}

Consider the behavior of the tangent line to the boundary of $T$. 
Corollary~\ref{col27} and Lemma~\ref{debil} imply that all points where tangent direction is discontinuous (i.e.\ points where the tangent line to $\partial T$ does not exist) except $A$ belong to connected components of $\Sigma \setminus N_r$.

\begin{definition} 
Let $\gamma$ be a $C^1$-smooth injective planar curve. We say that the \textit{turning} of $\gamma$ is the following object:
$$\turn (\gamma) := \int_{0}^{1} d\arg(\gamma'(t)),$$ where $\gamma: [0, 1] \to \mathbb R^2$ is some injective parameterization of $\gamma$ with $\gamma'(t) \neq 0$ and $\arg$ 
is a continuous branch of the multifunction $Arg$. 

Let $\gamma : [0,1] \rightarrow \mathbb{R}^2$ be a piecewise $C^1$-smooth injective planar curve, with a finite number of discontinuity points $\{t_i\}_{i=1}^{N}$ of $\gamma'$, $t_i < t_j$ if $i<j$ (it means that $\gamma$ is $C^1$-smooth on every $[t_it_{i+1}]$). 
We define 
$$\turn (\gamma) := \sum_{i=1}^N \turn(\gamma ([t_i, t_{i+1}])) + \sum_{i=1}^N \angle ([\gamma(t_i - 0) \gamma(t_i)), [\gamma(t_i) \gamma(t_i + 0))).$$

Let $\gamma : [0,1] \rightarrow \mathbb{R}^2$ be a simple closed (i.e.\ $\gamma(0) = \gamma(1)$) piecewise $C^1$-smooth planar curve, with a finite number of discontinuity points of $\gamma'$ $\{t_i\}_{i=1}^{N}$, $t_i < t_j$ if $i<j$. 
We define
$$\turn (\gamma) := \sum_{i=1}^N \turn(\gamma ([t_i, t_{i+1}])) + \sum_{i=1}^N \angle ([\gamma(t_i - 0) \gamma(t_i)), [\gamma(t_i) \gamma(t_i + 0))) + \angle ([\gamma(1-0) \gamma(1)), [\gamma(0) \gamma(0 + 0))).$$
\end{definition}

In our setting, the turning of an open curve $\gamma$ will almost always coincide with the directed angle between the tangent lines to the ends of $\gamma$. 
Note that for a self-avoiding closed piecewise $C^1$-smooth planar curve $\gamma$ we always have $\turn (\gamma) = 2\pi$ (for a parameterization with respect to the clockwise orientation).

Now we define the same quantity for the closure of a connected component of $\Sigma \setminus N_r$.

\begin{definition}
Under conditions of Theorem~\ref{theo1} let $S$ be the closure of a connected component of $\Sigma \setminus N_r$. Then $\turn (S)$ stands for the turning number of the $S \cap \partial T$ parameterized in the clockwise order. In particular, if $S=S_l$ then $\turn (S)$ stands for the turning number of the curve $S\cap \partial T$ parameterized so that it starts at the entering point and ends at point $S_l'$,
and if $S=S_r$ then $\turn (S)$ stands for the turning number of $S\cap\partial T$ parameterized so that it starts at point $S_r'$ and ends at the entering point.
\end{definition}

Now we are ready to state the central Lemma. Figure~\ref{defT} should simplify the reading of its statement.

    \begin {lemma}   \label{central}
    Under conditions of Theorem~\ref{theo1} let $\Sigma \in OPT_{\infty}^{*}(M, r)$ be a minimizer, 
    $S \in V(G)$ be the closure of a connected component of $\Sigma \setminus N_r$ or an arc of $M_r$. 
    Then the following assertions hold.
    \begin {itemize}
    \item If $S$ is a middle component or an arc of $M_r$ then $\turn (q_S)\leq \turn (S)$. 
    The equality holds if and only if $S$ is an arc of $M_r$.
    \item If $S$ is an ending component then for the left and the right components we have
    $$\turn (q_{S_l}) \leq \turn (S_l) + \angle([C_lS_l'), [S_l'A)) + \angle ([S_l'A), a),$$
    $$\turn (q_{S_r}) \leq \angle (a, [AS_r')) + \angle([AS_r'), [S_r'C_r)) + \turn (S_r),$$
    where $a$ stands for the tangent ray to $M$ at the point $A$ directed from the left to the right (see Fig.~\ref{defT}, angles $\angle ([S_l'A), a)$, $\angle (a, [AS_r'))$ are marked red) and $C_i$ is the branching point of $S_i$ if $S_i$ is a tripod and the entering point of $S_i$ in other cases, where $i \in \{l, r\}$ (the definition is correct by Lemma~\ref{compcon}). 
    The equality holds if and only if $S$ is a segment of the tangent line to $M_r$.
    \end {itemize}
    \end {lemma}

\begin{remark}
If in Lemma~\ref{central} we assume that $\Sigma$ has no Steiner points in $N_r$ then it is enough to request the inequality $r < R/2.9$ (see proof of Lemma~\ref{central}, Case~1a).
\end{remark}

Now the proof of Theorem~\ref{theo1} is just few lines.

\begin{proof}[Proof of Theorem~\ref{theo1}]
By Lemma~\ref{global_lm}(iii) $2a_M (r) + r < 5r$ for general $M$, and $2a_M (r) + r < 4.98r$ when $M$ is the circumference.
Note that 
\[ 2\pi = \turn (\partial T) = \sum_{S \in V(G)} \turn (S) + \angle([C_lS_l'), [S_l'A)) + \angle ([S_l'A), a) + \angle([AS_r'), [S_r'C_r)) + \angle (a, [AS_r'))\] 
by Lemma~\ref{27} and Lemma~\ref{debil}, and also $\turn (M)=2\pi$.
Hence by Lemma~\ref{central} 
\begin{align*}
    2\pi &= \sum_{S \in V(G)} \turn (S) + \angle([C_lS_l'), [S_l'A)) + \angle ([S_l'A), a) + \angle([AS_r'), [S_r'C_r)) + \angle (a, [AS_r')) \\
    &  \geq \sum_{S \in V(G)} \turn (q_S) \geq \turn (M) = 2\pi.
\end{align*}
Thus all the inequalities in Lemma~\ref{central} are equalities. Summing up, every global minimizer $\Sigma \in OPT_{\infty}^{*}(M, r)$ consists of arcs of $M_r$ and 
segments of tangent lines to $M_r$, i.e.\ components of the combinatorial type (a) on Fig.~\ref{lmn23}, tangent to $M_r$.
Every vertex, corresponding to a component of the combinatorial type (a) on Fig.~\ref{lmn23} has degree 1 in $G$. Thus $\Sigma$ has the unique arc of $M_r$, and because of the absence of loops it cannot coincide with $M_r$. 
By Lemma~\ref{27} every maximal arc $\breve{[BC]} \in V_A(G)$ is connected in the graph $G$ with two vertices, corresponding to connected components of $\Sigma \setminus N_r$.
Hence any minimizer is a horseshoe. 
\end{proof}

	\section {Proofs} \label{pr}

Recall that $\Sigma$ is an arbitrary minimizer for some convex closed curve $M$ and $N = \text{conv}\,(M)$.    
Clearly $\Sigma\subset N$ ($N$ is a convex set, so one can project on $N$ the part of $\Sigma$ belonging to $\R^2 \setminus N$ on $N$ and length of $\Sigma$ will strictly decrease). 

The following well-known fact will be used during the proof.

\begin{lemma}
Let $M$ be a convex closed curve with minimal radius of curvature $R$ and $B_R(O)$ be a ball of radius $R$ centered at point $O \in N$. 
If $\partial B_R(O)$ touches $M$ (tangentially to $M$), then $B_R(O) \subset N$.
\label{circ_curv}
\end{lemma}

Further on we assume by default $M$ and $\Sigma$ are as in Theorem~\ref{theo1}. Sometimes we will request weaker conditions.

The following assertion is valid.
   
\begin{lemma}
Let $M$ be a convex closed curve with minimal radius of curvature $R > r$ and $\Sigma$ be an arbitrary minimizer for $M$.
Then the set $E_\Sigma$ of non-discrete energetic points of $\Sigma$ is a subset of $M_r$.
\label{en_sub}
\end{lemma}

\begin{proof}
Suppose the contrary. Then there are such a point 
$x \in E_\Sigma \setminus M_r$ that $\dist(x, M) < r - \varepsilon$ for some positive $\varepsilon$ and a sequence $\{x_k\}$ of energetic points 
from $B_{\varepsilon/2}(x)$ converging to $x$.
Let us choose such a sequence of positive numbers $\{\varepsilon_k\}$ that $B_{\varepsilon_i}(x_i) \cap B_{\varepsilon_j}(x_j)=\emptyset$ for $i \neq j$.

    Because of convexity of $N$ and the fact that minimal radius of curvature of $M$ exceeds $r$, one has that each 
    $\gamma_k := \overline {B_{r+\varepsilon_k}(x_k)} \cap M$
    is connected, thus we can say that each $\overline{B_{\varepsilon_k}(x_k)}$ covers the arc $\gamma_k$. 
 	All $\gamma_j$ have a common point: in fact, for $z \in M$ such that $\dist (x,z) = \dist (x, M)$ one has
 \[
 \dist (x_j, z) \leq \dist (x,z) + \dist (x_j, x) <  r - \varepsilon + \varepsilon/2 = r - \varepsilon/2,
 \]
thus $z \in \gamma_j$ for all $j$.
Therefore $\gamma_i \subset (\gamma_j \cup \gamma_l)$ for some distinct $i,j,l$. 
So one of the points $x_i, x_j, x_l$ is not energetic because 
$F_M(\Sigma)=F_M(\Sigma \setminus B_{\varepsilon_i} (x_i))$ which is the desired contradiction.
\end{proof}

\begin{proof}[Proof of Lemma~\ref{global_lm}]

{\sc Proof of~(i)}: No change in the set $\Int (\Sigma \cap N_r)$ influences the value of $F_M(\Sigma)$, 
so  if we take the closure $S$ of any connected component of $\Sigma \cap \Int (N_r)$ and substitute it by a Steiner tree connecting $S\cap M_r$ (which must be nonempty if $\Sigma \cap \Int(N_r) \neq \emptyset$ because of connectedness of $\Sigma$ and the requirement $F_M(\Sigma)\leq r$ which gives $\Sigma\setminus \Int (N_r) \neq \emptyset$),
then the length of the resulting set should remain the same by optimality of $\Sigma$, and thus $S$ is itself a Steiner tree connecting $S\cap M_r$ as claimed. 

{\sc Proof of~(ii)}: Recall that $\Sigma=E_\Sigma \sqcup X_\Sigma \sqcup S_\Sigma$, where $X_\Sigma$ is a discrete set of points, $S_\Sigma$ consists of Steiner trees
(hence of line segments) and $E_\Sigma \subset M_r$ by Lemma~\ref{en_sub}.
	
{\sc Proof of~(iii)}:  
Remove an arbitrary open line segment $\Delta$ from the set $\Sigma \cap \Int (N_r)$. The value of $F_M$ 
does not change, i.e.\ $F_M(\Sigma \setminus \Delta) = F_M(\Sigma)$, and by Lemma~\ref{cycle} $\Sigma \setminus \Delta$ splits into two connected
components $\Sigma_1$ and $\Sigma_2$, so that $\Sigma \setminus \Delta = \Sigma_1 \sqcup \Sigma_2$ ($\Sigma$ is closed, so $\Sigma_1$, $\Sigma_2$ are closed too). 
Obviously $M \subset \overline{B_r(\Sigma_1)} \cup \overline{B_r(\Sigma_2)}$.
Then by connectedness of $M$ there is such a point $A \in M$ that 
$A \in \overline{B_r(\Sigma_1)} \cap \overline{B_r(\Sigma_2)}$, 
but then there are points $B \in \overline{\Sigma_1}$ 
and $C \in \overline{\Sigma_2}$ such that $|AB|\leq r$, 
$|AC|\leq r$. Hence the distance between $\Sigma_1$ and $\Sigma_2$ does not exceed $|BC|\leq 2r$ but the length of the deleted 
segment $\Delta$ does not exceed the distance between the $\Sigma_1$ and $\Sigma_2$ in view of optimality of $\Sigma$
(otherwise one could connect $\Sigma_1$ with $\Sigma_2$ with a shorter segment).
We let then $a_M(r)$ be the supremum of $|BC|$ over all the possible choices of $\Delta$, so that we have proven $a_M(r)\leq 2r$.

In the case $M = \partial B_R(O)$ the length of the segment $[BC]$ reaches its maximal value when $[BC]$ is a chord and $|AB| = |AC| = r$. 
Then we can calculate the maximal value of length of $[BC]$ in this case:
			\[
				\sin \frac{\angle AOC}{2} = \frac{|AC|}{2|OC|} = \frac{r}{2R},
			\]
			    so that
			\[
				|BC|=2 |OC| \sin \angle AOC = 4 |OC| \sin \frac{\angle AOC}{2} \cos \frac{\angle AOC}{2}  =  2r\sqrt{1- \frac{r^2}{4R^2}}.
			\]
\end{proof}

\begin{proof}[Proof of Lemma~\ref{hex}]
Assume the contrary i.e.\ that $\Sigma$ has a Steiner point $X \in \Int (N_r) \cup (S_\Sigma \cap M_r)$.
In view of Lemma~\ref{circ_curv} there is a point $O \in N$ such that $X \in B_R(O)$ and $B_R(O) \subset \Int(N)$ 
(hence $B_{R-r}(O) \subset \Int(N_r)$, and in particular, $O \in \Int (N_r)$). 
Recall that as defined in Lemma~\ref{global_lm} $\Sigma_r$ is the union of the closures of all connected components of $\Sigma \cap \Int (N_r)$.
Now denote by $X_0$ one of the Steiner points of $\Sigma_r \cup (S_\Sigma \cap M_r)$ nearest to $O$, and let $t \defeq |OX_0|$.
We claim that $X_0\in \Int(N_r)$. 
In fact, otherwise $X_0\in M_r$ and hence 
\[t = \dist(O,M_r) = \dist(O,M) - r \geq R - r > 3.98r,\] 
but $X_0$ is a Steiner point, hence, in view of the smoothness and convexity of $M_r$ there are two line segments $[X_0 Z_i]\subset \Sigma$, $i=1,2$ at angle $2\pi/3$ with respect to each other, intersecting $\Int(N_r)$. 
Suppose without loss of generality that $\angle OX_0 Z_1\leq \pi/3$. Then 
$Z_1\in B_t(O)\subset \Int(N_r)$, since otherwise there is an $Y\in [X_0Z_1]\cap \partial B_t(O)\subset\Sigma\cap \partial B_t(O)$ 
such that the line segment $[X_0Y]\subset \Sigma$ is a chord of $\partial B_t(O)$, which provides the estimate 
\[ 
|X_0Y|= 2t\cos \angle OX_0 Z_1\geq t > 3.98r 
\] 
contrary to Lemma~2.7(iii), this contradiction proving the claim.

	
Let $\Sigma'$ stand for the closure of the connected component of $\Sigma \cap \Int(N_r)$ containing $X_0$. 
By the structure of a Steiner tree sivce $X_0$ belongs to $\Int(N_r)$ then there are three maximal line segments of $\Sigma'$ starting from $X_0$.
Consider such a pair of them $[X_0X_{-1}], [X_0X_{1}]$ that the point $O$ belongs to the angle $\angle X_{-1}X_0X_{1}$ (not excluding the case it belongs to one of the sides of this angle). 
Recall that $\angle X_{-1}X_0X_{1} = 2\pi/3$.
Also note that points $X_{-1}$, $X_1$ lie outside of $B_t(O)$. 
Hence either $[X_0X_1]$ or $[X_0X_{-1}]$ intersects $B_t(O)$.
We assume without loss of generality that it is $[X_0X_1]$. 
Denote the intersection of the segment $[X_0X_1]$ and the circumference $\partial B_t (O)$ by $C$.

We claim that $t \leq a_M(r)$. Supposing the contrary, since $|X_0C| \leq a_M(r)$ and $|OX_0|= |OC| = t > a_M(r) \geq |X_0C|$, we have $\angle{OX_0C} > \pi/3$, hence the segment $[X_0X_{-1}]$ also intersects $B_t(O)$. 
Denote the intersection of the segment $[X_0X_{-1}]$ with $\partial B_t(O)$ by $D$ and note that also $\angle{OX_0D} > \pi/3$, and hence $\angle CX_0D > 2\pi/3$ which contradicts the local optimality of $\Sigma$, showing the claim. 

Note that $X_1, X_{-1}$ belong to $\Int (N_r)$ because $R - r > 2a_M(r) \geq t + a_M(r)$, and hence $X_1, X_{-1}$ are Steiner points.
Also by Lemma~\ref{global_lm} the lengths $[X_0X_{-1}]$ and $[X_0X_{1}]$ do not exceed $a_M (r)$. 
Consider a regular hexagon $P$ with sidelength $a_M(r)$ such that $X_0$ is a 
vertex of $P$ and the segments $[X_0X_1]$, $[X_0X_{-1}]$ belong to two sides of $P$.
The following assertions hold.

\begin{itemize}
\item $\diam P=2a_M (r)$.
\item The line segment $[OX_0]$ splits the angle $\angle X_{-1}X_0X_1 = 2 \pi/3$ 
in two angles, at least one of them is acute. 
Denote the latter angle by $\angle OX_0B$, where $B$ is the corresponding vertex of $P$ 
(so that $|X_0B| = a_M(r)$). Then the angle $\angle OBX_0$ is also acute 
because $|OX_0|=t \leq a_M(r) = |X_0B|$. Therefore the perpendicular from $O$ 
to the line $(X_0B)$ intersects the latter inside $[X_0B]$, so that 
$O$ is inside the square built on $[X_0B]$. But this square is a subset of $P$ hence $O\in P$.
\item The above assertions imply that $P \subset \overline{B_{2a_M(r)}(O)}$, and hence $P \subset \Int (N_r)$.  
\end{itemize}

Now let us pick such vertices $X_{-2}$ and $X_2$ that 
$[X_1X_2], [X_{-1}X_{-2}] \subset \Sigma_r$ and $O$ belongs to both angles $\angle X_0X_1X_2$ and $\angle X_0X_{-1}X_{-2}$.
Clearly $X_2, X_{-2} \in P \subset \Int (N_r)$ so they again are Steiner points. 
Let us define the points $X_3, X_{-3}$ in the same way: 
$[X_2X_3], [X_{-2}X_{-3}] \in \Sigma_r$ and $O$ belongs to the angles $\angle X_1X_2X_3$ and $\angle X_{-1}X_{-2}X_{-3}$.
Points $X_3, X_{-3}$ also belong to $P$, hence to $\Int (N_r)$, hence they also are Steiner points. The six constructed 
line segments belong to $\Int (N_r)$, so there is no endpoint there. 
Continuing inductively this construction, we arrive at two paths in $P \subset \Int (N_r)$: 
one path (starting from $X_0, X_1,X_2, X_3 \dots$) turns left every time and 
the other one  (starting from $X_0, X_{-1}, X_{-2}, X_{-3} \dots$) 
turns right every time.
Thus $\Sigma\cap P\subset \Sigma \cap \Int(N_r)$ contains a cycle or an endpoint of $\Sigma$ in $\Int (N_r)$, 
but both cases are impossible for a Steiner tree by Lemma~\ref{cycle} and Lemma~\ref{global_lm}.
\end{proof}

\begin{proof}[Proof of Lemma~\ref{compcon}]

Let $S \in V_C(G)$ be the closure of a connected component of $\Sigma \setminus N_r$.
First we prove that $n(S) \leq 2$. 
By property (a) of the set of energetic points for every energetic point $x \in S$ of $\Sigma$ 
there is such a point $Q_x \in M$ that $\dist (x, Q_x) = r$ and $B_r(Q_x) \cap \Sigma = \emptyset$. 
Then $Q_x$ can be only an end of the arc $q_S$, otherwise $S = S \setminus B_r(Q_x)$ is not connected.
If an end of $q_S$ corresponds to two different energetic points $W_1$, $W_2$ of $S$ then $q_{W_1} \subset q_{W_2}$ or $q_{W_2} \subset q_{W_1}$ which is impossible, and hence $n(S)\leq 2$ as claimed.

Now let us prove $m(S)\leq 2$. 
Assume the contrary i.e.\ the existence of at least three different entering points in $S$. Let us denote them  $I_1$, $I_2$ and $I_3$ such that $Q_{I_2} \in \breve{[Q_{I_1}Q_{I_3}]} \subset q_S$. Note that $I_2$ cannot be energetic, because $Q_{I_2}$ is not an end of $q_S$. So $I_2$ has such a neighbourhood $U$ that $U \cap \Sigma$ is a segment or a regular tripod; by Lemma~\ref{hex} it is a segment. 

We claim that $\Sigma$ contains a chord $[I_2J]$ of $M_r$.
It is true if $\Sigma$ is not tangent to $M_r$ at $I_2$.
Now, let $\Sigma$ be tangent to $M_r$ at $I_2$, so $I_2$ belongs to two closures of different connected components of $\Sigma \setminus N_r$; one of them is $S$; denote the second one by $S'$.
Let $P_1$ be the region bounded by the arc $\breve{[I_1I_2]}$ of $M_r$ (choosing in such a way that $P_1$ does not contain $N_r$) and the unique path between $I_1$ and $I_2$ in $S$.
Define $P_3$ analogously (with $I_3$ in place of $I_1$). Obviously, $S' \subset P_1$ or $S' \subset P_3$. Hence $q(S') \subset q(S)$ and replacing $S'$ in $\Sigma$ by a Steiner tree for $S' \cap M_r$ we get a connected competitor to $\Sigma$ still covering $M$. Also, any Steiner tree for $S' \cap M_r$ belongs to $N_r$ by the convexity of $M_r$, so this replacement decreases the length, which is impossible.
Hence, we get the claim, i.e.\ there is a chord $[I_2J] \subset \Sigma$ of $M_r$.

Then $|I_2J| \leq |I_1J|$ (otherwise we can replace $[I_2J]$ by $[I_1J]$ in $\Sigma$ producing the competitor of strictly lower length), and analogously $|I_2J| \leq |I_3J|$.
Note that $J \notin S$ because by Lemma~\ref{cycle} $\Sigma$ has no loops.
One can see that points $I_1, I_2, I_3, J$ belong to $M_r$ in the natural (clockwise) order
otherwise the arc $q_{S_J}$ is a subset of $q_S$, where $S_J$ is the closure of the connected component of $\Sigma \setminus N_r$ containing $J$, which is impossible.

Hence $|JI_2|$ is at least the diameter $d$ of the maximal ball inscribed in $N_r$ and touching $M_r$ at point $I_2$, i.e.\ the double \textit{inradius} of $M_r$. 
Since $d \geq 2(R-r)$, we have $|JI_2| \geq 2(R-r) > 2r$ contradicting Lemma~\ref{global_lm}(iii), showing the claim $m(S) \leq 2$.

Finally, note that $S$ should be locally minimal in a neighbourhood of any point $x \in S$ except energetic and entering points of $S$. 
We have proved that $n(S) \leq 2$, a non energetic point $x \in S$ has a neighbourhood $U_x$ such that $\Sigma \cap U_x$ is either a segment or a regular tripod.
If $x \in S$ is a non energetic endpoint of $S$ then $\Sigma \cap U_x \neq S \cap U_x$, so $x$ is an entering point.
So by definition of a locally minimal network, $S$ is a locally minimal network for its entering and energetic points.
\end{proof}

\begin{remark}
\label{corr}
During the proof of Lemma~\ref{compcon} (claim $n(S) \leq 2$) we show that if $x \in G_\Sigma \cap S$ then $Q_x$ can be only an end of the arc $q_S$. 
So in the case $n(S) = 2$ there is the unique one-to-one correspondence between energetic points of $S$ and endpoints of $q_S$.
\end{remark}

\begin{proof}[Proof of Lemma~\ref{dugi}]
The fact that $S$ has an energetic point immediately implies that $q_S$ does not belong to the union of $q_{S'}$ over $S' \in V(G) \setminus \{S\}$.
Suppose the contrary, i.e.\ that $S$ has no energetic point.

If $S$ is the closure of a  connected component of $\Sigma \setminus N_r$, then by Lemma~\ref{compcon} $S$ is a locally minimal network for its entering points, but $m(S) \leq 2$, hence $S$ is a segment with endpoints on $M_r$, which is impossible for a connected component of $\Sigma \setminus N_r$.

If $S$ is a non degenerate arc $\breve{[BC]}$, then $\breve{[BC]} \subset S_\Sigma$, which is impossible by the definition of $V_A(G)$.
\end{proof}

\begin{proof}[Proof of Lemma~\ref{accum}]
Suppose the contrary. 
Consider an arbitrary $\varepsilon > 0$ (which later will be chosen sufficiently small). 
First, note that Lemma~\ref{dugi} implies that every point of $M$ belongs to at most two different arcs $q_S$, where $S \in V(G)$ (otherwise, there are three arcs of $M$ containing a point $x \in M$, so one of them is contained in the union if others, which is impossible by Lemma~\ref{dugi}).   
Thus the sum of $\H(q_S)$ over $V(G)$ is at most $2\H(M)$, and therefore there is only a finite number of connected components and arcs with $\H(q_S) \geq \varepsilon$.
Denote by $V_\varepsilon(G)$ the infinite set of such $S \in V(G)$ that $\H(q_S) < \varepsilon$.

Obviously, if $V(G)$ is an infinite set, then $V_C(G)$ is an infinite set.
Let us show that there are infinitely many chords of $M_r$ in $\Sigma$ intersecting $\Int (N_r)$ (if $N$, and hence $N_r$, is strictly convex then in fact every chord of $M_r$ intersects $\Int (N_r)$).
Suppose the contrary. Then 
$\Sigma \setminus \Int (N_r)$ has a finite number of connected components; but $V_C(G)$ is infinite, hence there are components containing infinitely many elements of $V_C(G)$; 
let $K$ be one of these components containing at least five different elements of $V_C(G)$. Obviously, $q_K := \overline{B_r(K)} \cap \Sigma$ is connected. By Lemma~\ref{dugi} $K \setminus M_r$ contains $5$ energetic points, such that they belong to different elements of $V_C(G)$.
Call them $W_1$, $W_2$, $W_3$, $W_4$, $W_5$ such that $Q_{W_1}$, $Q_{W_2}$, $Q_{W_3}$, $Q_{W_4}$, $Q_{W_5} \in q_K$ belong to $M_r$ in the natural (clockwise) order. Then $B_r(Q_{W_i}) \cap \Sigma = \emptyset$, $i = 1, \ldots, 5$ and therefore $K$ should contain the points $I_2$, $I_3$, $I_4 \in M_r$ such that 
$$\dist (Q_{W_2}, I_2) = \dist (Q_{W_3}, I_3) = \dist (Q_{W_4}, I_4) = r$$
(because $K \setminus I_j$ must be disconnected, $j = 2, 3, 4$).
Consider the path between $I_2$ and $I_4$ in $K$. It should coincide with $\breve{[I_2I_4]} \subset M_r$, otherwise we reduce the length of $\Sigma$, projecting the path on $M_r$. So $W_3$ should belong to $M_r$ which is impossible by the choice of $W_i$, $i = 1, \ldots, 5$ and gives the desired contradiction. 
Thus the set $Ch$ of chords of $M_r$ in $\Sigma$ intersecting $\Int (N_r)$ is infinite.

There is at most a finite number of chords of length at least $\varepsilon$ because $\H(\Sigma)$ is finite.
Let us exclude from the infinite set $Ch$ a finite set of chords of length at least $\varepsilon$ and a finite set of chords adjacent to a component not in $V_\varepsilon(G)$;
denote the resulting set by $Ch'$: chords in $Ch'$ are adjacent only to the elements of $V_{\varepsilon}(G)$ and have length strictly less than $\varepsilon$.
Let us show that any of the chords in $Ch'$ connects components without Steiner points. 
Suppose the contrary. The following three cases have to be considered:

\begin{itemize}
\item[(i)] A chord in $Ch'$ is adjacent to a connected component $S \in V_\varepsilon(G)$ with $m(S) = 2$ containing a Steiner point.
Then the angle between the entering segments of the component is at most $2\pi/3$ (in fact, it must be between $\pi/3$ and $2\pi/3$).
Recall that $\H(q_S) < \varepsilon$, hence by the triangle inequality $S$ is a subset of an $\varepsilon$-neighbourhood of $M_r$
(otherwise $\dist (x,y) \leq r-\varepsilon$ for some $x \in S$, $y \in M$, so $B_\varepsilon (y) \cap M \subset q_S$ which contradicts $\H(q_S) < \varepsilon$).
So, when $\varepsilon$ is sufficiently small, recalling smoothness of $M_r$ one has that one of the entering segments has angle with $M_r$ at least $\pi/12$. 
It implies that the entering point $I$ of this segment is not energetic, so by Lemma~\ref{hex} its neighbourhood is a segment and it is an end of a chord $[IJ] \subset \Sigma$ of $M_r$.
So by the constraint on the radius of curvature of $M$ chord $[IJ]$ has length more than $\varepsilon$, which gives a contradiction with the assumption that our chord is in $Ch'$.

\item[(ii)] A chord in $Ch'$ is adjacent to a connected component $S \in V_\varepsilon(G)$ with $m(S) = 1$ containing a Steiner point.
Then it has the combinatorial type (b) on Fig.~\ref{lmn23}. Let us consider the triangle $\Delta QCI$, where $Q$ is an end of $q_S$, $C$ is the branching point of $S$, $I$ is the entering point of $S$. Since $\angle QCI = 2\pi/3$, we have $\angle QIC \leq \pi/3$, so the angle between the entering segment $[CI]$ and $M_r$ is at least $\pi/6$. Then again the chord $[IJ]$ has length more than $\varepsilon$, that contradicts the choice of the chord.

\item[(iii)] Finally, a chord in $Ch'$ is adjacent to an arc $S \in V_\varepsilon(G)$ containing a Steiner point $x$.
Then $x \in M_r$, and $x$ is an end of a chord of $M_r$ in $\Sigma$ which forms angle $\pi/3$ with $M_r$. 
Again by the condition on the radius of curvature of $M_r$ and with the choice of $\varepsilon$ sufficiently small, this chord has length more than $\varepsilon$ which is impossible.
\end{itemize}

Let us consider any chord $[I_1I_2] \in Ch'$, such that it connects some components from $V_\varepsilon (G)$  (which do not have Steiner points as proven). 
Note that the set $]I_2I_1) \cap \Sigma$ (resp. $]I_1I_2) \cap \Sigma$) contains an energetic point (it may coincide with $I_1$ ($I_2$); if $I_1$ ($I_2$) is not energetic, an energetic point on $]I_2I_1) \cap \Sigma$ (resp. $]I_1I_2) \cap \Sigma$) exists by Lemma~\ref{hex} and the absence of Steiner points in the considered connected components and arcs); denote the nearest to $I_1$ (resp. $I_2$) energetic point of 
$]I_2I_1) \cap \Sigma$ (resp. $]I_1I_2) \cap \Sigma$) by $W_1$ (resp. $W_2$).

Consider the region $P$ bounded by the segments $[W_1Q_{W_1}]$, $[W_2Q_{W_2}]$, $[W_1W_2]$ and the lesser arc $\breve{[Q_{W_1}Q_{W_2}]}$ of $M$.
Let us show that the intersection of $\Int (P)$ with $\Sigma$ is nonempty. 
There are two tangent lines to $M_r$ parallel to $[W_1W_2]$; let $l$ be the nearest line to $[W_1W_2]$.
Note that $[I_1I_2] \in Ch' \subset Ch$, so $[I_1I_2] \cap \Int(N_r) \neq \emptyset$ and $l \cap [W_1W_2] = \emptyset$.
Consider a point $w \in l \cap M_r$ and note that $Q_w$ is not covered by $\Sigma$, because $\dist (Q_W, \Sigma)  = \dist(Q_w, [W_1,W_2]) > \dist (Q_w,l) = r$.
We got a contradiction, so $\Int(P) \cap \Sigma \neq \emptyset$.


Let us pick a point $x \in \Int(P) \cap \Sigma$ and consider the path in $\Sigma$ connecting $x$ with the segment $[W_1W_2]$.
The existence of this path gives that for some $i \in \{1,2\}$ (say, without loss of generality, $i = 1$) one has 
$W_i = I_i$ (in fact, $]W_1W_2[ \subset S_\Sigma$, which means that this path connects $x$ with $W_1$ without touching $]W_1W_2]$, 
but a neighbourhood in $\Sigma$ of an energetic point of $\Sigma \setminus N_r$ is either a single line segment or two line segments with angle at least $2\pi/3$, see Fig.~\ref{lmn23} and~\ref{lmn4}, and thus $W_1 \in M_r$) and $B_\delta(I_1) \cap \Int (P) \cap \Sigma \neq \emptyset$ for sufficiently small $\delta > 0$.
Let $k$ be the tangent line to $M_r$ at $I_1 = W_1$. Since $|I_1I_2| \leq \varepsilon$, the angle between $k$ and $[I_1I_2]$ is $O(\varepsilon)$.
Consider an arbitrary point $y \in \partial B_\delta (I_1) \cap \Int (P) \cap \Sigma$. Since $B_r(Q_{I_1}) \cap \Sigma = \emptyset$ and $|yI_1| = \delta$
the angle between $k$ and $[yI_1]$ is $O(\delta)$. Let $z$ be a projection of $y$ on $[I_1I_2]$.
Then $\angle yI_1z = O(\varepsilon + \delta)$ is the smallest angle (for sufficiently small $\varepsilon$, $\delta$) in right-angled triangle $\Delta yI_1z$.
Hence one can replace $]I_1z[$ by $[zy]$ in $\Sigma$.
The new set is still connected, covers $M$ and has strictly lower length than $\Sigma$. 
We got in this way a contradiction with the optimality of $\Sigma$, concluding the proof.

\end{proof}

\begin{proof}[Proof of Lemma~\ref{bpMr}]
Note that in $\Sigma$ there are at most two chords of $M_r$ ending at $I$. It is true because of the properties of a locally minimal network: the angle between two segments ending at the same point is greater or equal to $2\pi/3$.

Let us show that $I \in S_\Sigma$. Assume the contrary: let $I \in G_\Sigma$. Then $B_r(Q_I) \cap \Sigma =\emptyset$. 
There are two possibilities:
\begin{enumerate}
    \item [(1)] $I \in S$, where $S \in V_C(G)$;
    \item [(2)] $I \in S$, where $S \in V_A(G)$ (as mentioned after Lemma~\ref{accum} $S$ is non degenerate i.e.\ does not reduce to a single point $I$).
\end{enumerate} 

Recall that $\Sigma$ consists of a finite number of segments and a finite number of arcs of $M_r$.
In the case (1) the smoothness of $M_r$, Lemma~\ref{compcon} and the fact $B_r(Q_I) \cap \Sigma =\emptyset$ imply that the intersection of a small neighbourhood of $I$ with $S \setminus N_r$ is a subset of the tangent line to $M_r$ at $I$. 

Thus the set $\Sigma \cap B_{\varepsilon}(I) \setminus \Int(N_r)$ is contained in the union of the tangent line $\tau$ to $M_r$ at $I$ and the arc $M_r\cap \partial B_\varepsilon(I)$. 
Both $\tau \cap B_\varepsilon (I)$ and $M_r \cap B_\varepsilon (I)$ are split by $I$ into $2$ segments $[IE'_1]$, $[IE'_2]$ and $2$ arcs $\breve{[IE_1]}$, $\breve{[IE_2]}$ of $M_r$, respectively, where $E_1, E_2, E_1', E_2' \in \partial B_\varepsilon(I)$. We may assume $E_1$ in the same halfplane with $E_1'$ bounded by the normal to $M_r$ passing throw $I$. 
At least one arc and one segment (say, $\breve{[IE_1]}$ and $[IE'_1]$) have angle at most $\pi/2$ with the chord $[IB]$.
The cases (i) and (ii) below deal with the situation with nonempty set $\Sigma \cap (\breve{[IE_1]} \cup [IE'_1])$.
In the remaining cases $\Sigma \cap B_{\varepsilon}(I) \setminus \Int(N_r)$ is a subset of $\breve{[IE'_2]} \cup [IE_2]$ and therefore in (iii)-(vi) we deal with all the possible cases of $B_\varepsilon (I) \cap \breve{[IE'_2]}$ and $B_\varepsilon(I) \cap [IE_2]$ empty/nonempty:

\begin{enumerate}
    \item[(i)] There is such a segment $[IE] \subset \Sigma$, that $(IE)$ is the tangent line to $M_r$, $|IE| = \varepsilon$ and $\angle BIE \leq \pi/2$; 
    \item[(ii)] There is such an $\varepsilon>0$ and an arc $\breve{[IE]} \subset \Sigma \cap M_r$ that $|IE|=\varepsilon$ and $\angle BIE \leq \pi/2$; 
    \item[(iii)] There is such a small $\varepsilon>0$ that $\overline{B_\varepsilon(I)} \cap \Sigma$  is equal to $[FI]\cup [IE]$ where $F,E \in \partial B_\varepsilon(I)$, $[FI]\subset [BI]$ and  $[IE]$ is a subset of the tangent line to $M_r$ at point $I$;
    \item[(iv)] There is such a small $\varepsilon>0$ that $\overline{B_\varepsilon(I)} \cap \Sigma$  is equal to $[FI]\cup \breve{[IE]}$, where $F,E \in \partial B_\varepsilon(I)$, $[FI]\subset [IB]$ and $\breve{[IE]} \subset M_r$;
    \item[(v)] There is such a small $\varepsilon>0$ that $\overline{B_\varepsilon(I)} \cap \Sigma$  contains $[FI]\cup [IC]\cup \breve{[ID]}$ where  $[IC]$ is a subset of the tangent line to $M_r$ at point $I$, $[FI] \subset [BI]$, $\breve{[ID]} \subset M_r$ and $\angle CID<\pi/6$;
    \item[(vi)] There is such an $\varepsilon>0$ that $\overline{B_\varepsilon(I)}$ is a subset of chord $[IB]$.
\end{enumerate}

We will show that all these cases are impossible. Let $\xi$ stand for the segment $[IE]$ in the cases (i) and (iii), and for $\breve{[IE]}$ in the cases (ii) and (iv).

{\sc Cases} (i), (ii): Let $F \defeq [BI] \cap B_{\varepsilon}(I)$ and $l^\varepsilon$ be the lesser arc of $\partial B_\varepsilon (I )$ limited by intersections with $\partial B_r(Q_I )$ and $M_r$. It is easy to see that $\H(l^\varepsilon)=O(\varepsilon^2)$ and $|FI|+ \H(\xi) - \H(\St(F, I , E)) = c\varepsilon + o(\varepsilon)$ with $c > 0$, where $\St(F, I , E)$ is a Steiner tree connecting points $F, I, E$. Then the length of $\Sigma' \defeq \Sigma \setminus ([F I ] \cup \xi) \cup l^\varepsilon \cup \St(F, I , E)$ is less than $\H(\Sigma)$ for sufficiently small $\varepsilon$. Moreover $\Sigma'$ is still connected and $F_M(\Sigma') \leq F_M(\Sigma)$. This gives us a contradiction with optimality of $\Sigma$.

  \begin{figure}
  \begin{center}
    \definecolor{ffqqqq}{rgb}{1.0,0.0,0.0}
    \definecolor{xdxdff}{rgb}{0.0,0.0,1.0}
    \definecolor{qqzzqq}{rgb}{0.0,0.6,0.0}
    \definecolor{qqqqff}{rgb}{0.0,0.0,1.0}
    \definecolor{uuuuuu}{rgb}{0.26666666666666666,0.26666666666666666,0.26666666666666666}
    \definecolor{qqzzqq}{rgb}{0.0,0.6,0.0}
    
    \hspace*{\fill}
    
    \begin{tikzpicture}[node distance=1mm][line cap=round,line join=round,>=triangle 45,x=0.5cm,y=0.5cm]
    \clip(-4.72, 1.7699999999999992) rectangle (9.3, 7.09999999999996);
    \draw [shift={(0.92,2.08)}]  plot[domain=1.382253704689843:3.4931410901876654,variable=\t]({1.0*5.0*cos(\t r)+-0.0*5.0*sin(\t r)},{0.0*5.0*cos(\t r)+1.0*5.0*sin(\t r)});

    \draw [shift={(0.92,2.08)},line width=2.pt,color=qqzzqq]  plot[domain = 2.002253704689843:2.50931410901876654,variable=\t]({1.0*5.0*cos(\t r)+-0.0*5.0*sin(\t r)},{0.0*5.0*cos(\t r)+1.0*5.0*sin(\t r)});
    
    \draw [shift={(0.92,2.08)},color=uuuuuu]  plot[domain=1.352253704689843:3.4931410901876654,variable=\t]({1.0*4.0*cos(\t r)+-0.0*4.0*sin(\t r)},{0.0*4.0*cos(\t r)+1.0*4.0*sin(\t r)});
    
    \draw [shift={(0.92,2.08)},line width=1.2pt,color=ffqqqq]  plot[domain=2.0122253704689843:2.4931410901876654,variable=\t]({1.0*4.0*cos(\t r)+-0.0*4.0*sin(\t r)},{0.0*4.0*cos(\t r)+1.0*4.0*sin(\t r)});
    \draw [line width=1.2pt,color=ffqqqq] (-0.7889278373897851,5.696568214010939)-- (1.84,4.56);
    \draw [dotted,domain=-4.0:0.92] plot(\x,{(-8.853752895565991--2.415811981703012*\x)/-3.1880797463457786});
    \draw [dotted,domain=-4.72:0.92] plot(\x,{(-6.8818126586608175--3.6165682140109388*\x)/-1.7089278373897852});

    \begin{scriptsize}
    \draw [fill=uuuuuu] (0.92,2.08) circle (1.5pt);
    \draw[color=uuuuuu] (1.060000000000002,2.3599999999999985) node {$O$};
    \draw [fill=ffqqqq] (-0.7889278373897851,5.696568214010939) circle (1.5pt);
    \draw[color=ffqqqq] (-0.6399999999999985,5.979999999999997) node {$I$};
    \draw [fill=ffqqqq] (-2.2680797463457787,4.495811981703012) circle (1.5pt);
    \draw[color=ffqqqq] (-2.219999999999999,4.779999999999998) node {$C$};
    \draw [fill=ffqqqq] (1.84,4.56) circle (1.5pt);
    \draw[color=ffqqqq] (1.9800000000000024,4.839999999999998) node {$B$};
    \draw [fill=qqzzqq] (-1.2161597967372333,6.600710267513673) circle (1.5pt);
    \draw[color=qqzzqq] (-1.0799999999999987,6.879999999999997) node {$Q_I$};
    \draw [fill=qqzzqq] (-3.0591597967372333,5.111710267513673) circle (1.5pt);
    \draw[color=qqzzqq] (-3.0799999999999987,5.459999999999997) node {$Q_C$};

    \end{scriptsize}

    \draw (7.92,2.08) circle (5.0cm);
    \draw (7.92,2.08) circle (4.0cm);
    \draw [shift={(7.92,2.08)},line width=1.2000000000000002pt,color=ffqqqq]  plot[domain=2.0122253704689843:2.4931410901876654,variable=\t]({1.0*4.0*cos(\t r)+-0.0*4.0*sin(\t r)},{0.0*4.0*cos(\t r)+1.0*4.0*sin(\t r)});
    \draw [shift={(7.92,2.08)},line width=2.pt,color=qqzzqq]  plot[domain=2.0122253704689843:2.4931410901876654,variable=\t]({1.0*5.0*cos(\t r)+-0.0*5.0*sin(\t r)},{0.0*5.0*cos(\t r)+1.0*5.0*sin(\t r)});
    
    \draw [line width=1.20pt,color=ffqqqq] (6.2289278373897867,5.696568214010938)-- (8.84,4.56);
    
    \draw [dotted,domain=-4.0:0.92] plot(\x + 7.0,{(-8.853752895565991--2.415811981703012*\x)/-3.1880797463457786});
    \draw [dotted,domain=-4.72:0.92] plot(\x + 7.0,{(-6.8818126586608175--3.6165682140109388*\x)/-1.7089278373897852});
    
    \draw [line width=2.0pt,color=qqqqff] (5.503176570769273,5.277615106560769)-- (7.00621278219520972,5.362499184906717);
    \draw [dash pattern=on 3pt off 3pt] (5.7961597967372333,6.600710267513673) circle (1.0000000000000002cm);
    \draw [line width=2.0pt,color=qqqqff] (5.503176570769273,5.277615106560769)-- (5.593837188194806,5.620472696385129);
    \draw [line width=0.4pt] (5.7961597967372333,6.600710267513673)-- (5.5939837188194806,5.620472696385129);
    \draw [shift={(7.92,2.08)},line width=2.0pt,color=qqqqff]  plot[domain=2.215290483921135:2.4931410901876654,variable=\t]({1.0*4.0*cos(\t r)+-0.0*4.0*sin(\t r)},{0.0*4.0*cos(\t r)+1.0*4.0*sin(\t r)});
    \draw [line width=2.0pt,color=qqqqff] (8.84,4.56)-- (6.98621278219520972,5.362499184906717);
    
    \begin{scriptsize}
    
    \draw[color=ffqqqq] (-4.0,6.9) node {$(a)$};
    \draw[color=ffqqqq] (3.0,6.9) node {$(b)$};

    \draw [fill=uuuuuu] (7.92,2.08) circle (1.5pt);
    \draw[color=uuuuuu] (8.060000000000002,2.3599999999999985) node {$O$};
    \draw [fill=ffqqqq] (6.2289278373897867,5.696568214010938) circle (1.5pt);
    \draw[color=ffqqqq] (6.2799999999999985,5.979999999999997) node {$I$};
    \draw [fill=ffqqqq] (4.7480797463457787,4.495811981703012) circle (1.5pt);
    \draw[color=ffqqqq] (4.649999999999999,4.819999999999998) node {$C$};
    \draw [fill=ffqqqq] (8.84,4.56) circle (1.5pt);
    \draw[color=ffqqqq] (8.9800000000000024,4.839999999999998) node {$B$};
    
    \draw [fill=xdxdff] (5.523176570769273,5.277615106560769) circle (1.5pt);
    \draw[color=xdxdff] (5.3099999999999988,5.399999999999998) node {$E$};
    \draw [fill=xdxdff] (6.98621278219520972,5.362499184906717) circle (1.5pt);
    \draw[color=xdxdff] (7.12000000000000174,5.639999999999997) node {$F$};
    \draw [fill=qqzzqq] (5.8061597967372333,6.600710267513673) circle (1.5pt);
    \draw[color=qqzzqq] (5.9399999999999987,6.879999999999997) node {$Q_I$};
    \draw [fill=xdxdff] (5.5939837188194806,5.620472696385129) circle (1.5pt);
    \draw[color=xdxdff] (5.7599999999999987,5.899999999999998) node {$H$};
    \draw [fill=qqzzqq] (3.9591597967372333,5.111710267513673) circle (1.5pt);
    \draw[color=qqzzqq] (3.87099999999999987,5.459999999999997) node {$Q_C$};

    \end{scriptsize}
    \end{tikzpicture}\hspace*{\fill}
    \caption{The case (iv) from Lemma~\ref{bpMr}: (a) the (impossible) part of the minimizer;
(b) the better competitor.}

    \label{inner2}
    \end{center}
    \end{figure}
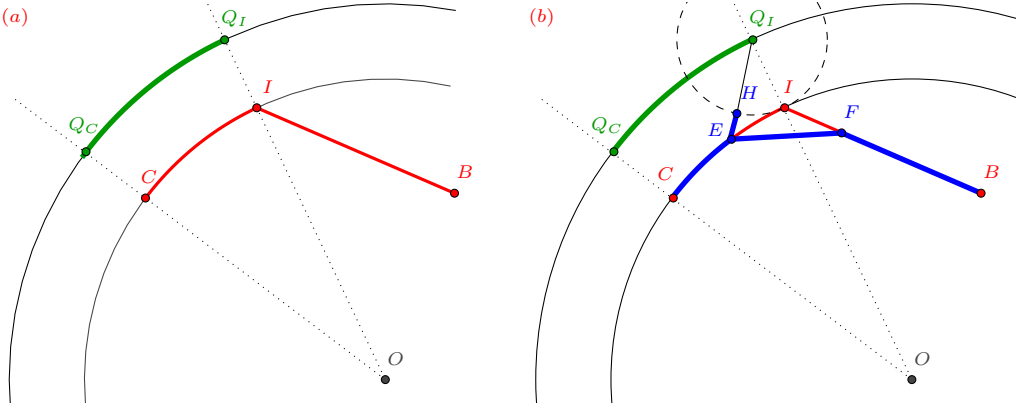

  

{\sc Cases} (iii), (iv): Note that $|FI|=|IE|=\varepsilon$ (see Fig.~\ref{inner2}(a)), so $\H(\xi)=\varepsilon+o(\varepsilon)$ when $\varepsilon \to 0^+$, 
because $M_r$ is smooth. Let $H$ be the point of intersection of $[EQ_I]$ and $\partial B_r(Q_I)$  (see Fig.~\ref{inner2}(b)). 
Note that $(IQ_I)$ is perpendicular to the tangent line to $M_r$ at the point $I$. 
Thus 
    \begin{align*}
    |EH|&=|EQ_I |-|Q_I H|=\sqrt{|EI |^2+r^2}-r=\sqrt{\varepsilon^2+r^2}-r\\
    &=r\sqrt{1+o(\varepsilon)}-r=o(\varepsilon).
    \end{align*}
Now, since the angle between $\xi$ and 
the segment $[FI]$ is less than $\pi$, we get
    \[
    |EF|=\sqrt{2\varepsilon^2-2\varepsilon^2\cos\angle EI F}= \sqrt 2 \varepsilon \sqrt {1-\cos\angle EI F}< 2 \varepsilon - c \varepsilon,\ \  \mbox{for some}  \ \ c > 0
    \]
and therefore
    \[
    |EH|+|EF| < \H(\xi)+|IF| = 2\varepsilon+o(\varepsilon)
    \]
for sufficiently small $\varepsilon>0$. So we have a contradiction with the optimality of $\Sigma$, because we show that $(\Sigma \setminus B_\varepsilon (I)) \cup [EH] \cup [EF]$ is the better competitor.

{\sc Case} (v): Let $H \in [IC)$ be such a point that $(DH) \perp (IC)$. Then the set 
\[
\Sigma'=\Sigma \setminus \breve{]ID]} \cup [HD] 
\]
is still connected, has energy $F_M$ not greater than $\Sigma$ and strictly smaller length, since $|HD|<|ID|/2 \leq \H(\breve{[ID]})/2$. It means $\Sigma'$ is the better competitor than $\Sigma$, again a contradiction.

{\sc Case} (vi): In this case $S \in V_A(G)$ and $S = \{I\}$, which is impossible.


So all cases are impossible and we have a contradiction which implies $I \in S_\Sigma$. Because of Lemma~\ref{hex} $I$ can not be a Steiner point. Then there exists an $\varepsilon > 0$ such that $S_\Sigma \cap B_{\varepsilon}(I)$ is a segment. 
\end{proof}

\begin{proof}[Proof of Lemma~\ref{27}] 
Let $\breve{BC}$ be as in the statement being proven.

Suppose that there is a segment $[IJ] \subset \Sigma$ such that $I = \breve{]BC[}\cap [IJ]$. We claim that  $B_\varepsilon(I) \cap \Sigma \subset \breve{[BC]}$. In fact, by Lemma~\ref{bpMr} $[IJ]$ cannot be a part of a chord of $M_r$, so $]IJ] \subset \Sigma \setminus \Int (N_r)$. Note that in this case $I$ is energetic (because $B_\varepsilon(I)$ is not a segment or a tripod for every $\varepsilon > 0$). 
Hence $B_r(Q_I) \cup \Sigma = \emptyset$, so $[IJ]$ is a part of the tangent line to $M_r$ at $I$.
Let us choose an $\varepsilon > 0$ and set $\{D_{1},D_2\} := \breve{[BC]} \cap \partial B_\varepsilon(I)$, $E := [IJ] \cap \partial B_\varepsilon(I)$.
If $\varepsilon > 0$ is sufficiently small one of the angles $\angle D_1IE$, $\angle D_2IE$ is less than $\pi/6$ (say $\angle D_1IE$).
Let $H \in [IJ)$ be such a point that $(D_1H) \perp (IJ)$. Then the set 
\[
\Sigma' := \Sigma \setminus \breve{]ID_1]} \cup [HD_1] 
\]
is still connected, has energy $F_M$ not greater than $F_M(\Sigma)$ and strictly smaller length, since $|HD_1|<|ID_1|/2 \leq \H(\breve{[ID_1]})/2$. It means that $\Sigma'$ is better competitor than $\Sigma$. We got a contradiction, showing thus $B_\varepsilon(I) \cap \Sigma \subset \breve{[BC]}$ for $I \in \breve{]BC[}$.

Let us prove now that $B_\varepsilon (B) \setminus \breve{[BC]}$ is a subset of the tangent line to $M_r$ at $B$ (the analogous statement for the point $C$ is completely symmetric). 
By Lemma~\ref{bpMr} there is no chord of $M_r$ in $\Sigma$ with endpoint $B$. So the set $B_\varepsilon (B) \setminus \breve{[BC]}$ is a subset of $\Sigma \setminus N_r$.

\begin{figure}
\begin{center}

\definecolor{qqttcc}{rgb}{0.,0.2,0.8}
\definecolor{uuuuuu}{rgb}{0.26666666666666666,0.26666666666666666,0.26666666666666666}
\definecolor{wwzzqq}{rgb}{0.0,0.6,0.0}
\definecolor{xdxdff}{rgb}{0.,0.2,0.8}
\definecolor{ffqqqq}{rgb}{1.,0.,0.}
\begin{tikzpicture}[line cap=round,line join=round,>=triangle 45,x=1.0cm,y=1.0cm]
\clip(-14.18,1.24) rectangle (-0.1,6.98);
\draw(-1.26,1.94) circle (4.cm);
\draw(-1.26,1.94) circle (3.cm);
\draw [shift={(-1.26,1.94)},line width=2pt,color=ffqqqq]  plot[domain=1.9756881130799804:2.647651284670212,variable=\t]({1.*3.*cos(\t r)+0.*3.*sin(\t r)},{0.*3.*cos(\t r)+1.*3.*sin(\t r)});
\draw [dotted,domain=-14.179999999999994:-1.26] plot(\x,{(--1.1817578957375012--2.7574350900541735*\x)/-1.1817578957375037});
\draw [dash pattern=on 3pt off 3pt] (-2.835677194316671,5.616580120072231) circle (1.cm);
\draw (-2.835677194316671,5.616580120072231)-- (-3.2106825881547616,4.21921860300188);
\draw [line width=2.pt,color=qqttcc] (-3.2106825881547616,4.21921860300188)-- (-3.0948725082227573,4.650755193261414);
\draw [shift={(-1.26,1.94)},line width=2.pt,color=qqttcc]  plot[domain=2.278680208561553:2.647651284670212,variable=\t]({1.*3.*cos(\t r)+0.*3.*sin(\t r)},{0.*3.*cos(\t r)+1.*3.*sin(\t r)});
\draw [shift={(-1.26,1.94)},line width=2pt,color=wwzzqq]  plot[domain=1.9756881130799802:2.647651284670212,variable=\t]({1.*4.*cos(\t r)+0.*4.*sin(\t r)},{0.*4.*cos(\t r)+1.*4.*sin(\t r)});
\draw [dotted,domain=-14.179999999999994:-9.0] plot(\x,{(--9.277405779246862--1.4986199169028378*\x)/-2.5988725141226396});
\draw [dotted,domain=-14.179999999999994:-9.0] plot(\x,{(--24.602993658787383--2.8830012801984877*\x)/-0.8296406561722272});
\draw [shift={(-9.,1.62)},line width=2pt,color=ffqqqq]  plot[domain=1.8509953187487473:2.6185249908185098,variable=\t]({1.*3.*cos(\t r)+0.*3.*sin(\t r)},{0.*3.*cos(\t r)+1.*3.*sin(\t r)});
\draw [shift={(-9.,1.62)},line width=2pt,color=wwzzqq]  plot[domain=1.8509953187487476:2.6185249908185098,variable=\t]({1.*4.*cos(\t r)+0.*4.*sin(\t r)},{0.*4.*cos(\t r)+1.*4.*sin(\t r)});
\draw [shift={(-9.,1.62)}] plot[domain=1.5443124541971658:3.2529718962988716,variable=\t]({1.*3.*cos(\t r)+0.*3.*sin(\t r)},{0.*3.*cos(\t r)+1.*3.*sin(\t r)});
\draw [shift={(-9.,1.62)}] plot[domain=1.544312454197166:3.2529718962988716,variable=\t]({1.*4.*cos(\t r)+0.*4.*sin(\t r)},{0.*4.*cos(\t r)+1.*4.*sin(\t r)});
\draw [dotted] (-1.26,1.94)-- (-8.005688073394497,5.572293577981651);
\draw [shift={(-9.,1.62)},line width=2pt,color=wwzzqq]  plot[domain=1.8509953187487476:2.6185249908185098,variable=\t]({1.*4.*cos(\t r)+0.*4.*sin(\t r)},{0.*4.*cos(\t r)+1.*4.*sin(\t r)});
\begin{scriptsize}
\draw [fill=black] (-1.26,1.94) circle (1.5pt);
\draw[color=black] (-1.08,2.28) node {$O$};
\draw [fill=ffqqqq] (-2.4417578957375037,4.6974350900541735) circle (1.5pt);
\draw[color=ffqqqq] (-2.3,4.98) node {$B$};
\draw [fill=xdxdff] (-3.901413299766526,3.3622994691050523) circle (1.5pt);
\draw[color=xdxdff] (-3.95,3.64) node {$C$};
\draw [fill=xdxdff] (-3.2106825881547616,4.21921860300188) circle (1.5pt);
\draw[color=xdxdff] (-3.,4.16) node {$B_1$};
\draw [fill=wwzzqq] (-2.835677194316671,5.616580120072231) circle (1.5pt);
\draw[color=wwzzqq] (-2.66,5.96) node {$Q_B$};
\draw [fill=xdxdff] (-3.0948725082227573,4.650755193261414) circle (1.5pt);
\draw[color=xdxdff] (-2.91,4.94) node {$I$};
\draw [fill=wwzzqq] (-4.781884399688702,3.8363992921400696) circle (1.5pt);
\draw[color=wwzzqq] (-4.8,4.28) node {$Q_C$};
\draw [fill=black] (-9.,1.62) circle (1.5pt);
\draw[color=black] (-8.86,1.9) node {$O$};
\draw [fill=ffqqqq] (-11.59887251412264,3.118619916902838) circle (1.5pt);
\draw[color=ffqqqq] (-11.74,3.5) node {$C$};
\draw [fill=ffqqqq] (-9.829640656172227,4.503001280198488) circle (1.5pt);
\draw[color=ffqqqq] (-9.68,4.78) node {$B$};
\draw [fill=wwzzqq] (-12.465163352163524,3.6181598892037865) circle (1.5pt);
\draw[color=wwzzqq] (-12.54,3.92) node {$Q_C$};
\draw [fill=wwzzqq] (-10.10618754156297,5.4640017069313185) circle (1.5pt);
\draw[color=wwzzqq] (-9.96,5.74) node {$Q_B$};
\end{scriptsize}
\end{tikzpicture}

\caption{Picture to Lemma~\ref{27}. An end of an arc of $M_r \cap \Sigma$ cannot be an endpoint of $\Sigma$.}
\label{L210}

\end{center}
\end{figure}
We claim first that $B$ is not an endpoint of $\Sigma$ i.e.\ $B_\varepsilon (B) \setminus \breve{[BC]} \neq \emptyset$.
Assume the contrary and recall that $Q_B, Q_C \in M$ are such points that $\dist (B, Q_B) = \dist (C, Q_C) = r$. 
Then one can set $B_1 := \partial B_\varepsilon (B) \cap \breve{[BC]}$ and replace $\breve{[B_1 B]}$ by the segment $[B_1 I] \defeq [B_1 Q_B] \setminus B_r(Q_B)$, producing the competitor of strictly lower length because $\breve{[BC]} \setminus \breve{[B_1 B]} \cup [B_1 I] = \breve{[B_1C]} \cup [B_1 I]$ still covers the arc $\breve{[Q_B Q_C]}$ of $M$ (when $\varepsilon$ is sufficiently small), see Fig.~\ref{L210}.

Therefore we have proven that for sufficiently small $\varepsilon > 0$ the set $B_\varepsilon (B) \setminus \breve{[BC]}$ is a nonempty subset of $\Sigma \setminus N_r$.
If $B$ is energetic then $B_r(Q_B) \cap \Sigma = \emptyset$, hence $B_\varepsilon (B) \setminus \breve{[BC]}$ is a subset of the tangent line to $M_r$ at point $B$ showing the claim.
So $B \in S_\Sigma$, hence $B_\varepsilon (B)$ is a segment or a tripod for sufficiently small $\varepsilon > 0$. But the case of a tripod is impossible by Lemma~\ref{hex},
while the case of a segment is only possible recalling smoothness of $M_r$ (and part of $M_r$ in a neighbourhood of $B$ is in fact flat).

Summing up, the only segments intersecting $\breve{[BC]}$ are segments tangent to $M_r$ at points $B$ and $C$.
As a consequence of Lemma~\ref{accum}  $\Sigma$ consists of a finitely many segments and maximal arcs of $M_r$, so when $\varepsilon$ is small, $B_\varepsilon (\breve{[BC]})$
contains only $2$ segments which is proven to be tangent to $M_r$ at points $B$ and $C$, respectively. 
The statement is proven. 
\end{proof}

\begin{proof}[Proof of Lemma~\ref{debil}]
Consider a point $C \in M_r \cap \Sigma$. By Lemma~\ref{27} if $C$ belongs to some non degenerate arc of $\Sigma \cap M_r$ with an energetic point in its interior (i.e.\ an element of $V_A(G)$) the statement is true.
Note that if there is a chord $[IC] \subset \Sigma$ of $M_r$ then Lemma~\ref{bpMr} implies the claim. 
Thus $B_\varepsilon (C) \cap \Int(N_r) = \emptyset$.
If $C \in S_\Sigma$ then by Lemma~\ref{hex} its neighbourhood cannot be a tripod, so it is a segment and the statement of Lemma is obvious.
It remains to consider the case when $B_\varepsilon (C) \cap \Int(N_r) = \emptyset$ and $C$ is energetic, which implies $B_r(Q_C) \cap \Sigma = \emptyset$ so 
the set $B_\varepsilon (C)\cap \Sigma$ is just a segment (because $\Sigma$ consists of a finite number of arcs of $M_r$ and segments by Lemma~\ref{accum}) 
which must be a subset of the tangent line to $M_r$ at $C$, the claim follows. 
\end{proof}

\begin{figure}
\begin{center}
\definecolor{qqzzqq}{rgb}{0.,0.6,0.}
\definecolor{qqwuqq}{rgb}{0.,0.39215686274509803,0.}
\definecolor{ffqqqq}{rgb}{1.,0.,0.}
\definecolor{uuuuuu}{rgb}{0.26666666666666666,0.26666666666666666,0.26666666666666666}
\begin{tikzpicture}[line cap=round,line join=round,>=triangle 45,x=2.0cm,y=2.0cm]
\clip(-4.806773630713251,2.23198626522033975) rectangle (2.2105066097074855,5.227743261451338);
\draw[color=qqwuqq,fill=qqwuqq,fill opacity=0.1] (-0.7828834073650064,3.4051645353014964) -- (-1.0077261207531172,3.1946321540262237) -- (-0.7921937394778444,2.959789440638113) -- (-0.5573510260897336,3.1753218219133856) -- cycle; 
\draw[color=qqwuqq,fill=qqwuqq,fill opacity=0.1] (-1.514457129560125,2.594896844456767) -- (-1.5239369096190678,2.9274660008583168) -- (-1.8405060660206177,2.899986220799374) -- (-1.832026285961675,2.5674170643978242) -- cycle; 
\draw(0.,0.) circle (6.cm);
\draw(0.,0.) circle (10.cm);
\draw [line width=1.6pt,color=ffqqqq] (-1.5039208365385677,2.5958085671761575)-- (-2.38,2.52);
\draw [line width=1.6pt,color=ffqqqq] (-0.8983931722680689,2.8623224325750085)-- (-0.16,3.54);
\draw [dash pattern=on 4pt off 4pt] (-1.8160410203327413,4.658540008679631)-- (-1.832026285961675,2.5674170643978242);
\draw [dash pattern=on 2pt off 2pt] (-2.38,2.52)-- (-4.377316181929332,2.4164236059556314);
\draw [dash pattern=on 2pt off 2pt] (-0.16,3.54)-- (1.3910844552857882,4.802591387809527);
\draw [dash pattern=on 4pt off 4pt] (-2.00622010237341,4.579855991276672)-- (-0.5573510260897336,3.1753218219133856);
\draw [shift={(0.,0.)},line width=1.6pt,color=qqzzqq]  plot[domain=1.2888591209561593:2.637189829252206,variable=\t]({1.*5.*cos(\t r)+0.*5.*sin(\t r)},{0.*5.*cos(\t r)+1.*5.*sin(\t r)});
\draw [dotted,domain=-1.5039208365385677:2.2105066097074855] plot(\x,{(--0.9484688340602507--0.030107927696170655*\x)/0.3479412564060971});
\draw [dotted,domain=-4.806773630713251:-0.8983931722680689] plot(\x,{(-0.949680833872411-0.23640593770268037*\x)/-0.2575864078644017});
\begin{scriptsize}
\draw [fill=uuuuuu] (0.,0.) circle (1.5pt);
\draw[color=uuuuuu] (0.10682516718306762,0.21647353943781225) node {$A$};
\draw[color=qqwuqq] (-2.862671451900328,4.278231616071924) node {$q_{S_l}$};
\draw [fill=ffqqqq] (-1.5039208365385677,2.5958085671761575) circle (1.5pt);
\draw [fill=ffqqqq] (-2.38,2.52) circle (1.5pt);
\draw[color=ffqqqq] (-2.367329603665918,2.7208912704671144) node {$V_l$};
\draw[color=ffqqqq] (-1.9217247952511924,2.7362331187015247) node {$W_l$};
\draw [fill=ffqqqq] (-0.8983931722680689,2.8623224325750085) circle (1.5pt);
\draw[color=ffqqqq] (-0.7947525938988257,2.7964960788818405) node {$C_r$};
\draw [fill=ffqqqq] (-0.16,3.54) circle (1.5pt);
\draw[color=ffqqqq] (-0.05846408901527949,3.7626793996932606) node {$V_r$};
\draw [fill=qqwuqq] (-4.377316181929332,2.4164236059556314) circle (1.5pt);
\draw[color=qqwuqq] (-4.365826974064115,2.6357071983408935) node {$Q_l^{S_l}$};
\draw [fill=qqwuqq] (1.3910844552857882,4.802591387809527) circle (1.5pt);
\draw[color=qqwuqq] (1.4892444008419707,5.024888265207912) node {$Q_r^{S_r}$};
\draw [fill=qqwuqq] (-1.8160410203327413,4.658540008679631) circle (1.5pt);
\draw[color=qqwuqq] (-1.7113566509987506,4.874625305027596) node {$Q_r^{S_l}$};
\draw [fill=ffqqqq] (-1.832026285961675,2.5674170643978242) circle (1.5pt);
\draw[color=ffqqqq] (-1.6563829470167822,2.7559701585212093) node {$C_l$};
\draw[color=black] (-2.026908867377413,3.7476531036752294) node {$r$};
\draw [fill=qqwuqq] (-2.00622010237341,4.579855991276672) circle (1.5pt);
\draw[color=qqwuqq] (-1.9566984992331608,4.799493824937438) node {$Q_l^{S_r}$};
\draw [fill=uuuuuu] (-1.1559795801324706,2.625916494872328) circle (1.5pt);
\draw[color=uuuuuu] (-1.0501996262053621,2.6460753425933354) node {$O$};
\draw[color=black] (-3.34922291696419,2.6460753425933354) node {$r$};
\draw[color=black] (0.7980347840125191,4.123310504126018) node {$r$};
\draw [fill=ffqqqq] (-0.5573510260897336,3.1753218219133856) circle (1.5pt);
\draw[color=ffqqqq] (-0.48914778548409995,3.3870219992424717) node {$W_r$};
\draw[color=black] (-1.4108307306381194,3.8378108797834187) node {$r$};
\draw[color=qqzzqq] (-1.0864881310889082,4.944572712991534) node {$q_{S_r}$};
\end{scriptsize}
\end{tikzpicture}
\caption{Picture to Lemma~\ref{disj}}
\label{ha}
\end{center}

\end{figure}
\begin{proof}[Proof of Lemma~\ref{disj}]
Recall that $m(S_l) = m(S_r) = 1$. 
Denote the ends of $q_{S_l}$ and $q_{S_r}$ in the following way: $q_{S_l} = \breve{[Q_l^{S_l}Q_r^{S_l}]}$, $q_{S_r} = \breve{[Q_l^{S_r}Q_r^{S_r}]}$.
Suppose the contrary, i.e.\ that $Q_r^{S_l} \in \breve{]Q_l^{S_r}Q_r^{S_r}[}$, $Q_l^{S_r} \in \breve{]Q_l^{S_l}Q_r^{S_l}[}$.
Suppose that $n(S_l) = 2$ or $n(S_r) = 2$ (let $n(S_l) = 2$, the case $n(S_r) = 2$ is completely analogous). Then by Remark~\ref{corr} there is an energetic point of $S_l$ corresponding to the point $Q_r^{S_l}$. 
But $B_r (Q_r^{S_l}) \cap \Sigma \neq \emptyset$, because $Q_r^{S_l} \in \breve{]Q_l^{S_r}Q_r^{S_r}[} = q_{S_r}$.
So we have a contradiction with the assumption $n(S_l) = 2$, and hence $S_l$ coincides with the segment $[C_lV_l]$. 
Clearly, $V_l$, $C_l$ and $Q_l^{S_l}$ lie on the same line (otherwise one can replace $[V_lV']$ by the part of the segment $[V'Q_l^{S_l}]$, where $V' := \partial B_\varepsilon (V_l) \cap [V_lC_l]$ producing a competitor of strictly lower length). Hence $[C_lV_l]$ is tangent to $B_r(Q^{S_l}_r)$ (see Fig.~\ref{ha}).

Let $W_l$ be such a point of $[C_lV_l]$ that $\dist (W_l, Q_r^{S_l}) = r$, $W_r$ be such a point of $[C_rV_r]$ that $\dist (W_r, Q_l^{S_r}) = r$.
Note that the points $C_l$, $V_l$, $Q_l^{S_l}$ lie on the same line, so $\dist (W_lQ_l^{S_l}) \geq r = \dist (W_l, Q_r^{S_l})$, so $\angle Q_r^{S_l}Q_l^{S_l}W_l \leq \angle Q_l^{S_l}Q_r^{S_l}W_l$.
The segment $[C_lV_l]$ is tangent to  $B_r(Q^{S_l}_r)$, hence $(Q_r^{S_l}W_l) \perp (V_lC_l)$. Calculating angles in triangle $\Delta Q_r^{S_l}Q_l^{S_l}W_l$ we have $\angle Q_r^{S_l}Q_l^{S_l}W_l \leq \pi/4$. Obviously, $\angle Q_r^{S_r}Q_l^{S_l}W_l \leq \angle Q_r^{S_l}Q_l^{S_l}W_l$, so $\angle Q_r^{S_r}Q_l^{S_l}W_l \leq \pi/4$.
By symmetry we have inequality $\angle Q_l^{S_l}Q_r^{S_r}W_r \leq \pi/4$. Denote by $O$ the intersection point of $(V_lC_l)$ and $(V_rC_r)$.
From the triangle $\Delta Q_r^{S_r}Q_l^{S_l}O$ we have $\angle Q_r^{S_r}OQ_l^{S_l} \geq \pi/2$.

Note that $2r > |W_lW_r| \geq |C_lC_r|$ and $\angle Q_r^{S_r}OQ_l^{S_l} = \angle C_lOC_r \geq \pi/2$.
It means that $|C_lO| <2r$ and $|C_rO| < 2r$.
Hence the intersection point of the rays $[V_lC_l)$ and $[V_rC_r)$ belongs to $N_r$, that contradicts the optimality of $\Sigma$.
\end{proof}

\subsection{Proof of the central Lemma}

First, let us prove the following technical lemma.

\begin{lemma}
Let $S$ be the closure of a connected component of $\Sigma \setminus N_r$ such that $n(S) = 2$. Let $W \in G_\Sigma \cap S$ be an energetic point of $S$, such that $\overline{B_\varepsilon (W)} \cap S = [J_1W] \cup [WJ_2]$. Then $(QW)$ is the bisector of $\angle J_1WJ_2$, where $Q$ is the end of arc $q_S$ corresponding to $W$ (in the sense of Remark~\ref{corr}).
\label{tech}
\end{lemma}

\begin{proof}[Proof of Lemma~\ref{tech}]
Suppose the contrary i.e.\ that $\angle J_1WQ \neq \angle QWJ_2$. Let $l$ be the tangent line to $B_r(Q)$ at the point $W$. Then $\H (\Sigma \cap B_\varepsilon (W)) - \H ([J_1Y] \cup [YJ_2]) = O(\varepsilon)$, where $Y$ is such a point in $l$ that $\angle J_1YQ = \angle QYJ_2$. On the other hand, $\dist (Y, \partial B_r(Q)) = O(\varepsilon^2)$; let $V \in \partial B_r(Q)$ be such a point that $\dist (Y, \partial B_r(Q)) = \dist (Y, V)$. Then the set  $(\Sigma \setminus B_\varepsilon(W)) \cup [J_1Y] \cup [J_2Y] \cup [YV]$ is connected, covers $q(S)$ and has strictly lower length than $\Sigma$, giving a desired contradiction.
\end{proof}

Finally, we are ready to prove the central Lemma.

\begin{proof}[Proof of Lemma~\ref{central}]
Obviously, if $S$ is an arc, then the compared values are equal.

It suffices thus to consider the case when $S$ is the closure of a connected component of $\Sigma \setminus N_r$.
Denote by $Q_l$ and $Q_r$ the ends of $q_S$. Let $O$ be an intersection point of the normals to $M$ at points $Q_l$ and $Q_r$. 
It exists unless $\turn (q_S) = 0$ in which case the claim is obvious.
Note that $\turn (q_S) = \angle Q_lOQ_r$ and denote for brevity thus value by $\gamma$. Also one has $|Q_lO|\geq R$, $|Q_rO| \geq R$.
Note that Lemmas~\ref{hex} and~\ref{compcon} as well as Corollary~\ref{co_graph1} hold true when $R > 2a_M(r)+r$ which is
guaranteed when $R > 5r$ (or $R > 4.98 r$ in the case when $M$ is a circumference of radius $R$), i.e.\ under the conditions of the statement being proven.

By Lemma~\ref{compcon} $S$ is a locally minimal network for at most $n(S) + m(S) \leq 4$ points. 
All the possible combinatorial types of such networks are listed in Figures~\ref{lmn23} and~\ref{lmn4}.
Note that if $S$ is a middle component then $m(S) = 2$, otherwise $m(S) = 1$. Let us analyze all the possible types one by one, first when $S$ is a middle component, then for $S$ an ending component.

\begin{enumerate}
\item Let $S$ be a middle component.
By Lemma~\ref{compcon} it is a locally minimal network, moreover it has two entering points (if one, then it is an ending component) and one or two energetic points.

\begin{enumerate}
    
\item \textit{The case $n = 2$, $m = 2$, the combinatorial type (a) on Fig.~\ref{lmn4}  (see Fig.~\ref{pic_m_lmn4(a)}).  \label{m_lmn4(a)}}
\begin{figure}[h]
	\begin{center}
	\definecolor{qqzzqq}{rgb}{0.,0.6,0.}
		\definecolor{qqwuqq}{rgb}{0.,0.39215686274509803,0.}
		\definecolor{uuuuuu}{rgb}{0.26666666666666666,0.26666666666666666,0.26666666666666666}
		\definecolor{ffqqqq}{rgb}{1.,0.,0.}
		\definecolor{xdxdff}{rgb}{0.49019607843137253,0.49019607843137253,1.}
		\definecolor{qqqqff}{rgb}{0.,0.,1.}
		\definecolor{qqwuqq}{rgb}{0.,0.39215686274509803,0.}
\definecolor{qqqqff}{rgb}{0.,0.,1.}
\definecolor{uuuuuu}{rgb}{0.26666666666666666,0.26666666666666666,0.26666666666666666}
    		
\begin{tikzpicture}[line cap=round,line join=round,>=triangle 45,x=1.0cm,y=1.0cm]
\clip(-0.0765659999999999,-0.195068) rectangle (6.203433999999997,3.344932);
\draw [shift={(2.136254000000003,1.0019980000000008)},color=qqwuqq,fill=qqwuqq,fill opacity=0.1] (0,0) -- (48.927970089032904:0.3) arc (48.927970089032904:288.92797008903295:0.3) -- cycle;
\draw [shift={(2.2719753320461393,0.6062189960079492)},color=qqwuqq,fill=qqwuqq,fill opacity=0.1] (0,0) -- (108.92797008903294:0.3) arc (108.92797008903294:348.92797008903284:0.3) -- cycle;
\draw [shift={(0.,0.)},color=qqwuqq,fill=qqwuqq,fill opacity=0.1] (0,0) -- (0.8306750120280327:0.6) arc (0.8306750120280327:37.949878481816086:0.6) -- cycle;
\draw(0.,0.) circle (5.cm);
\draw(0.,0.) circle (2.cm);
\draw [line width=1.6pt, color=ffqqqq] (1.6754605210256992,1.0921685046201857)-- (2.136254000000003,1.0019980000000008);
\draw [line width=1.6pt, color=ffqqqq] (2.136254000000003,1.0019980000000008)-- (2.4447406871602975,1.3559716063798666);
\draw [line width=1.6pt, color=ffqqqq] (2.136254000000003,1.0019980000000008)-- (2.2719753320461393,0.6062189960079492);
\draw [line width=1.6pt, color=ffqqqq] (2.2719753320461393,0.6062189960079492)-- (1.981320542569633,0.2727066328338473);
\draw [line width=1.6pt, color=ffqqqq] (3.0784777466063318,0.4483983387795395)-- (2.2719753320461393,0.6062189960079492);

\draw [color=uuuuuu] (3.9427451188420477,3.074859497253038)-- (0.,0.);
\draw (0.,0.)-- (4.999474528181308,0.0724875304192808);
\draw [dash pattern=on 3pt off 3pt] (2.4447406871602975,1.3559716063798666)-- (3.9427451188420477,3.074859497253038);
\draw [dash pattern=on 3pt off 3pt] (3.0784777466063318,0.4483983387795395)-- (4.999474528181308,0.0724875304192808);

\draw [shift={(0.,0.)}, line width=1.6pt,color=qqzzqq]  plot[domain=0.0:0.6451180954624042,variable=\t]({1.*5.*cos(\t r)+0.*5.*sin(\t r)},{0.*5.*cos(\t r)+1.*5.*sin(\t r)});

\begin{scriptsize}
\draw [fill=uuuuuu] (0.,0.) circle (1.5pt);
\draw[color=uuuuuu] (0.14343399999999967,0.284932) node {$O$};
\draw[color=qqzzqq] (5.24343399999999967,0.284932) node {$Q_r$};
\draw[color=qqzzqq] (4.24343399999999967,3.184932) node {$Q_l$};
\draw[color=qqzzqq] (4.5 , 1.6) node {$q_S$};

\draw [fill=uuuuuu] (-0.07248753041928158,4.999474528181308) circle (1.5pt);
\draw [fill=ffqqqq] (2.4447406871602975,1.3559716063798666) circle (1.5pt);
\draw [fill=ffqqqq] (2.136254000000003,1.0019980000000008) circle (1.5pt);
\draw [fill=ffqqqq] (2.2719753320461393,0.6062189960079492) circle (1.5pt);
\draw [fill=ffqqqq] (1.981320542569633,0.2727066328338473) circle (1.5pt);
\draw [fill=ffqqqq] (1.6754605210256992,1.0921685046201857) circle (1.5pt);
\draw [fill=ffqqqq] (3.0784777466063318,0.4483983387795395) circle (1.5pt);
\draw [fill=qqzzqq] (3.9427451188420477,3.074859497253038) circle (1.5pt);
\draw [fill=qqzzqq] (4.999474528181308,0.0724875304192808) circle (1.5pt);

\draw[color=ffqqqq] (2.534339999999987, 0.704932) node {$V_r$};
\draw[color=ffqqqq] (2.424339999999987, 1.01932) node {$V_l$};

\draw[color=ffqqqq] (3.134339999999987, 0.624932) node {$S$};

\draw[color=qqwuqq] (2.0934339999999987,1.374932) node {$4\pi/3$};
\draw[color=qqwuqq] (2.6634339999999985,0.224932) node {$4\pi/3$};
\draw[color=qqwuqq] (0.7634339999999992,0.184932) node {$\gamma$};
\end{scriptsize}
\end{tikzpicture}
    \caption{Picture to the case~\ref{m_lmn4(a)}: middle component, $n=2, m=2$.}
    \label{pic_m_lmn4(a)}
    
	\end{center}
\end{figure}
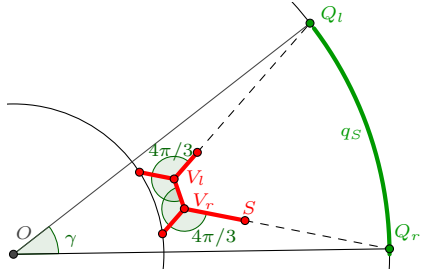   
Denote the Steiner points of $S$ by $V_l$ and $V_r$.
In this case $\turn (S)=\pi/3+\pi/3=2\pi/3$. Assuming 
the contrary (it means that $\gamma \geq 2\pi/3$) and connecting $O$ with $Q_l$ and $Q_r$, we get a (non convex) pentagon $Q_lV_lV_rQ_rO$ with two angles equal to $4 \pi/3$ and one angle at least $2\pi/3$, which is impossible.

\item \textit{The case $n = 2,  m = 2$, the combinatorial type  (b) on Fig.~\ref{lmn4}  (see Fig.~\ref{pic_m_lmn4(b)_1}).} \label{m_lmn4(b)}
\begin{figure}[h]

\begin{center}


\begin{minipage}[h]{0.49\textwidth}

\definecolor{qqwuqq}{rgb}{0.,0.39215686274509803,0.}
\definecolor{qqzzqq}{rgb}{0.,0.6,0.}
\definecolor{ffqqqq}{rgb}{1.,0.,0.}
\definecolor{ttqqqq}{rgb}{0.2,0.,0.}
\begin{tikzpicture}[x=1.0cm,y=1.0cm]
\clip(-5.18,-0.28) rectangle (0.6,5.96);
\draw [shift={(0.,0.)},color=qqwuqq,fill=qqwuqq,fill opacity=0.1] (0,0) -- (100.32253191296135:0.6) arc (100.32253191296135:140.50580788203823:0.6) -- cycle;
\draw [shift={(-1.9,3.18)},color=qqwuqq,fill=qqwuqq,fill opacity=0.1] (0,0) -- (60.:0.6) arc (60.:180.:0.6) -- cycle;
\draw [color=ttqqqq] (0.,0.) circle (5.cm);
\draw(0.,0.) circle (3.cm);
\draw (-1.9,3.18)-- (-2.42,3.18);
\draw (0.,0.)-- (-3.858445282753146,3.18);
\draw (-0.8959456002089823,4.919073234001113)-- (0.,0.);
\draw [shift={(0.,0.)},line width=2.pt,color=qqzzqq]  plot[domain=1.7509584958182614:2.452288965716167,variable=\t]({1.*5.*cos(\t r)+0.*5.*sin(\t r)},{0.*5.*cos(\t r)+1.*5.*sin(\t r)});
\draw [line width=2.pt,color=ffqqqq] (-1.9,3.18)-- (-1.64,3.630333209967908);
\draw [line width=2.pt,color=ffqqqq] (-1.9,3.18)-- (-1.64,2.729666790032092);
\draw [line width=2.pt,color=ffqqqq] (-1.64,2.729666790032092)-- (-1.8545007695124587,2.3581405589738917);
\draw [line width=2.pt,color=ffqqqq] (-1.64,2.729666790032092)-- (-1.2445558305660274,2.729666790032092);
\draw [line width=2.pt,color=ffqqqq] (-1.9,3.18)-- (-2.42,3.18);
\draw (-2.42,3.18)-- (-3.858445282753146,3.18);
\draw (-1.64,3.630333209967908)-- (-0.8959456002089823,4.919073234001113);
\begin{scriptsize}
\draw [fill=ttqqqq] (0.,0.) circle (1.5pt);
\draw[color=ttqqqq] (0.14,0.36) node {$O$};
\draw [fill=ffqqqq] (-2.42,3.18) circle (1.5pt);
\draw [fill=ffqqqq] (-1.9,3.18) circle (1.5pt);
\draw [fill=ffqqqq] (-1.64,2.729666790032092) circle (1.5pt);
\draw [fill=ffqqqq] (-1.64,3.630333209967908) circle (1.5pt);
\draw [fill=ffqqqq] (-1.8545007695124587,2.3581405589738917) circle (1.5pt);
\draw[color=ffqqqq] (-2.00,2.99) node {$B$};
\draw [fill=ffqqqq] (-1.2445558305660274,2.729666790032092) circle (1.5pt);
\draw[color=ffqqqq] (-1.46,2.96) node {$S$};
\draw [fill=qqzzqq] (-3.858445282753146,3.18) circle (1.5pt);
\draw[color=qqzzqq] (-3.98,3.66) node {$Q_l$};
\draw [fill=qqzzqq] (-0.8959456002089823,4.919073234001113) circle (1.5pt);
\draw[color=qqzzqq] (-0.76,5.2) node {$Q_r$};
\draw[color=qqzzqq] (-2.1,4.94) node {$q_S$};
\draw[color=qqwuqq] (-0.54,0.62) node {$\gamma$};
\draw[color=qqwuqq] (-2.08,3.52) node {$2\pi/3$};
\end{scriptsize}
\end{tikzpicture}
\caption{A general picture to the case~\ref{m_lmn4(b)}: middle component, $n = 2,  m = 2$.}
\label{pic_m_lmn4(b)_1}
\end{minipage}\begin{minipage}[h]{0.49\textwidth}
\definecolor{ffqqqq}{rgb}{1.,0.,0.}
\definecolor{wwzzqq}{rgb}{0.4,0.6,0.}
\definecolor{qqzzqq}{rgb}{0.,0.6,0.}
\definecolor{qqwuqq}{rgb}{0.,0.39215686274509803,0.}
\definecolor{xdxdff}{rgb}{0.49019607843137253,0.49019607843137253,1.}
\definecolor{uuuuuu}{rgb}{0.26666666666666666,0.26666666666666666,0.26666666666666666}
\begin{tikzpicture}[line cap=round,line join=round,>=triangle 45,x=1.0cm,y=1.0cm]
\clip(-2.8142709143324165, -0.4006007107986708) rectangle (3.2880494733233574,5.367340407067486);
\draw [shift={(-0.8694609743185158,2.8712432175169518)},color=qqwuqq,fill=qqwuqq,fill opacity=0.1] (0,0) -- (107.02423270507386:0.51568904684415) arc (107.02423270507386:137.02423270507384:0.51568904684415) -- cycle;
\draw[color=qqwuqq,fill=qqwuqq,fill opacity=0.1] (-0.5205949019919707,2.9773566478496094) -- (-0.6267083323246282,3.3262227201761543) -- (-0.9755744046511733,3.2201092898434966) -- (-0.8694609743185158,2.8712432175169518) -- cycle; 
\draw [shift={(0.,0.)},color=qqwuqq,fill=qqwuqq,fill opacity=0.1] (0,0) -- (53.787877677846886:0.51568904684415) arc (53.787877677846886:119.47025934561455:0.51568904684415) -- cycle;
\draw(0.,0.) circle (5.cm);
\draw(0.,0.) circle (3.cm);
\draw (0.,0.)-- (-1.45501492851883,4.783610723897519);
\draw (0.,0.)-- (2.953881932857728,4.034176685116393);
\draw (-2.4598585810853084,4.353055910628873)-- (-1.45501492851883,4.783610723897519);
\draw (2.953881932857728,4.034176685116393)-- (-1.45501492851883,4.783610723897519);
\draw (-2.4598585810853084,4.353055910628873)-- (-0.8694609743185158,2.8712432175169518);
\draw (-0.8694609743185158,2.8712432175169518)-- (2.953881932857728,4.034176685116393);
\draw [line width=1.6pt,color=ffqqqq] (1.0404376432987206,3.452170713708861)-- (-0.8694609743185158,2.8712432175169518);    
\draw [line width=1.6pt,color=ffqqqq] (-0.8694609743185158,2.8712432175169518)-- (-1.14065744192462878,3.1296779506665316);  
\draw (-2.4598585810853084,4.353055910628873)-- (0.,0.);
\draw [shift={(0.,0.)},line width=1.6pt,color=qqzzqq]  plot[domain=0.93877556313839:2.0851493837925004,variable=\t]({1.*5.*cos(\t r)+0.*5.*sin(\t r)},{0.*5.*cos(\t r)+1.*5.*sin(\t r)});

\begin{scriptsize}
\draw [fill=uuuuuu] (0., 0.) circle (1.5pt);
\draw [fill=ffqqqq] (1.0404376432987206,3.452170713708861) circle (1.5pt);
\draw [fill=ffqqqq] (-0.8694609743185158,2.8712432175169518) circle (1.5pt);
\draw[color=uuuuuu] (0.12515665267923792, -0.13541578030911325) node {$O$};
\draw [fill=ffqqqq] (-1.14065744192462878,3.1296779506665316) circle (1.5pt);
\draw[color=ffqqqq] (-0.7015147269558169,2.724843347320764) node {$B$};
\draw [fill=qqzzqq] (-1.45501492851883,4.783610723897519) circle (1.5pt);
\draw [fill=qqzzqq] (2.953881932857728,4.034176685116393) circle (1.5pt);
\draw[color=qqzzqq] (3.0817738545856974,4.274979980588283) node {$Q_r$};
\draw [fill=uuuuuu] (-4.6998985135451825,1.7061517993355064) circle (1.5pt);
\draw[color=qqwuqq] (-1.3824970597467216,3.6098775056376436) node {$\pi/6$};
\draw [fill=qqzzqq] (-2.4598585810853084,4.353055910628873) circle (1.5pt);
\draw[color=qqzzqq] (-2.3329611372778767,4.651583043589577) node {$Q_l$};
\draw[color=ffqqqq] (0.05639811310001794,3.3702049499012785) node {$S$};
\draw [fill=uuuuuu] (-2.0323944419179574,6.6945861246931955) circle (1.5pt);
\draw[color=uuuuuu] (-2.7111331049635865,5.8048574862259255) node {$D'$};
\draw[color=qqwuqq] (-0.63268787074283706,3.5122921098454447) node {$\pi/2$};
\draw[color=qqzzqq] (0.43457008078572784,5.237599534697361) node {$q_S$};
\draw[color=qqzzqq] (0.0267938739974128,0.3557432245727481) node {$\gamma$};
\end{scriptsize}
\end{tikzpicture}
\caption{A marginal picture to the case~\ref{m_lmn4(b)}: middle component, $n = 2,  m = 2$.}
      \label{pic_m_lmn4(b)_2}
    \end{minipage}

\end{center}


\end{figure}      
Note that in this case there exists a Steiner point adjacent to both entering points, and also there exists a Steiner point (we call it $B$) adjacent to both energetic points.
Clearly $\turn (S) = \pi/3$. Let us prove that $\turn(q_S) < \pi/3$.
We evaluate the arc of $M$ bounded by continuations of segments starting from $B$. 
Clearly this arc is maximal when $B$ belongs to $M_r$ (it is the marginal case). 
Hence it is enough to look at the angle in $N \setminus N_r$ of size $2\pi/3$ with vertex $B$ on $M_r$. 
It is well-known that the arc is maximal when $S$ is tangent to $M_r$ and when $M$ is a circumference. 
In this case the normal to $M_r$ at $B$ splits the angle $\angle Q_lBQ_r = 2\pi/3$ in 
two angles: one of size $\pi/2$ and another of size $\pi/6$ (see Fig.~\ref{pic_m_lmn4(b)_2}), so that the size of the arc is 
    \[
    \arccos\left(1-\frac{1}{\delta}\right) + \frac{\pi}{6} - \arcsin\left(\frac{1}{2}\left(1-\frac{1}{\delta}\right)\right),
    \]
where $\delta \defeq R/r$, hence it is strictly less than $\pi/3$ for $\delta \geq 2.9$.

\textit{\item The case $n = 2$, $m = 2$, the combinatorial type (c) on Fig.~\ref{lmn4}.  \label{m_lmn4(c)}}
\begin{figure}
\begin{center}

\definecolor{ttqqqq}{rgb}{0.2,0.,0.}
\definecolor{ttzzqq}{rgb}{0.2,0.6,0.}
\definecolor{qqwuqq}{rgb}{0.,0.39215686274509803,0.}
\definecolor{uuuuuu}{rgb}{0.26666666666666666,0.26666666666666666,0.26666666666666666}
\definecolor{ffqqqq}{rgb}{1.,0.,0.}
\definecolor{xdxdff}{rgb}{0.49019607843137253,0.49019607843137253,1.}
\definecolor{qqqqff}{rgb}{0.,0.,1.}
\begin{tikzpicture}[line cap=round,line join=round,>=triangle 45,x=3.0cm,y=3.0cm]
\clip(7.684817300548393,6.3690559354760365) rectangle (11.61161342494308,9.836467256856569);
\draw [shift={(0.8342481486993167,1.6425662766580746)},color=qqwuqq,fill=qqwuqq,fill opacity=0.1] (0,0) -- (109.1027822100106:0.18878827521128302) arc (109.1027822100106:146.04369252437712:0.18878827521128302) -- cycle;
\draw [shift={(-0.2550202615275577,3.398667261271955)},color=qqwuqq,fill=qqwuqq,fill opacity=0.1] (0,0) -- (-95.68993774918141:0.18878827521128302) arc (-95.68993774918141:-29.28109573597084:0.18878827521128302) -- cycle;
\draw [shift={(-0.2550202615275577,3.398667261271955)},color=qqqqff,fill=qqqqff,fill opacity=0.1] (0,0) -- (-149.28109573597084:0.18878827521128302) arc (-149.28109573597084:-95.68993774918141:0.18878827521128302) -- cycle;
\draw [shift={(-0.669513117927742,3.152374404570396)},color=qqqqff,fill=qqqqff,fill opacity=0.1] (0,0) -- (-72.29809973885571:0.12585885014085535) arc (-72.29809973885571:-9.56241946371832:0.12585885014085535) -- cycle;
\draw [shift={(-0.669513117927742,3.152374404570396)},color=qqwuqq,fill=qqwuqq,fill opacity=0.1] (0,0) -- (-9.562419463718317:0.12585885014085535) arc (-9.562419463718317:30.718904264029195:0.12585885014085535) -- cycle;
\draw [shift={(11.060956949431324,6.464360883202412)},color=qqwuqq,fill=qqwuqq,fill opacity=0.1] (0,0) -- (109.32343564751898:0.18878827521128302) arc (109.32343564751898:147.03128937897878:0.18878827521128302) -- cycle;
\draw [shift={(9.623043592717565,7.992635102293055)},color=qqwuqq,fill=qqwuqq,fill opacity=0.1] (0,0) -- (-97.37744328017351:0.18878827521128302) arc (-97.37744328017351:-5.519660831925377:0.18878827521128302) -- cycle;
\draw [shift={(9.623043592717565,7.992635102293055)},color=qqwuqq,fill=qqwuqq,fill opacity=0.1] (0,0) -- (-5.519660831925366:0.18878827521128302) arc (-5.519660831925366:40.774484146288046:0.18878827521128302) -- cycle;
\draw [shift={(9.980597144254066,8.300989285138138)},color=qqwuqq,fill=qqwuqq,fill opacity=0.1] (0,0) -- (-139.22551585371195:0.18878827521128302) arc (-139.22551585371195:-95.03788695330876:0.18878827521128302) -- cycle;
\draw [shift={(9.95062355741404,7.960979285495686)},color=qqwuqq,fill=qqwuqq,fill opacity=0.1] (0,0) -- (84.96211304669126:0.12585885014085535) arc (84.96211304669126:174.48033916807464:0.12585885014085535) -- cycle;
\draw(0.8342481486993167,1.6425662766580746) circle (6.cm);
\draw(0.8342481486993167,1.6425662766580746) circle (9.cm);
\draw [dash pattern=on 2pt off 2pt,domain=-0.669513117927742:11.61161342494308] plot(\x,{(--0.0476156992637804-0.1414384752429001*\x)/0.045143944073481324});
\draw [line width=2pt,color=ffqqqq] (-0.669513117927742,3.152374404570396)-- (-0.2550202615275577,3.398667261271955);
\draw [domain=7.684817300548393:0.8342481486993167] plot(\x,{(-5.485281875386651--1.6756816043594478*\x)/-2.488391279684798});
\draw [dash pattern=on 2pt off 2pt,domain=-0.2550202615275577:11.61161342494308] plot(\x,{(--1.3691457618740381-0.23581491497895568*\x)/0.4205422988742826});
\draw [domain=7.684817300548393:0.8342481486993167] plot(\x,{(-2.6517221587893767--1.8898660435188797*\x)/-0.6545275682156528});
\draw [dotted] (-0.1475432036241624,4.477365341936394) circle (3.cm);
\draw [dash pattern=on 2pt off 2pt,domain=-1.6541431309854813:11.61161342494308] plot(\x,{(--2.9928679827374025-0.16587347644712658*\x)/0.9846300130577392});
\draw [dash pattern=on 2pt off 2pt,domain=7.684817300548393:-0.1475432036241624] plot(\x,{(-0.6403686246744196-1.0786980806644388*\x)/-0.10747705790339529});
\draw [dotted] (-1.6541431309854813,3.3182478810175224) circle (3.cm);
\draw [shift={(0.8342481486993167,1.6425662766580746)},line width=2pt,color=ttzzqq]  plot[domain=1.9042027726509805:2.5489432863206094,variable=\t]({1.*3.*cos(\t r)+0.*3.*sin(\t r)},{0.*3.*cos(\t r)+1.*3.*sin(\t r)});
\draw [line width=2pt,color=ffqqqq] (-0.6243691738542607,3.0109359293274958)-- (-0.669513117927742,3.152374404570396);
\draw [line width=2pt,color=ffqqqq] (-0.2550202615275577,3.398667261271955)-- (-0.24668820076800613,3.4822923448445473);
\draw [line width=2pt,color=ffqqqq] (-0.2550202615275577,3.398667261271955)-- (-0.18842735425026425,3.3613259502793817);
\draw(11.060956949431324,6.464360883202412) circle (6.cm);
\draw(11.060956949431324,6.464360883202412) circle (9.cm);
\draw [line width=2pt,color=ffqqqq] (9.623043592717565,7.992635102293055)-- (9.980597144254066,8.300989285138138);
\draw [dotted] (10.068255727791723,9.295357930985269) circle (3.cm);
\draw [dotted] (8.544053334295505,8.096903743931112) circle (3.cm);
\draw [shift={(11.060956949431324,6.464360883202412)},line width=2pt,color=ttzzqq]  plot[domain=1.9080539016413451:2.566180103115748,variable=\t]({1.*3.*cos(\t r)+0.*3.*sin(\t r)},{0.*3.*cos(\t r)+1.*3.*sin(\t r)});
\draw [line width=2pt,color=ffqqqq] (9.602339626877747,7.832730535871833)-- (9.623043592717565,7.992635102293055);
\draw [line width=2pt,color=ffqqqq] (9.980597144254066,8.300989285138138)-- (10.146585211676499,8.243103226881018);
\draw [color=ttqqqq] (8.544053334295505,8.096903743931112)-- (11.060956949431324,6.464360883202412);
\draw (11.060956949431324,6.464360883202412)-- (10.068255727791723,9.295357930985269);
\draw [line width=2.pt,color=ffqqqq] (9.623043592717565,7.992635102293055)-- (9.539416584949084,8.000716430586989);
\draw [dash pattern=on 2pt off 2pt,domain=7.684817300548393:11.61161342494308] plot(\x,{(--1.3732893692180141-0.15990456642122197*\x)/-0.02070396583981804});
\draw (10.477727340379703,8.12762195855155)-- (10.146585211676499,8.243103226881018);
\draw (10.068255727791723,9.295357930985269)-- (9.95062355741404,7.960979285495686);
\draw (9.95062355741404,7.960979285495686)-- (8.544053334295505,8.096903743931112);
\begin{scriptsize}

\draw [fill=uuuuuu] (11.060956949431324,6.464360883202412) circle (1.5pt);
\draw[color=uuuuuu] (11.114470966886701,6.595601865729574) node {$O$};
\draw [fill=ffqqqq] (9.623043592717565,7.992635102293055) circle (1.5pt);
\draw[color=ffqqqq] (9.697094190266864,8.137029239898694) node {$V_l$};
\draw [fill=ffqqqq] (9.602339626877747,7.832730535871833) circle (1.5pt);
\draw[color=ffqqqq] (9.711922420238695,7.756089172039518) node {$Y_l$};
\draw [fill=ffqqqq] (9.980597144254066,8.300989285138138) circle (1.5pt);
\draw[color=ffqqqq] (9.880597144254066,8.400989285138138) node {$W_r$};

\draw [fill=ttzzqq] (8.544053334295505,8.096903743931112) circle (1.5pt);
\draw[color=ttzzqq] (8.521778653985082,8.250645745081806) node {$Q_l$};
\draw [fill=ffqqqq] (10.146585211676499,8.243103226881018) circle (1.5pt);
\draw[color=ffqqqq] (10.2019943033655,8.16254454998321) node {$Y_r$};

\draw [fill=ttzzqq] (10.068255727791723,9.295357930985269) circle (1.5pt);
\draw[color=ttzzqq] (10.120186050773944,9.440011878912879) node {$Q_r$};
\draw [fill=uuuuuu] (9.95062355741404,7.960979285495686) circle (1.5pt);
\draw[color=uuuuuu] (10.00062014314013,8.011513929814182) node {$L$};
\draw [fill=ffqqqq] (9.539416584949084,8.000716430586989) circle (1.5pt);
\draw[color=ffqqqq] (9.42796237499924,8.124786894940954) node {$W_l$};
\end{scriptsize}
\end{tikzpicture}

\caption{Picture to the case~\ref{m_lmn4(c)}: middle component, $n = 2$, $m = 2$.}
\label{pic_m_lmn4(c)_1}

\end{center}
\end{figure}  
\begin{figure}
\begin{center}

\definecolor{ffqqqq}{rgb}{1.,0.,0.}
\definecolor{uuuuuu}{rgb}{0.26666666666666666,0.26666666666666666,0.26666666666666666}
\definecolor{qqwuqq}{rgb}{0.,0.39215686274509803,0.}
\definecolor{ttqqqq}{rgb}{0.2,0.,0.}
\definecolor{ttzzqq}{rgb}{0.2,0.6,0.}
\definecolor{qqqqff}{rgb}{0.,0.,1.}
\begin{tikzpicture}[line cap=round,line join=round,>=triangle 45,x=3.0cm,y=3.0cm]
\clip(7.442853729287281,7.795712893980238) rectangle (9.9430849417728,9.804267805945461);
\draw [shift={(10.434814367731908,6.8604919042775485)},color=qqwuqq,fill=qqwuqq,fill opacity=0.1] (0,0) -- (108.80796898250493:0.3833526010404598) arc (108.80796898250493:147.76435727299193:0.3833526010404598) -- cycle;
\draw [shift={(9.0754913785471,8.780395769358007)},color=qqwuqq,fill=qqwuqq,fill opacity=0.1] (0,0) -- (-164.8194938929899:0.1233526010404598) arc (-164.8194938929899:-113.08715652800417:0.1233526010404598) -- cycle;
\draw [shift={(9.0754913785471,8.780395769358007)},color=qqwuqq,fill=qqwuqq,fill opacity=0.1] (0,0) -- (-113.08715652800417:0.1233526010404598) arc (-113.08715652800417:-53.23122241346264:0.1233526010404598) -- cycle;
\draw(10.434814367731908,6.8604919042775485) circle (6.cm);
\draw(10.434814367731908,6.8604919042775485) circle (9.cm);
\draw [dotted] (9.46762229730479,9.700305190200218) circle (3.cm);
\draw [dotted] (7.897229840206437,8.460699630062964) circle (3.cm);
\draw [shift={(10.434814367731908,6.8604919042775485)},line width=1.2pt,color=ttzzqq]  plot[domain=1.8990573111525755:2.5789745515069384,variable=\t]({1.*3.*cos(\t r)+0.*3.*sin(\t r)},{0.*3.*cos(\t r)+1.*3.*sin(\t r)});
\draw [color=ttqqqq] (7.897229840206437,8.460699630062964)-- (10.434814367731908,6.8604919042775485);
\draw (10.434814367731908,6.8604919042775485)-- (9.46762229730479,9.700305190200218);
\draw [dash pattern=on 2pt off 2pt] (9.0754913785471,8.780395769358007)-- (7.897229840206437,8.460699630062964);
\draw [dash pattern=on 2pt off 2pt] (8.718460820019688,7.942829133810776)-- (9.46762229730479,9.700305190200218);
\draw [line width=1.2pt,color=ffqqqq] (8.862335484133046,8.722560463447799)-- (9.0754913785471,8.780395769358007);
\draw [line width=1.2pt,color=ffqqqq] (9.135886835917882,8.381276986758156)-- (9.222664938413201,8.583441144238197);
\draw [line width=1.2pt,color=ffqqqq] (9.470952753540628,8.612910458358897)-- (9.222664938413201,8.583441144238197);
\draw [line width=1.2pt,color=ffqqqq] (9.222664938413201,8.583441144238197)-- (9.0754913785471,8.780395769358007);
\begin{scriptsize}
\draw [fill=qqqqff] (10.434814367731908,6.8604919042775485) circle (2.5pt);
\draw[color=qqqqff] (10.525266059553152,7.087974936130511) node {$O$};

\draw [fill=ttzzqq] (7.897229840206437,8.460699630062964) circle (1.5pt);
\draw[color=ttzzqq] (7.750121551749703,8.608606920257646) node {$Q_l$};

\draw [fill=ttzzqq] (9.46762229730479,9.700305190200218) circle (1.5pt);
\draw[color=ttzzqq] (9.586884556952003,9.588676284003281) node {$Q_r$};

\draw [fill=ffqqqq] (9.0754913785471,8.780395769358007) circle (1.5pt);
\draw[color=ffqqqq] (9.207424056084952,8.817516361367464) node {$W_r$};

\draw [fill=ffqqqq] (8.862335484133046,8.722560463447799) circle (1.5pt);
\draw[color=ffqqqq] (8.668503054350852,8.779181101263418) node {$W_l$};

\draw [fill=uuuuuu] (8.718460820019688,7.942829133810776) circle (1.5pt);
\draw[color=uuuuuu] (8.87129303500981,8.046356438731646) node {$K$};

\draw [fill=ffqqqq] (9.222664938413201,8.583441144238197) circle (1.5pt);
\draw[color=ffqqqq] (9.151316156258362,8.281947960673828) node {$Y_l$};

\draw [fill=ffqqqq] (9.135886835917882,8.381276986758156) circle (1.5pt);

\draw [fill=ffqqqq] (9.470952753540628,8.612910458358897) circle (1.5pt);
\draw[color=ffqqqq] (9.454106136917321,8.507504800743191) node {$Y_r$};

\draw[color=qqwuqq] (8.929605173865305,8.680822719910828) node {$\alpha$};
\draw[color=qqwuqq] (9.130776657125411,8.62969747914783) node {$\alpha$};
\end{scriptsize}
\end{tikzpicture}

\caption{Picture to the case~\ref{m_lmn4(c)}: middle component, $n = 2$, $m = 2$.}
\label{pic_m_lmn4(c)_2}

\end{center}
\end{figure}  

There are two possibilities for $S$ in this case, see Fig.~\ref{pic_m_lmn4(c)_1} and Fig.~\ref{pic_m_lmn4(c)_2}. 

{\sc The case on Fig.~\ref{pic_m_lmn4(c)_2}} can be reduced to the previous case~\ref{m_lmn4(b)}. Obviously, $\turn (S) = \pi/3$. Let us fix the entering points $Y_l$, $Y_r$ and the left energetic point $W_l$ and move the right energetic point $W_r$ to the right (in the direction of the ray $[W_lW_r)$). Then at some time the combinatorial type changes to (b) on Fig.~\ref{lmn4}, during this process
 $\turn (S) = \pi/3$, and $\turn (q_S)$ grows, but $\turn (q_S) \leq \pi/3$. By case~\ref{m_lmn4(b)}. 
 
{\sc The case on Fig.~\ref{pic_m_lmn4(c)_1}}:
denote the energetic points of $S$ by $W_l$ and $W_r$, and the entering points by $Y_l$, $Y_r$ respectively, and the branching point by $V_l$ (without loss of generality it is connected with $W_l$ and $Y_l$). Let $2\beta := \angle V_lW_rY_r$, and note that $\angle Y_lV_lW_r = 2\pi/3$. 
Then $\turn (S) = (\pi - 2\pi/3) + (\pi - 2\beta) = 4\pi/3 - 2\beta$. Assume the contrary (i.e.\ in this case $\gamma \geq 4\pi/3 - 2\beta$) and 
call $L$ the point of intersection of $(Q_lW_l)$ and $(Q_rW_r)$. By Lemma~\ref{tech} $\angle LW_rV_l = \angle Y_rW_rV_l/2 = \beta$.
Then $$\pi - \pi/3 - \beta = \angle Q_lLQ_r > \angle Q_lOQ_r = \gamma,$$ 
(the first equality coming from $\Delta V_lW_rL$)
which implies $$\gamma \geq  4\pi/3 - 2\beta > 2\pi/3 - \beta = \angle Q_lLQ_r > \gamma,$$ a contradiction.

\textit{\item The case $n = 2$, $m = 2$, the combinatorial type (d) on Fig.~\ref{lmn4} (see Fig.~\ref{pic_m_lmn4(d)}).  \label{m_lmn4(d)}}
\begin{figure}
\begin{center}

\definecolor{ttzzqq}{rgb}{0.2,0.6,0.}
\definecolor{qqwuqq}{rgb}{0.,0.39215686274509803,0.}
\definecolor{uuuuuu}{rgb}{0.26666666666666666,0.26666666666666666,0.26666666666666666}
\definecolor{ffqqqq}{rgb}{1.,0.,0.}
\definecolor{xdxdff}{rgb}{0.49019607843137253,0.49019607843137253,1.}
\definecolor{qqqqff}{rgb}{0.,0.,1.}
\begin{tikzpicture}[line cap=round,line join=round,>=triangle 45,x=2.5cm,y=2.5cm]
\clip(7.954817300548392,6.369055935476034) rectangle (11.411613424943078,9.836467256856569);
\draw [shift={(0.8342481486993167,1.6425662766580746)},color=qqwuqq,fill=qqwuqq,fill opacity=0.1] (0,0) -- (109.1027822100106:0.18878827521128297) arc (109.1027822100106:146.04369252437712:0.18878827521128297) -- cycle;
\draw [shift={(-0.2550202615275577,3.398667261271955)},color=qqwuqq,fill=qqwuqq,fill opacity=0.1] (0,0) -- (-95.68993774918141:0.18878827521128297) arc (-95.68993774918141:-29.28109573597084:0.18878827521128297) -- cycle;
\draw [shift={(-0.2550202615275577,3.398667261271955)},color=qqqqff,fill=qqqqff,fill opacity=0.1] (0,0) -- (-149.28109573597084:0.18878827521128297) arc (-149.28109573597084:-95.68993774918141:0.18878827521128297) -- cycle;
\draw [shift={(-0.669513117927742,3.152374404570396)},color=qqqqff,fill=qqqqff,fill opacity=0.1] (0,0) -- (-72.29809973885571:0.12585885014085532) arc (-72.29809973885571:-9.56241946371832:0.12585885014085532) -- cycle;
\draw [shift={(-0.669513117927742,3.152374404570396)},color=qqwuqq,fill=qqwuqq,fill opacity=0.1] (0,0) -- (-9.562419463718317:0.12585885014085532) arc (-9.562419463718317:30.718904264029195:0.12585885014085532) -- cycle;
\draw [shift={(11.060956949431324,6.464360883202412)},color=qqwuqq,fill=qqwuqq,fill opacity=0.1] (0,0) -- (109.10278221001069:0.18878827521128297) arc (109.10278221001069:146.04369252437712:0.18878827521128297) -- cycle;
\draw [shift={(9.557195682804265,7.97416901111473)},color=qqwuqq,fill=qqwuqq,fill opacity=0.1] (0,0) -- (-72.29809973885514:0.18878827521128297) arc (-72.29809973885514:-9.562419463718436:0.18878827521128297) -- cycle;
\draw [shift={(9.557195682804265,7.97416901111473)},color=qqwuqq,fill=qqwuqq,fill opacity=0.1] (0,0) -- (-9.56241946371844:0.18878827521128297) arc (-9.56241946371844:37.664269383202026:0.18878827521128297) -- cycle;
\draw [shift={(9.980597144254066,8.300989285138138)},color=qqwuqq,fill=qqwuqq,fill opacity=0.1] (0,0) -- (-142.33573061679797:0.18878827521128297) arc (-142.33573061679797:-95.63962260795269:0.18878827521128297) -- cycle;
\draw [shift={(9.980597144254066,8.300989285138138)},color=qqwuqq,fill=qqwuqq,fill opacity=0.1] (0,0) -- (-95.6396226079527:0.18878827521128297) arc (-95.6396226079527:-22.335730616798642:0.18878827521128297) -- cycle;
\draw(0.8342481486993167,1.6425662766580746) circle (5.cm);
\draw(0.8342481486993167,1.6425662766580746) circle (7.5cm);
\draw [dash pattern=on 2pt off 2pt,domain=-0.669513117927742:11.611613424943078] plot(\x,{(--0.0476156992637804-0.1414384752429001*\x)/0.045143944073481324});
\draw [line width=2pt,color=ffqqqq] (-0.669513117927742,3.152374404570396)-- (-0.2550202615275577,3.398667261271955);
\draw [domain=7.684817300548392:0.8342481486993167] plot(\x,{(-5.485281875386651--1.6756816043594478*\x)/-2.488391279684798});
\draw [dash pattern=on 2pt off 2pt,domain=-0.2550202615275577:11.611613424943078] plot(\x,{(--1.3691457618740381-0.23581491497895568*\x)/0.4205422988742826});
\draw [domain=7.684817300548392:0.8342481486993167] plot(\x,{(-2.6517221587893767--1.8898660435188797*\x)/-0.6545275682156528});
\draw [dotted] (-0.1475432036241624,4.477365341936394) circle (1.cm);
\draw [dash pattern=on 2pt off 2pt,domain=-1.6541431309854813:11.611613424943078] plot(\x,{(--2.9928679827374025-0.16587347644712658*\x)/0.9846300130577392});
\draw [dash pattern=on 2pt off 2pt,domain=7.684817300548392:-0.1475432036241624] plot(\x,{(-0.6403686246744196-1.0786980806644388*\x)/-0.10747705790339529});
\draw [dotted] (-1.6541431309854813,3.3182478810175224) circle (1.cm);
\draw [shift={(0.8342481486993167,1.6425662766580746)},line width=2pt,color=ttzzqq]  plot[domain=1.9042027726509805:2.5489432863206094,variable=\t]({1.*3.*cos(\t r)+0.*3.*sin(\t r)},{0.*3.*cos(\t r)+1.*3.*sin(\t r)});
\draw [line width=2pt,color=ffqqqq] (-0.6243691738542607,3.0109359293274958)-- (-0.669513117927742,3.152374404570396);
\draw [line width=2pt,color=ffqqqq] (-0.2550202615275577,3.398667261271955)-- (-0.24668820076800613,3.4822923448445473);
\draw [line width=2pt,color=ffqqqq] (-0.2550202615275577,3.398667261271955)-- (-0.18842735425026425,3.3613259502793817);
\draw(11.060956949431324,6.464360883202412) circle (5.cm);
\draw(11.060956949431324,6.464360883202412) circle (7.5cm);
\draw [line width=2pt,color=ffqqqq] (9.557195682804265,7.97416901111473)-- (9.980597144254066,8.300989285138138);
\draw [dotted] (10.079165597107842,9.299159948480728) circle (2.5cm);
\draw [dotted] (8.572565669746526,8.14004248756186) circle (2.5cm);
\draw [shift={(11.060956949431324,6.464360883202412)},line width=2pt,color=ttzzqq]  plot[domain=1.9042027726509818:2.548943286320609,variable=\t]({1.*3.*cos(\t r)+0.*3.*sin(\t r)},{0.*3.*cos(\t r)+1.*3.*sin(\t r)});
\draw [line width=2pt,color=ffqqqq] (9.602339626877747,7.832730535871833)-- (9.557195682804265,7.97416901111473);
\draw [line width=2pt,color=ffqqqq] (9.980597144254066,8.300989285138138)-- (10.135484038520287,8.23735267426347);
\draw (8.572565669746526,8.14004248756186)-- (11.060956949431324,6.464360883202412);
\draw (11.060956949431324,6.464360883202412)-- (10.079165597107842,9.299159948480728);
\draw [domain=8.572565669746526:11.611613424943078] plot(\x,{(--9.436891410530762-0.16587347644713013*\x)/0.984630013057739});
\draw [domain=7.684817300548392:10.079165597107842] plot(\x,{(--9.14412360104339-0.9981706633425897*\x)/-0.09856845285377602});
\begin{scriptsize}

\draw [fill=uuuuuu] (11.060956949431324,6.464360883202412) circle (1.5pt);
\draw[color=uuuuuu] (11.114470966886701,6.595601865729572) node {$O$};
\draw [fill=ffqqqq] (9.557195682804265,7.97416901111473) circle (1.5pt);
\draw[color=ffqqqq] (9.456750650210523,8.087029239898696) node {$W_l$};
\draw [fill=ffqqqq] (9.602339626877747,7.832730535871833) circle (1.5pt);
\draw[color=ffqqqq] (9.660801247759823,7.717463332264884) node {$Y_l$};
\draw [fill=ffqqqq] (9.980597144254066,8.300989285138138) circle (1.5pt);
\draw[color=ffqqqq] (9.874074212886593,8.42055519277196) node {$W_r$};
\draw [fill=ttzzqq] (8.572565669746526,8.14004248756186) circle (1.5pt);
\draw[color=ttzzqq] (8.546950424013254,8.288403400124063) node {$Q_1$};
\draw [fill=ffqqqq] (10.135484038520287,8.23735267426347) circle (1.5pt);
\draw[color=ffqqqq] (10.189408418351414,8.370211652715618) node {$Y_r$};
\draw [fill=ttzzqq] (10.079165597107842,9.299159948480728) circle (1.5pt);
\draw[color=ttzzqq] (10.126478993280987,9.446304821419922) node {$Q_2$};
\draw [fill=uuuuuu] (9.941923775683486,7.909356659869646) circle (1.5pt);
\draw[color=uuuuuu] (9.988034258126046,7.998928044800097) node {$L$};
\draw[color=qqwuqq] (10.815158046999385,6.776723038208444) node {$\gamma$};

\draw[color=qqwuqq] (9.641831867928847,7.920139769588816) node {$\alpha$};
\draw[color=qqwuqq] (9.68139777556266,8.006685699842353) node {$\alpha$};

\draw[color=qqwuqq] (9.90844837311196,8.195473975053635) node {$\beta$};
\draw[color=qqwuqq] (10.011560210999312, 8.208302205025466) node {$\beta$};
\end{scriptsize}
\end{tikzpicture}

\caption{Picture to the case~\ref{m_lmn4(d)}: middle component, $n = 2$, $m = 2$.}
\label{pic_m_lmn4(d)}

\end{center}

\end{figure}  
Denote the energetic points of $S$ by $W_l$ and $W_r$, and the entering points by $Y_l$, $Y_r$ respectively. Let $2\alpha := \angle Y_lW_lW_r$, $2\beta := \angle W_lW_rY_r$.
Then $\turn (S) = (\pi - 2\alpha) + (\pi - 2\beta)$. Assume the contrary (it means that $\gamma \geq 2\pi - 2\alpha - 2\beta$) and 
denote by $L$ the point of intersection of $(Q_lW_l)$ and $(Q_rW_r)$. 
By Lemma~\ref{tech} $\angle LW_lW_r = \angle Y_lW_lW_r/2 = \alpha$, $\angle LW_rW_l = \angle Y_rW_rW_l/2 = \beta$.  
Then $$\pi - \alpha - \beta = \angle Q_lLQ_r > \angle Q_lOQ_r = \gamma,$$ 
(the first equality coming from $\Delta W_lW_rL$) which implies 
$$\gamma \geq  2\pi - 2\alpha - 2\beta > \pi - \alpha - \beta = \angle Q_lLQ_r > \gamma,$$ 
a contradiction.

\item \textit{The case $n = 1$, $m = 2$, the combinatorial type (b) on Fig.~\ref{lmn23} (see Fig.~\ref{pic_m_lmn23(b)}, Fig.~\ref{pic_m_lmn23(b)_2},  Fig.~\ref{pic_m_lmn23(b)_3}). \label{m_lmn23(b)}}
\begin{figure}
\begin{center}
    
\definecolor{qqwuqq}{rgb}{0.,0.39215686274509803,0.}
    \definecolor{qqzzqq}{rgb}{0.,0.6,0.}
    \definecolor{ttttff}{rgb}{0.2,0.2,1.}
    \definecolor{uuuuuu}{rgb}{0.26666666666666666,0.26666666666666666,0.26666666666666666}
    \definecolor{ffqqqq}{rgb}{1.,0.,0.}
    \definecolor{qqqqff}{rgb}{0.,0.,1.}
    \definecolor{xdxdff}{rgb}{0.49019607843137253,0.49019607843137253,1.}
    \begin{tikzpicture}[line cap=round,line join=round,>=triangle 45,x=1.0cm,y=1.0cm]
    \clip(-3.7,-0.22) rectangle (0.86,5.4);
    \draw [shift={(0.,0.)},color=qqwuqq,fill=qqwuqq,fill opacity=0.1] (0,0) -- (87.36226398798078:0.6) arc (87.36226398798078:109.3186313943467:0.6) -- cycle;
    \draw [shift={(-1.2990482868580284,3.705636253936034)},color=qqwuqq,fill=qqwuqq,fill opacity=0.1] (0,0) -- (-70.68136860565333:0.6) arc (-70.68136860565333:40.13067521228883:0.6) -- cycle;
    \draw(0.,0.) circle (5.cm);
    \draw(0.,0.) circle (3.cm);
    
    \draw [line width=1.6pt, color=ffqqqq] (-1.6888952863865674,2.4794420161845356)-- (-1.14,2.98);
    \draw [line width=1.6pt, color=ffqqqq] (-1.2990482868580284,3.705636253936034)-- (-1.14,2.98);

    \draw [dash pattern=on 3pt off 3pt] (-1.2990482868580284,3.705636253936034) circle (2.cm);
    \draw (-3.298368942978354,3.7577602792083287)-- (0.,0.);
    
   \draw (0.23010458949184187,4.994702381313104)-- (0.,0.);
    \draw (-1.2990482868580284,3.705636253936034)-- (0.,0.);
    \draw  (-1.2990482868580284,3.705636253936034)-- (0.23010458949184187,4.994702381313104);
    \draw [line width=1.6pt, color=ffqqqq] (-0.8324030180784497,2.8822049225365443)-- (-1.14,2.98);
    
    \draw (-3.298368942978354,3.7577602792083287)-- (-1.2990482868580284,3.705636253936034);
    \draw [shift={(0.,0.)}, line width=1.6pt,color=qqzzqq]  plot[domain=1.5247591485867362:2.291180954624042,variable=\t]({1.*5.*cos(\t r)+0.*5.*sin(\t r)},{0.*5.*cos(\t r)+1.*5.*sin(\t r)});
    \begin{scriptsize}
    
    \draw [fill=uuuuuu] (0.,0.) circle (1.5pt);
    \draw[color=uuuuuu] (0.14,0.28) node {$O$};
    \draw [fill=ffqqqq] (-1.14,2.98) circle (1.5pt);
    \draw [fill=ffqqqq] (-1.6888952863865674,2.4794420161845356) circle (1.5pt);
    
    \draw [fill=ffqqqq] (-1.2990482868580284,3.705636253936034) circle (1.5pt);
    \draw [fill=qqzzqq] (-3.298368942978354,3.7577602792083287) circle (1.5pt);
    \draw [fill=qqzzqq] (0.23010458949184187,4.994702381313104) circle (1.5pt);
    \draw[color=ttttff] (0.28,2.1) node {$\geq R$};
    \draw[color=qqzzqq] (0.58,5.12) node {$Q_r$};
    \draw[color=qqzzqq] (-3.55,4.1) node {$Q_l$};
    \draw[color=qqzzqq] (-1.53,5.0) node {$q_S$};

    \draw[color=ttttff] (-0.6,4.42) node {$r$};
    \draw [fill=ffqqqq] (-0.8324030180784497,2.8822049225365443) circle (1.5pt);
    \draw[color=qqwuqq] (0.258,0.78) node {$\alpha \geq \pi/6$};
    \draw[color=qqwuqq] (-0.606,3.68) node {$\beta$};
    \draw[color=ffqqqq] (-1.406,3.38) node {$S$};
    \draw[color=ffqqqq] (-1.406,3.98) node {$W$};
    
    \end{scriptsize}
			\end{tikzpicture}
	\caption{Picture to the case~\ref{m_lmn23(b)}: middle component, $m=2, n=1$.}
	\label{pic_m_lmn23(b)}
	\end{center}

\end{figure}
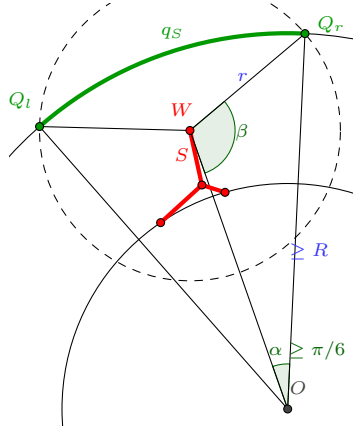    

Clearly, $\turn (S)=\pi/3$. 
To prove the statement, assume the contrary (i.e.\ $\gamma \geq \pi/3$) and as in the previous case connect $O$ with $Q_l$ and $Q_r$.
Denote the energetic point of $S$ by $W$.
Let us consider three subcases: 

\begin{itemize}
\item the point $W$ covers both $Q_r$ and $Q_l$ (see Fig.~\ref{pic_m_lmn23(b)}); 
\item the point $W$ covers $Q_l$ and $Q_r$ is covered by an entering point (see Fig.~\ref{pic_m_lmn23(b)_2}); 
\item $W$ covers $Q_l$ and $Q_r$ is covered by $H \in S \setminus (M_r \cup W)$ (see Fig.~\ref{pic_m_lmn23(b)_3}). 
\end{itemize}

{\sc In the subcase} (i) $|WQ_r| = |WQ_l| = r$.
Let us connect $O$ with $W$, and note that the angle $\angle Q_lOQ_r = \gamma$ splits into two parts; 
let us pick the largest one (without loss of generality it is $\angle WOQ_r$). Consider the triangle $\Delta OQ_rW$ with side $|OQ_r| \geq R$ and acute angle ($\alpha$ on the Fig.~\ref{pic_m_lmn23(b)}) at least $\pi/6$ against the side $|WQ_r| = r$. Recalling that $R > 2r$ and denoting by $\beta \defeq \angle OWQ_r$, by the law of sines for triangle 
$\Delta OQ_rW$ we get 
    \[
    \sin \beta = \frac{|OQ_r|}{r}\sin \alpha \geq \frac{R}{2r} > 1,
    \]
a contradiction.

  \begin{figure}
  \begin{center}
  
 \definecolor{zzttqq}{rgb}{0.6,0.2,0.}
\definecolor{qqwuqq}{rgb}{0.,0.39215686274509803,0.}
\definecolor{ffqqqq}{rgb}{1.,0.,0.}
\definecolor{uuuuuu}{rgb}{0.26666666666666666,0.26666666666666666,0.26666666666666666}
\begin{tikzpicture}[line cap=round,line join=round,>=triangle 45,x=1.5cm,y=1.5cm]
\clip(-4.610456379691199,-0.1821286294868982) rectangle (0.19983288376133055,1.972896537396661);
\fill[color=zzttqq,fill=zzttqq,fill opacity=0.1] (-4.49803372198501,0.13301366804085513) -- (0.,0.) -- (-2.80019456340442,1.0765270117745906) -- (-3.0766623281492755,0.3573963970035279) -- cycle;
\draw[color=qqwuqq,fill=qqwuqq,fill opacity=0.1] (-2.8712174710259895,0.8917866776967163) -- (-2.686477136948115,0.8207637700751464) -- (-2.6154542293265455,1.005504104153021) -- (-2.80019456340442,1.0765270117745906) -- cycle; 
\draw(0.,0.) circle (3.cm);
\draw(0.,0.) circle (4.5cm);
\draw [line width=1.6pt,color=ffqqqq] (-3.0766623281492755,0.3573963970035279)-- (-2.80019456340442,1.0765270117745906);
\draw [dash pattern=on 1pt off 1pt] (-4.49803372198501,0.13301366804085513)-- (-3.380592368902775,0.30941692291690637);
\draw [line width=1.6pt,color=ffqqqq] (-3.380592368902775,0.30941692291690637)-- (-3.0766623281492755,0.3573963970035279);
\draw [line width=1.6pt,color=ffqqqq] (-3.0766623281492755,0.3573963970035279)-- (-2.9895857492994153,0.2497539741142314);
\draw [shift={(0.,0.)},line width=1.6pt,color=qqwuqq]  plot[domain=2.7745653230580634:3.1120297547235607,variable=\t]({1.*4.5*cos(\t r)+0.*4.5*sin(\t r)},{0.*4.5*cos(\t r)+1.*4.5*sin(\t r)});
\draw [color=zzttqq] (-4.49803372198501,0.13301366804085513)-- (0.,0.);
\draw [color=zzttqq] (0.,0.)-- (-2.80019456340442,1.0765270117745906);
\draw [color=zzttqq] (-2.80019456340442,1.0765270117745906)-- (-3.0766623281492755,0.3573963970035279);
\draw [color=zzttqq] (-3.0766623281492755,0.3573963970035279)-- (-4.49803372198501,0.13301366804085513);
\draw (-4.20029184510663,1.6147905176618862)-- (-2.80019456340442,1.0765270117745906);
\begin{scriptsize}
\draw [fill=uuuuuu] (0.,0.) circle (1.5pt);
\draw[color=ffqqqq] (-2.8511253828103217,1.2447857804646967) node {$I$};
\draw [fill=qqwuqq] (-4.49803372198501,0.13301366804085513) circle (1.5pt);
\draw[color=qqwuqq] (-4.302561508752775,0.3689482526507701) node {$Q_l$};
\draw [fill=uuuuuu] (0.,0.) circle (1.5pt);
\draw[color=uuuuuu] (0.09456963646904813,0.13832618619204518) node {$O$};
\draw [fill=ffqqqq] (-2.80019456340442,1.0765270117745906) circle (1.5pt);
\draw [fill=ffqqqq] (-3.380592368902775,0.30941692291690637) circle (1.5pt);
\draw [fill=ffqqqq] (-3.380592368902775,0.30941692291690637) circle (1.5pt);
\draw [fill=ffqqqq] (-3.0766623281492755,0.3573963970035279) circle (1.5pt);
\draw [fill=ffqqqq] (-2.9895857492994153,0.2497539741142314) circle (1.5pt);
\draw[color=ffqqqq] (-3.096340844165524,0.596938695463354) node {$C$};
\draw[color=ffqqqq] (-3.496340844165524,0.496938695463354) node {$W$};
\draw [fill=qqwuqq] (-4.20029184510663,1.6147905176618862) circle (1.5pt);
\draw[color=qqwuqq] (-4.1319872282311305,1.7524417217177175) node {$Q_r$};
\draw[color=qqwuqq] (-4.204594958908732,1.0480155776010223) node {$q_S$};
\draw[color=qqwuqq] (-2.5618908070802865,0.7093791753611733) node {$\pi/2$};
\end{scriptsize}
\end{tikzpicture}

	\caption{Picture to the case~\ref{m_lmn23(b)}: middle component, $m=2, n=1$.}
	\label{pic_m_lmn23(b)_2}

\end{center}
\end{figure}
    

{\sc In the subcase} (ii) $Q_r$ is covered by the entering point $I$. Then $(CI)$ is perpendicular to $(IQ_r)$, where $C$ is the branching point of $S$, so points $Q_r$, $O$, $I$ lie on the same line. Consider the sum of the angles in the non convex quadrilateral $Q_lCIO$: it is $\angle Q_l + \angle C + \angle I + \angle O \geq \angle Q_l + 4\pi/3 + \pi/2 + \pi/3 > 2\pi$, a contradiction.

\begin{figure}
\begin{center}
  
\definecolor{zzttqq}{rgb}{0.6,0.2,0.} 
\definecolor{ffqqqq}{rgb}{1.,0.,0.} 
\definecolor{uuuuuu}{rgb}{0.26666666666666666,0.26666666666666666,0.26666666666666666} 
\definecolor{qqwuqq}{rgb}{0.,0.39215686274509803,0.} 
\begin{tikzpicture}[line cap=round,line join=round,>=triangle 45,x=2.0cm,y=2.0cm] 
\clip(-4.8753457531577835,-0.0786821450474917) rectangle (0.44365474793524708,1.2502305127944033); 
\fill[color=zzttqq,fill=zzttqq,fill opacity=0.1] (-4.49803372198501,0.13301366804085513) -- (0.,0.) -- (-4.375675180720522,1.0504602385747044) -- (-3.034901895363588,0.46602101633298965) -- (-3.0766623281492755,0.3573963970035279) -- cycle; 
\draw(0.,0.) circle (9.0cm); 
\draw [line width=1.6pt,dash pattern=on 1pt off 1pt] (-4.49803372198501,0.13301366804085513)-- (-3.38059236890277,0.30941692291690703); 
\draw [line width=1.6pt,color=ffqqqq] (-3.38059236890277,0.30941692291690703)-- (-3.0766623281492755,0.3573963970035279); 
\draw(0.,0.) circle (6.10cm); 
\draw [line width=1.6pt,color=ffqqqq] (-3.0766623281492755,0.3573963970035279)-- (-2.9975578488964283,0.5631580084837802); 
\draw [line width=1.6pt,color=ffqqqq] (-3.0346635714410906,0.3054783268390784)-- (-3.0766623281492755,0.3573963970035279); 
\draw [shift={(0.,0.)},line width=1.6pt,color=qqwuqq] plot[domain=2.9059832370405028:3.1120297547235607,variable=\t]({1.*4.5*cos(\t r)+0.*4.5*sin(\t r)},{0.*4.5*cos(\t r)+1.*4.5*sin(\t r)}); 
\draw [color=zzttqq] (-4.49803372198501,0.13301366804085513)-- (0.,0.); 
\draw [color=zzttqq] (0.,0.)-- (-4.375675180720522,1.0504602385747044); 
\draw [color=zzttqq] (-4.375675180720522,1.0504602385747044)-- (-3.034901895363588,0.46602101633298965); 
\draw [color=zzttqq] (-3.034901895363588,0.46602101633298965)-- (-3.0766623281492755,0.3573963970035279); 
\draw [color=zzttqq] (-3.0766623281492755,0.3573963970035279)-- (-4.49803372198501,0.13301366804085513); 
\begin{scriptsize} 
\draw [fill=qqwuqq] (-4.49803372198501,0.13301366804085513) circle (1.5pt); 
\draw[color=qqwuqq] (-4.65918247728505,0.2821315033108384) node {$Q_l$}; 
\draw [fill=uuuuuu] (0.,0.) circle (1.5pt); 
\draw[color=uuuuuu] (0.0853351448086585,0.12381190018424948) node {$O$}; 
\draw [fill=qqwuqq] (-4.49803372198501,0.13301366804085513) circle (1.5pt); 
\draw [fill=ffqqqq] (-3.38059236890277,0.30941692291690703) circle (1.5pt); 
\draw[color=ffqqqq] (-3.3921497218918976,0.4404511064374273) node {$W$}; 
\draw [fill=ffqqqq] (-3.0766623281492755,0.3573963970035279) circle (1.5pt); 
\draw [fill=ffqqqq] (-2.9975578488964283,0.5631580084837802) circle (1.5pt); 
\draw [fill=ffqqqq] (-3.0346635714410906,0.3054783268390784) circle (1.5pt); 
\draw [fill=qqwuqq] (-4.375675180720522,1.0504602385747044) circle (1.5pt); 
\draw[color=qqwuqq] (-4.512431608707691,1.1868669880422411) node {$Q_r$}; 
\draw [fill=ffqqqq] (-3.034901895363588,0.46602101633298965) circle (1.5pt); 
\draw[color=ffqqqq] (-2.9555105156387205,0.4075662430710489) node {$H$}; 
\draw[color=ffqqqq] (-3.2055105156387205,0.2575662430710489) node {$C$}; 
\draw[color=qqwuqq] (-4.56080704299637,0.7526925459370888) node {$q_S$}; 
\end{scriptsize} 
\end{tikzpicture} 

\caption{Picture to the case~\ref{m_lmn23(b)}: middle component, $m=2, n=1$.}
	\label{pic_m_lmn23(b)_3}

\end{center}
\end{figure}

{\sc In the subcase} (iii) $Q_r$ is covered by $H \in ]CI[$, where $C$ is the branching point of $S$, $I$ is an entering point of $S$.
Note that $(CI)$ is perpendicular to $(HQ_r)$; points $Q_l$, $W$, $C$ lie on the same line.
Consider the sum of the angles in the non convex pentagon $Q_lCHQ_rO$: it is $\angle Q_l + \angle C + \angle H + \angle Q_r + \angle O \geq \angle Q_l + 4\pi/3 + 3\pi/2 + \angle Q_r + \pi/3 > 3\pi$, a contradiction.

\item \textit{The last case $n = 1$, $m = 2$, the combinatorial type (c) on Fig.~\ref{lmn23} (see Fig.~\ref{Pic_m_lmn23(c)}). \label{m_lmn23(c)}}
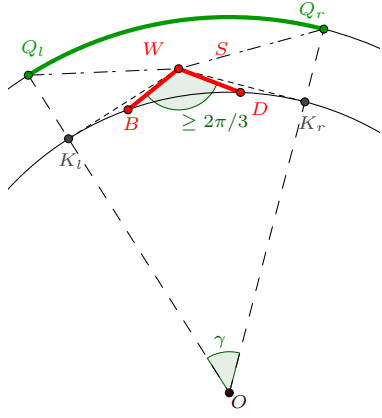
\begin{figure}[h]
\begin{center}
   \definecolor{ffqqqq}{rgb}{1.,0.,0.}
\definecolor{qqzzqq}{rgb}{0.,0.6,0.}
\definecolor{qqwuqq}{rgb}{0.,0.39215686274509803,0.}
\definecolor{uuuuuu}{rgb}{0.26666666666666666,0.26666666666666666,0.26666666666666666}
\definecolor{ttqqqq}{rgb}{0.2,0.,0.}
\begin{tikzpicture}[line cap=round,line join=round,>=triangle 45,x=1.0cm,y=1.0cm]
\clip(-2.933603001946984,-0.236561137307231) rectangle (1.9935627043530988,5.227510471506927);
\draw [shift={(-0.6741205292953993,4.312389449139797)},color=qqwuqq,fill=qqwuqq,fill opacity=0.1] (0,0) -- (-140.95797006931792:0.5474628562555648) arc (-140.95797006931792:-20.95797006931793:0.5474628562555648) -- cycle;
\draw [shift={(0.,0.)},color=qqwuqq,fill=qqwuqq,fill opacity=0.1] (0,0) -- (75.46268021082884:0.5474628562555648) arc (75.46268021082884:122.30667084540485:0.5474628562555648) -- cycle;
\draw(0.,0.) circle (5.cm);
\draw(0.,0.) circle (4.cm);
\draw [dash pattern=on 5pt off 5pt] (-2.672253790476793,4.225998069010729)-- (0.,0.);
\draw [dash pattern=on 5pt off 5pt] (0.,0.)-- (1.2550527847876427,4.839921745999287);
\draw [dash pattern=on 1pt off 2pt on 5pt off 4pt] (-0.6741205292953993,4.312389449139797)-- (-2.672253790476793,4.225998069010729);
\draw [dash pattern=on 1pt off 2pt on 5pt off 4pt] (-0.6741205292953993,4.312389449139797)-- (1.2550527847876427,4.839921745999287);
\draw [dash pattern=on 2pt off 2pt] (1.0040422278301142,3.8719373967994297)-- (-0.6741205292953993,4.312389449139797);
\draw [dash pattern=on 2pt off 2pt] (-2.1378030323814343,3.3807984552085837)-- (-0.6741205292953993,4.312389449139797);
\draw [line width=1.6pt,color=ffqqqq] (0.14868696221123126,3.9972355681481164)-- (-0.6741205292953993,4.312389449139797);
\draw [line width=1.6pt,color=ffqqqq] (-1.3476800996158633,3.766133076392731)-- (-0.6741205292953993,4.312389449139797);
\draw [shift={(0.,0.)}, line width=1.6pt,color=qqzzqq]  plot[domain=1.3247591485867362:2.131180954624042,variable=\t]({1.*5.*cos(\t r)+0.*5.*sin(\t r)},{0.*5.*cos(\t r)+1.*5.*sin(\t r)});

\begin{scriptsize}
\draw [fill=ttqqqq] (0.,0.) circle (1.5pt);
\draw[color=ttqqqq] (0.1321889930841786, -0.12113067082648408) node {$O$};
\draw [fill=ffqqqq] (-0.6741205292953993,4.312389449139797) circle (1.5pt);
\draw [fill=qqzzqq] (1.2550527847876427,4.839921745999287) circle (1.5pt);
\draw [fill=qqzzqq] (-2.672253790476793,4.225998069010729) circle (1.5pt);
\draw [fill=uuuuuu] (-2.1378030323814343,3.3807984552085837) circle (1.5pt);
\draw [fill=uuuuuu] (1.0040422278301142,3.8719373967994297) circle (1.5pt);
\draw [fill=ffqqqq] (0.14868696221123126,3.9972355681481164) circle (1.5pt);
\draw[color=qqwuqq] (-0.19569146933380757,3.597151617735788) node {$\geq 2\pi/3$};
\draw[color=ffqqqq] (-0.10504491129323279,4.582584758995801) node {$S$};
\draw[color=ffqqqq] (-1.00504491129323279,4.582584758995801) node {$W$};
\draw[color=ffqqqq] (-1.30504491129323279,3.582584758995801) node {$B$};
\draw[color=ffqqqq] (0.40504491129323279,3.782584758995801) node {$D$};

\draw[color=qqzzqq] (1.10504491129323279, 5.082584758995801) node {$Q_r$};
\draw[color=qqzzqq] (-2.60504491129323279, 4.582584758995801) node {$Q_l$};
\draw[color=uuuuuu] (1.10504491129323279, 3.582584758995801) node {$K_r$};
\draw[color=uuuuuu] (-2.10504491129323279, 3.082584758995801) node {$K_l$};

\draw [fill=ffqqqq] (-1.3476800996158633,3.766133076392731) circle (1.5pt);
\draw[color=qqwuqq] (-0.12343280246010605, 0.6853695277031529) node {$\gamma$};
\end{scriptsize}
\end{tikzpicture}
    \caption{Picture to the case~\ref{m_lmn23(c)}: middle component, $m=2$, $n=1$.}
    \label{Pic_m_lmn23(c)}
\end{center}
\end{figure}  
Then $S$ consists of two segments, i.e.\ $S = [BW] \cup [WD]$, where $B, D \in M_r$ are entering points, $W$ is energetic and $\angle BWD \geq 2 \pi/3$. In this case $\turn (S) = \pi - \angle BWD$.

First, connect $O$ with $Q_l$ and $Q_r$ then denote $K_l = [OQ_l] \cap M_r$ and $K_r = [OQ_r] \cap M_r$. 
Now consider the convex quadrilateral $P = K_lOK_rW$. 
The sum of the angles $\angle K_l + \angle K_r + \angle W$ of $P$ is at least $\pi/2+\angle BWD+ \pi/2$, so that the remaining angle (which is equal to $\gamma$) 
is at most $\pi - \angle BWD = \turn (S)$ as claimed. 

If one has the equality then both $[BW]$ and $[WD]$ are tangent to $M_r$, but $W$ is not energetic point in this case, 
because $Q_l$ is covered by $B = K_l$, $Q_r$ is covered by $D = K_r$, so we got a contradiction.
\end{enumerate}

    \item \textit{Let $S$ be an ending component (without loss of generality let it be the left one, so $Q_r = A$). Recall that $C$ denotes the branching point if $S$ is a tripod and the entering point if $S$ is a seqment. Then there are two options:}
    \begin{enumerate}
    \item \textit{The case $n = 1$, $ m = 1$, the combinatorial type (a) on Fig.~\ref{lmn23}  (see Fig.~\ref{pic_e_lmn23(a)}).} \label{e_lmn23(a)}
    In this case $S = [CS'_l]$, where $C \in M_r$, $|S'_lQ_r| = r$, and
    $\turn(S)=0$.  Denote by $K$ such a point that $K \in [OQ_l)$ and $\angle OQ_rK = \pi/2$. Define the points 
$L \defeq [S'_lC) \cap (OQ_l)$ and $P \defeq [CS'_l) \cap (Q_rK)$, and introduce the angles $\alpha \defeq \angle PS'_lQ_r$ and $\beta \defeq \angle S'_lQ_rK$.

    The following two situations have to be considered. Note that $|S'_lQ_l| = r$, otherwise one can replace $[CS'_l] \cap B_\varepsilon(S'_l)$ in $\Sigma$ by the part $[DF]$ of the segment $[DQ_r]$ where $D = [CS'_l] \cap \partial B_\varepsilon(S'_l)$, $F$ is the point satisfying $\dist (F, Q_r) = r$, producing the competitor of strictly lower length.

    \begin{itemize}
    \item Case $\angle CS'_lQ_r \leq \pi$ (see the top picture on Fig.~\ref{pic_e_lmn23(a)}).
    \input{pictures/e_lmn23a}
    
    Then $\angle ([S_l'A), a) = \beta$ and $\angle([CS_l'), [S_l'A))=\alpha$, so that
    $$\turn (S) + \angle([CS_l'), [S_l'A)) + \angle ([S_l'A), a) = \alpha + \beta.$$ 
    Note that $\angle S'_lPK = \alpha + \beta$ and $\angle OKQ_r = \pi/2 - \gamma$.  If $\alpha + \beta \leq \gamma$ (contrary to the claim being proven), then $\angle OKP + \angle KPS'_l < \pi/2$ so $\angle KLP > \pi/2$, which is impossible because then $|CQ_l| < |S'_lQ_l|$ which contradicts $|S'_lQ_l| = r$, $|CQ_l|\geq r$.
    
    \item Case $\angle CS'_lQ_r > \pi$ (see the bottom picture on Fig.~\ref{pic_e_lmn23(a)}).
    In this case $\angle ([S_l'A), a) = \beta$ and $\angle([CS_l'), [S_l'A)) = -\alpha$, so that 
    $$\turn (S) + \angle([CS_l'), [S_l'A)) + \angle ([S_l'A), a) = \beta - \alpha$$
    and we know that $\angle KPC = \beta - \alpha$. If $\beta - \alpha \leq \gamma$ (the contrary to the claim being proven), then  $\angle OKP + \angle KPC < \pi/2$, which is impossible because then $|CQ_l| < |S'_lQ_l|$ which contradicts $|S'_lQ_l| = r$, $|CQ_l|\geq r$.
    \end{itemize}

    \item \textit{The case $n = 2$, $m = 1$, the combinatorial type (b) on Fig.~\ref{lmn23} (see Fig.~\ref{pic_e_lmn23(b)}).}  \label{e_lmn23(b)}
    \begin{figure}
\begin{center}
\definecolor{qqttcc}{rgb}{0.,0.2,0.8}
    \definecolor{qqzzqq}{rgb}{0.,0.6,0.}
    \definecolor{qqwuqq}{rgb}{0.,0.39215686274509803,0.}
    \definecolor{ffqqqq}{rgb}{1.,0.,0.}
    \definecolor{qqqqff}{rgb}{0.,0.,1.}
    \definecolor{xdxdff}{rgb}{0.49019607843137253,0.49019607843137253,1.}
    \definecolor{uuuuuu}{rgb}{0.26666666666666666,0.26666666666666666,0.26666666666666666}
    \begin{tikzpicture}[line cap=round,line join=round,>=triangle 45,x=2.3cm,y=2.3cm]
    \clip(-2.7257172412593683,1.88660295599792038) rectangle (1.026573021317395,5.186534550985518);
    \draw [shift={(-1.26,4.06)},color=qqwuqq,fill=qqwuqq,fill opacity=0.1] (0,0) -- (-142.85965881923755:0.2900956275786282) arc (-142.85965881923755:-82.85965881923765:0.2900956275786282) -- cycle;
    \draw [shift={(-0.018961932606240354,4.999964044381903)},color=qqwuqq,fill=qqwuqq,fill opacity=0.1] (0,0) -- (-179.78271173720103:0.2900956275786282) arc (-179.78271173720103:-142.85965881923758:0.2900956275786282) -- cycle;
    \draw [shift={(0.,0.)},color=qqwuqq,fill=qqwuqq,fill opacity=0.1] (0,0) -- (90.21728826279899:0.2900956275786282) arc (90.21728826279899:116.57392441624:0.2900956275786282) -- cycle;
    \draw [shift={(-1.3771056238513,4.9948134015143975)},color=qqwuqq,fill=qqwuqq,fill opacity=0.1] (0,0) -- (-82.85965881923765:0.2900956275786282) arc (-82.85965881923765:-82.85965881923765:0.2900956275786282) -- cycle;
    \draw [shift={(-1.3771056238513,4.9948134015143975)},color=qqwuqq,fill=qqwuqq,fill opacity=0.1] (0,0) -- (-179.78271173720103:0.2900956275786282) arc (-179.78271173720103:-82.85965881923762:0.2900956275786282) -- cycle;
    \draw [shift={(-2.4962507394703195,4.990569140141595)},color=qqwuqq,fill=qqwuqq,fill opacity=0.1] (0,0) -- (-63.426075583760024:0.2900956275786282) arc (-63.426075583760024:0.21728826279897984:0.2900956275786282) -- cycle;
    \draw[color=qqttcc,fill=qqttcc,fill opacity=0.1] (-0.22408904295463972,4.999186117499307) -- (-0.22331111607204351,4.794059007150907) -- (-0.01818400572364414,4.794836934033503) -- (-0.018961932606240354,4.999964044381903) -- cycle; 
    \draw(0.,0.) circle (11.5cm);
    \draw(0.,0.) circle (9.2cm);
    \draw [line width=1.6pt,color=ffqqqq] (-1.2282748750046435,3.8067493785948567)-- (-1.26,4.06);
    \draw [dash pattern=on 3pt off 3pt,domain=-1.26:1.026573021317395] plot(\x,{(--1.020209265560636--0.15410007488678712*\x)/0.2034589091633705});
    \draw [dash pattern=on 3pt off 3pt,domain=-3.257172412593683:-1.26] plot(\x,{(-0.8299174900713036--0.09915054651835575*\x)/-0.23518403415872702});
    \draw [dotted](-2.236760538028662,4.471789607698212) circle (2.3cm);
    \draw [dotted](-0.018961932606240354,4.999964044381903) circle (2.3cm);
    \draw [line width=1.6pt,color=ffqqqq] (-1.26,4.06)-- (-0.8161209532126548,4.39619463817965);
    \draw [line width=1.2pt,color=ffqqqq] (-1.26,4.06)-- (-1.315301383926456,4.08331434809826);
    \draw (0.,0.)-- (-0.018961932606240354,4.999964044381903);
    \draw [domain=-3.257172412593683:1.026573021317395] plot(\x,{(-25.-0.018961932606240354*\x)/-4.999964044381903});
    \draw [domain=-3.257172412593683:0.0] plot(\x,{(-0.--4.471789607698212*\x)/-2.236760538028662});
    \draw [dash pattern=on 3pt off 3pt,domain=-3.257172412593683:-0.8161209532126548] plot(\x,{(-2.225754174062981-0.33619463817965034*\x)/-0.44387904678734524});
    \draw [domain=-3.257172412593683:1.026573021317395] plot(\x,{(-0.19029177548933252-0.2532506214051429*\x)/0.03172512499535651});
    \draw [shift={(0.,0.)},line width=1.6pt,color=qqzzqq]  plot[domain=1.5745887224066777:2.034598803034396,variable=\t]({1.*5.*cos(\t r)+0.*5.*sin(\t r)},{0.*5.*cos(\t r)+1.*5.*sin(\t r)});
    \begin{scriptsize}
    \draw [fill=uuuuuu] (0.,0.) circle (1.5pt);
    \draw[color=uuuuuu] (0.06925745030792105,0.13580369181236068) node {$O$};
    \draw [fill=ffqqqq] (-1.2282748750046435,3.8067493785948567) circle (1.5pt);
    \draw[color=ffqqqq] (-1.1588140397749385,3.9457262673450013) node {$B$};
    \draw [fill=ffqqqq] (-1.26,4.06) circle (1.5pt);
    \draw[color=ffqqqq] (-1.1878236025328015,4.197142477913145) node {$C$};
    \draw[color=ffqqqq] (-0.8878236025328015,4.197142477913145) node {$S$};

    \draw [fill=qqzzqq] (-2.236760538028662,4.471789607698212) circle (1.5pt);
    \draw[color=qqzzqq] (-2.1644788820475167,4.6032763565232235) node {$Q_l$};
    \draw[color=qqzzqq] (-1.1644788820475167,4.7532763565232235) node {$q_S$};
    \draw [fill=qqzzqq] (-0.018961932606240354,4.999964044381903) circle (1.5pt);
    \draw[color=qqzzqq] (0.04991774180267916,5.135118340417374) node {$Q_r$};
    \draw [fill=ffqqqq] (-1.315301383926456,4.08331434809826) circle (1.5pt);
    \draw[color=ffqqqq] (-1.3658427280485271,4.2164821864183875) node {$W$};
    \draw [fill=ffqqqq] (-0.8161209532126548,4.39619463817965) circle (1.5pt);
    \draw[color=ffqqqq] (-0.752680161164859,4.5355873767548776) node {$S'_l$};
    \draw [fill=uuuuuu] (0.7781970880001738,5.603733450584156) circle (1.5pt);
    \draw[color=uuuuuu] (2.041907717842593,5.347855133975034) node {$I$};
    \draw[color=black] (-3.1798135785727153,4.912711692607093) node {$l$};
    \draw[color=qqwuqq] (-1.29750868089751407,3.923085101365969) node {$\alpha $};
    \draw[color=qqwuqq] (-0.2659751183498561,4.880400672375439) node {$\beta$};
    \draw[color=qqwuqq] (-0.044008916758634,0.2034926715807071) node {$\gamma$};
    \draw [fill=uuuuuu] (-1.0024593490929847,2.0041426979773016) circle (1.5pt);
    \draw[color=uuuuuu] (-0.9364073919646569,2.1374635221048903) node {$L$};
    \draw [fill=uuuuuu] (-2.4962507394703195,4.990569140141595) circle (1.5pt);
    \draw[color=uuuuuu] (-2.425564946868282,5.125448486164753) node {$K$};
    \draw [fill=uuuuuu] (-1.3771056238513,4.9948134015143975) circle (1.5pt);
    \draw[color=uuuuuu] (-1.3135317078168736,5.125448486164753) node {$P$};
    \draw[color=qqwuqq] (-2.967076785015055,5.463893385006485) node {$\zeta = 0\textrm{\degre}$};
    \draw[color=qqwuqq] (-1.501441855223177,4.903041838354472) node {$\alpha+\beta$};
    \draw[color=qqwuqq] (-2.19817420884898558,4.912051401112335) node {$\frac{\pi}{2}-\gamma$};

    \end{scriptsize}
    \end{tikzpicture}
    \caption{Picture to the case~\ref{e_lmn23(b)}: ending component, $n = 2$, $m = 1$.}
    \label{pic_e_lmn23(b)}
	\end{center}
    \end{figure}
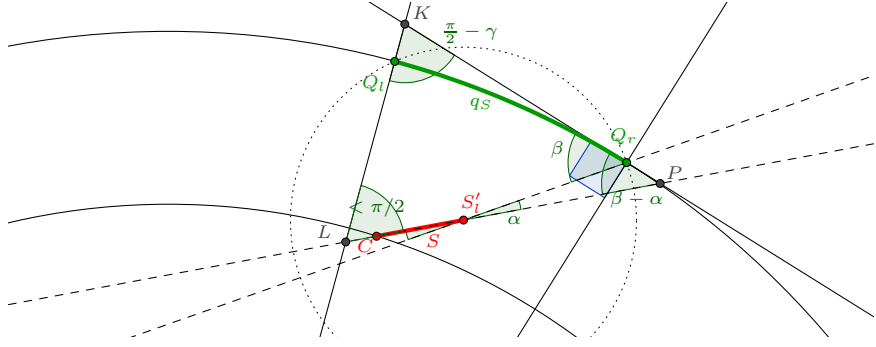
    
    Note that $S$ is a tripod: $S = [BC]\cup [CW] \cup [CS'_l] \subset \overline{(N\setminus N_r)}$, where $B \in M_r$. 
    Let us prove that $Q_r = [CS'_l) \cap M$ and $Q_l = [CW) \cap M$. Suppose the contrary i.e.\ without loss of generality $C$, $S_l'$, and $Q_r$ do not lie on the same line.
    Let us pick a sufficiently small $\varepsilon > 0$ and denote by $J$ the intersection point of $\partial B_\varepsilon (S_l')$ with $[CS_l']$. Then one may replace $[JS_l']$ by $[JI]$ in $\Sigma$, where $I$ stands for the intersection point of $\partial B_r(Q_r)$ with $[JQ_r]$. Clearly the resulting set covers $q_{S_l}$, so it has the same energy $F_M$; by the triangle inequality it has strictly lower length, so we got a contradiction.  
    
    Note that $|S'_lQ_r|=r=|WQ_l|$; $B_r(Q_r) \cap \Sigma = B_r(Q_l) \cap \Sigma = \emptyset$. 
    Let $K\in [OQ_l)$ be the point satisfying $(Q_rK) \perp (OQ_r)$. 
    Then $\alpha \defeq \turn(S) = \angle([BC), [CQ_r)) = \pi/3$, $\angle([CS_l'), [S_l'A)) = 0$ and  $\beta \defeq ([CQ_r), [Q_rK)) = \angle ([S_l'A), a)$. We have to show $\alpha+ \beta > \gamma$.
    Let $P$ be the point of intersection of $(KQ_r]$ and $[BC)$. Then $\angle OKP = \pi/2 - \gamma$ and $\angle KPC = \alpha + \beta$. Assume the contrary, i.e.\ $\alpha+ \beta \leq \gamma$. Then $\angle OKP + \angle KPC \leq \pi/2$ hence $\angle KLP \geq \pi/2$, where $L$ is the point of intersection of $(BC)$ and $(OK)$, but since $\angle Q_lCL = 2\pi/3$, then the sum of the angles of the triangle $\Delta CLQ_l$ exceeds $\pi$, which is impossible.

    \item \textit{The case $n = 2$, $ m = 1$, the combinatorial type (c) on Fig.~\ref{lmn23} (see Fig.~\ref{pic_e_lmn23(c)})}.
    \label{e_lmn23(c)}
    \begin{figure}
\begin{center}

\definecolor{ttqqqq}{rgb}{0.2,0.,0.}
\definecolor{ttzzqq}{rgb}{0.2,0.6,0.}
\definecolor{qqwuqq}{rgb}{0.,0.39215686274509803,0.}
\definecolor{uuuuuu}{rgb}{0.26666666666666666,0.26666666666666666,0.26666666666666666}
\definecolor{ffqqqq}{rgb}{1.,0.,0.}
\definecolor{xdxdff}{rgb}{0.49019607843137253,0.49019607843137253,1.}
\definecolor{qqqqff}{rgb}{0.,0.,1.}
\begin{tikzpicture}[line cap=round,line join=round,>=triangle 45,x=3.0cm,y=3.0cm]
\clip(8.087244680931269,6.330215894322556) rectangle (11.512990429972701,9.973735211762673);
\draw [shift={(0.8342481486993167,1.6425662766580746)},color=qqwuqq,fill=qqwuqq,fill opacity=0.1] (0,0) -- (109.1027822100106:0.25127719430621753) arc (109.1027822100106:146.04369252437712:0.25127719430621753) -- cycle;
\draw [shift={(-0.2550202615275577,3.398667261271955)},color=qqwuqq,fill=qqwuqq,fill opacity=0.1] (0,0) -- (-95.68993774918141:0.25127719430621753) arc (-95.68993774918141:-29.28109573597084:0.25127719430621753) -- cycle;
\draw [shift={(-0.2550202615275577,3.398667261271955)},color=qqqqff,fill=qqqqff,fill opacity=0.1] (0,0) -- (-149.28109573597084:0.25127719430621753) arc (-149.28109573597084:-95.68993774918141:0.25127719430621753) -- cycle;
\draw [shift={(-0.669513117927742,3.152374404570396)},color=qqqqff,fill=qqqqff,fill opacity=0.1] (0,0) -- (-72.29809973885571:0.16751812953747836) arc (-72.29809973885571:-9.56241946371832:0.16751812953747836) -- cycle;
\draw [shift={(-0.669513117927742,3.152374404570396)},color=qqwuqq,fill=qqwuqq,fill opacity=0.1] (0,0) -- (-9.562419463718317:0.16751812953747836) arc (-9.562419463718317:30.718904264029195:0.16751812953747836) -- cycle;
\draw [shift={(11.060956949431324,6.464360883202412)},color=qqwuqq,fill=qqwuqq,fill opacity=0.1] (0,0) -- (108.807968982505:0.25127719430621753) arc (108.807968982505:147.76435727299193:0.25127719430621753) -- cycle;
\draw(0.8342481486993167,1.6425662766580746) circle (6.cm);
\draw(0.8342481486993167,1.6425662766580746) circle (9.cm);
\draw [line width=2.pt,color=ffqqqq] (-0.669513117927742,3.152374404570396)-- (-0.2550202615275577,3.398667261271955);
\draw [domain=8.087244680931269:0.8342481486993167] plot(\x,{(-5.485281875386651--1.6756816043594478*\x)/-2.488391279684798});
\draw [domain=8.087244680931269:0.8342481486993167] plot(\x,{(-2.6517221587893767--1.8898660435188797*\x)/-0.6545275682156528});
\draw [dotted] (-0.1475432036241624,4.477365341936394) circle (3.cm);
\draw [dash pattern=on 2pt off 2pt,domain=-1.6541431309854813:11.512990429972701] plot(\x,{(--2.9928679827374025-0.16587347644712658*\x)/0.9846300130577392});
\draw [dash pattern=on 2pt off 2pt,domain=8.087244680931269:-0.1475432036241624] plot(\x,{(-0.6403686246744196-1.0786980806644388*\x)/-0.10747705790339529});
\draw [dotted] (-1.6541431309854813,3.3182478810175224) circle (3.cm);
\draw [shift={(0.8342481486993167,1.6425662766580746)},line width=2.pt,color=ttzzqq]  plot[domain=1.9042027726509805:1.5489432863206094,variable=\t]({1.*3.*cos(\t r)+0.*3.*sin(\t r)},{0.*3.*cos(\t r)+1.*3.*sin(\t r)});
\draw [line width=2.pt,color=ffqqqq] (-0.6243691738542607,3.0109359293274958)-- (-0.669513117927742,3.152374404570396);
\draw [line width=2.pt,color=ffqqqq] (-0.2550202615275577,3.398667261271955)-- (-0.24668820076800613,3.4822923448445473);
\draw [line width=2.pt,color=ffqqqq] (-0.2550202615275577,3.398667261271955)-- (-0.18842735425026425,3.3613259502793817);
\draw(11.060956949431324,6.464360883202412) circle (6.cm);
\draw(11.060956949431324,6.464360883202412) circle (9.cm);
\draw [dotted] (10.093764879004203,9.304174169125075) circle (3.cm);
\draw [dotted] (8.523372421905853,8.064568608987827) circle (3.cm);
\draw [shift={(11.060956949431324,6.464360883202412)},line width=2.pt,color=ttzzqq]  plot[domain=1.8990573111525768:2.5789745515069384,variable=\t]({1.*3.*cos(\t r)+0.*3.*sin(\t r)},{0.*3.*cos(\t r)+1.*3.*sin(\t r)});
\draw [color=ttqqqq] (8.523372421905853,8.064568608987827)-- (11.060956949431324,6.464360883202412);
\draw (11.060956949431324,6.464360883202412)-- (10.093764879004203,9.304174169125075);
\draw [line width=1.6pt,color=ffqqqq] (9.912658247971578,8.101862056400726)-- (9.814004234050453,8.3441043896116);
\draw [dash pattern=on 1pt off 1pt] (9.814004234050453,8.3441043896116)-- (8.523372421905853,8.064568608987827);
\draw [dash pattern=on 1pt off 1pt] (9.544847256736501,7.420423791480868)-- (10.093764879004203,9.304174169125075);
\draw [dash pattern=on 1pt off 1pt] (10.413176317780554,6.872853116717308)-- (9.814004234050453,8.3441043896116);
\draw [line width=2.pt,color=ffqqqq] (9.500711404992064,8.276248820958644)-- (9.814004234050453,8.3441043896116);
\begin{scriptsize}

\draw [fill=uuuuuu] (11.060956949431324,6.464360883202412) circle (1.5pt);
\draw[color=uuuuuu] (11.119322825559625,6.614996714536266) node {$O$};
\draw [fill=ttzzqq] (8.523372421905853,8.064568608987827) circle (1.5pt);
\draw[color=ttzzqq] (8.363649594668107,8.164539412757925) node {$Q_l$};
\draw [fill=ttzzqq] (10.093764879004203,9.304174169125075) circle (1.5pt);
\draw[color=ttzzqq] (10.164469487196,9.496308542580865) node {$Q_r$};
\draw [fill=ffqqqq] (9.912658247971578,8.101862056400726) circle (1.5pt);
\draw[color=ffqqqq] (9.9718236382279,8.256674384003537) node {$C$};
\draw [fill=ffqqqq] (9.814004234050453,8.3441043896116) circle (1.5pt);
\draw[color=ffqqqq] (10.021568199366655,8.433079297740374) node {$W_r$};
\draw [fill=ffqqqq] (9.500711404992064,8.276248820958644) circle (1.5pt);
\draw[color=ffqqqq] (9.33525474598548,8.440944326494762) node {$W_l$};
\draw [fill=uuuuuu] (9.544847256736501,7.420423791480868) circle (1.5pt);
\draw[color=uuuuuu] (9.628411472676069,7.486090988131145) node {$K$};
\draw [fill=uuuuuu] (10.413176317780554,6.872853116717308) circle (1.5pt);
\draw[color=uuuuuu] (10.474378026840334,6.991912505995589) node {$L$};
\end{scriptsize}
\end{tikzpicture}

\caption{Picture to the case~\ref{e_lmn23(c)}: ending component, $n = 2$, $m = 1$.}
\label{pic_e_lmn23(c)}

\end{center}
\end{figure}
    In this case $A = Q_r$, $S_l' = W_r$. Denote $\angle ([CW_r),[W_rQ_r))$ by $\alpha$, $\angle ([S_l'A), a) = \angle ([W_rQ_r),a)$ by $\beta$, 
    Clearly $\turn (S) = \alpha + \beta$, $\turn (q_S) = \gamma$. 
    Let $L$ be the point of intersection of $(W_rC)$ and $(Q_lO)$.
    Suppose the contrary, i.e. $\gamma \geq \alpha + \beta$.
    Then $$\angle W_rLQ_l = \pi - \angle W_rLO = \pi - (2\pi - \angle LW_rQ_r - \angle W_rQ_rO - \angle Q_rOL) = $$
    $$\pi - (2\pi - (\pi - \alpha) - (\pi/2 - \beta) - \gamma) =  \pi/2 - \beta - \alpha + \gamma \geq \pi/2,$$
    which is impossible because then $|CQ_l| < |S'_lQ_l|$, which contradicts $|S'_lQ_l| = r$, $|CQ_l|\geq r$.
    \end{enumerate}
    
    \end{enumerate}
    \end{proof}

\begin{proof} [Proof of Corollary~\ref{coroll}]





Let $\hat \Sigma$ be a local minimizer in the sense of Definition~\ref{local}.
Suppose the claim is false, i.e.\ 
\begin{equation}
\H(\hat \Sigma) - \H(\Sigma) < (R - 5r)/2
\label{qq1}
\end{equation}
and $\hat \Sigma$ is not a horseshoe.
Suppose first that $\hat \Sigma_r$ contains no line segment of 
length exceeding 
$$a_M'(r) := 2r + \H(\hat \Sigma) - \H(\Sigma) < 2r + (R - 5r)/2.$$
Then Lemma~\ref{hex} remains true for this situation with $a_M'$ instead of $a_M$, because $2a'_M (r) + r < R$. 
Lemma~\ref{compcon} also remains true with $a_M'(r)$ instead of $a_M$ by the same reason.
We may repeat now line by line the proof of Theorem~\ref{theo1} without any change because all the arguments used in this proof as well as
in Lemma~\ref{central} are local, except the Lemma~\ref{hex} and part of Lemma~\ref{compcon} (the claim $m(S) \leq 2$) which hold true with $a_M'$ instead of $a_M$.
This proves that $\hat \Sigma$ is a horseshoe in the considered case.

On the other hand it is impossible to $\hat \Sigma_r$ to have a  segment of length at least $a_M'(r)$, otherwise using the replacement from Lemma~\ref{global_lm}(iii) and get a contradiction with (\ref{qq1}).



\end{proof}

\begin{figure}
\begin{center}

\definecolor{ffqqqq}{rgb}{1.,0.,0.}
\definecolor{qqwuqq}{rgb}{0.,0.39215686274509803,0.}
\definecolor{yqqqqq}{rgb}{0.5019607843137255,0.,0.}
\definecolor{qqqqcc}{rgb}{0.,0.,0.8}
\definecolor{xdxdff}{rgb}{0.49019607843137253,0.49019607843137253,1.}
\definecolor{uuuuuu}{rgb}{0.26666666666666666,0.26666666666666666,0.26666666666666666}
\definecolor{qqqqff}{rgb}{0.,0.,1.}
\begin{tikzpicture}[line cap=round,line join=round,>=triangle 45,x=1.0cm,y=1.0cm]
\clip(-4.26,-0.24) rectangle (7.34,6.12);
\draw [shift={(0.7606597282413003,3.529612582362496)},color=qqwuqq,fill=qqwuqq,fill opacity=0.1] (0,0) -- (131.96270954699338:0.6) arc (131.96270954699338:251.96270954699332:0.6) -- cycle;
\draw [shift={(4.82,3.56)},color=qqwuqq,fill=qqwuqq,fill opacity=0.1] (0,0) -- (-101.7251120151651:0.6) arc (-101.7251120151651:39.596208639750294:0.6) -- cycle;
\draw [shift={(0.7606597282413003,3.529612582362496)},color=qqwuqq,fill=qqwuqq,fill opacity=0.1] (0,0) -- (11.96270954699333:0.6) arc (11.96270954699333:131.9627095469934:0.6) -- cycle;
\draw [shift={(0.7606597282413003,3.529612582362496)},color=qqwuqq,fill=qqwuqq,fill opacity=0.1] (0,0) -- (-108.03729045300666:0.6) arc (-108.03729045300666:11.962709546993352:0.6) -- cycle;
\draw [line width=1.2pt,color=ffqqqq] (-3.48,1.42)-- (-2.36,4.42);
\draw [line width=1.2pt,color=ffqqqq] (-0.04,4.42)-- (0.7606597282413003,3.529612582362496);
\draw [line width=1.2pt,color=ffqqqq] (0.7606597282413003,3.529612582362496)-- (2.32,3.86);
\draw [line width=1.2pt,color=ffqqqq] (0.7606597282413003,3.529612582362496)-- (-0.01581521935548058,1.1451500585551955);
\draw [line width=1.2pt,color=ffqqqq] (4.38,1.44)-- (4.82,3.56);
\draw [line width=1.2pt,color=ffqqqq] (6.44,4.9)-- (4.82,3.56);
\begin{scriptsize}
\draw [fill=qqqqff] (-3.48,1.42) circle (2.5pt);
\draw [fill=qqqqff] (-2.36,4.42) circle (2.5pt);
\draw [fill=qqqqff] (-0.04,4.42) circle (2.5pt);
\draw [fill=qqqqff] (2.32,3.86) circle (2.5pt);
\draw [fill=qqqqff] (0.7606597282413003,3.529612582362496) circle (2.5pt);
\draw [fill=qqqqff] (-0.01581521935548058,1.1451500585551955) circle (2.5pt);
\draw[color=qqwuqq] (-0.19,3.52) node {$2\pi/3$};
\draw [fill=qqqqff] (4.38,1.44) circle (2.5pt);
\draw [fill=qqqqff] (4.82,3.56) circle (2.5pt);
\draw [fill=qqqqff] (6.44,4.9) circle (2.5pt);
\draw[color=qqwuqq] (6.10,3.44) node {$\geq 2\pi/3$};
\draw[color=qqwuqq] (1.48,4.42) node {$2\pi/3$};
\draw[color=qqwuqq] (1.7,3.34) node {$2\pi/3$};

\draw[color=uuuuuu] (-3.0,1.0) node {$(a)$};
\draw[color=uuuuuu] (1.0,1.0) node {$(b)$};
\draw[color=uuuuuu] (4.0,1.0) node {$(c)$};

\end{scriptsize}
\end{tikzpicture}

\caption{Locally miminal networks for sets of $2$ and $3$ points.}
\label{lmn23}
\end{center}
\end{figure}
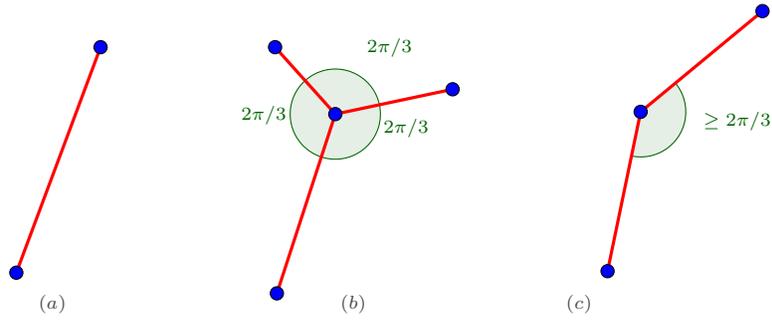
\begin{figure}
\begin{center}

\definecolor{qqwuqq}{rgb}{0.,0.39215686274509803,0.}
\definecolor{ttqqqq}{rgb}{0.2,0.,0.}
\definecolor{ffqqqq}{rgb}{1.,0.,0.}
\definecolor{uuuuuu}{rgb}{0.26666666666666666,0.26666666666666666,0.26666666666666666}
\definecolor{qqqqff}{rgb}{0.,0.,1.}
\begin{tikzpicture}[line cap=round,line join=round,>=triangle 45,x=0.7cm,y=0.7cm]
\clip(-6.464675079481678,-7.669232706185321) rectangle (16.339264209268627,6.203551972605543);
\draw [shift={(0.8452994616207482,2.)},color=qqwuqq,fill=qqwuqq,fill opacity=0.1] (0,0) -- (180.:0.6315471617947253) arc (180.:300.:0.6315471617947253) -- cycle;
\draw [shift={(10.5,2.5566243270259372)},color=qqwuqq,fill=qqwuqq,fill opacity=0.1] (0,0) -- (150.:0.47366037134604394) arc (150.:270.:0.47366037134604394) -- cycle;
\draw [shift={(8.92339922363637,-6.046390732727287)},color=qqwuqq,fill=qqwuqq,fill opacity=0.1] (0,0) -- (-1.0050860052541697:0.6315471617947253) arc (-1.0050860052541697:160.30586724853868:0.6315471617947253) -- cycle;
\draw [shift={(-3.0877843563636187,-4.983454132727285)},color=qqwuqq,fill=qqwuqq,fill opacity=0.1] (0,0) -- (-0.14469213169344994:0.6315471617947253) arc (-0.14469213169344994:162.40757543781837:0.6315471617947253) -- cycle;
\draw [shift={(-1.1546710824314497,-4.988335938636718)},color=qqwuqq,fill=qqwuqq,fill opacity=0.1] (0,0) -- (59.85530786830656:0.6315471617947253) arc (59.85530786830656:179.8553078683066:0.6315471617947253) -- cycle;
\draw [shift={(-1.1546710824314497,-4.988335938636718)},color=qqwuqq,fill=qqwuqq,fill opacity=0.1] (0,0) -- (-60.14469213169347:0.6315471617947253) arc (-60.14469213169347:59.855307868306554:0.6315471617947253) -- cycle;
\draw [shift={(-1.1546710824314497,-4.988335938636718)},color=qqwuqq,fill=qqwuqq,fill opacity=0.1] (0,0) -- (179.85530786830657:0.6315471617947253) arc (179.85530786830657:299.85530786830657:0.6315471617947253) -- cycle;
\draw [shift={(10.942978763636367,-6.081821952727287)},color=qqwuqq,fill=qqwuqq,fill opacity=0.1] (0,0) -- (19.487907600526214:0.6315471617947253) arc (19.487907600526214:178.99491399474581:0.6315471617947253) -- cycle;
\draw [shift={(0.8452994616207482,2.)},color=qqwuqq,fill=qqwuqq,fill opacity=0.1] (0,0) -- (60.:0.6315471617947253) arc (60.:180.:0.6315471617947253) -- cycle;
\draw [shift={(0.8452994616207482,2.)},color=qqwuqq,fill=qqwuqq,fill opacity=0.1] (0,0) -- (-60.:0.6315471617947253) arc (-60.:60.:0.6315471617947253) -- cycle;
\draw [shift={(-1.8452994616207485,2.)},color=qqwuqq,fill=qqwuqq,fill opacity=0.1] (0,0) -- (120.:0.6315471617947253) arc (120.:240.:0.6315471617947253) -- cycle;
\draw [shift={(-1.8452994616207485,2.)},color=qqwuqq,fill=qqwuqq,fill opacity=0.1] (0,0) -- (-120.:0.6315471617947253) arc (-120.:0.:0.6315471617947253) -- cycle;
\draw [shift={(-1.8452994616207485,2.)},color=qqwuqq,fill=qqwuqq,fill opacity=0.1] (0,0) -- (0.:0.6315471617947253) arc (0.:120.:0.6315471617947253) -- cycle;
\draw [shift={(10.5,2.5566243270259372)},color=qqwuqq,fill=qqwuqq,fill opacity=0.1] (0,0) -- (30.:0.47366037134604394) arc (30.:150.:0.47366037134604394) -- cycle;
\draw [shift={(10.5,2.5566243270259372)},color=qqwuqq,fill=qqwuqq,fill opacity=0.1] (0,0) -- (-90.:0.47366037134604394) arc (-90.:30.:0.47366037134604394) -- cycle;
\draw [shift={(10.5,1.4433756729740648)},color=qqwuqq,fill=qqwuqq,fill opacity=0.1] (0,0) -- (90.:0.47366037134604394) arc (90.:210.:0.47366037134604394) -- cycle;
\draw [shift={(10.5,1.4433756729740648)},color=qqwuqq,fill=qqwuqq,fill opacity=0.1] (0,0) -- (-150.:0.47366037134604394) arc (-150.:-30.:0.47366037134604394) -- cycle;
\draw [shift={(10.5,1.4433756729740648)},color=qqwuqq,fill=qqwuqq,fill opacity=0.1] (0,0) -- (-30.:0.47366037134604394) arc (-30.:90.:0.47366037134604394) -- cycle;
\draw [line width=1.2pt,color=ffqqqq] (-3.,4.)-- (-1.8452994616207485,2.);
\draw [line width=1.2pt,color=ffqqqq] (-1.8452994616207485,2.)-- (-3.,0.);
\draw [line width=1.2pt,color=ffqqqq] (-1.8452994616207485,2.)-- (0.8452994616207482,2.);
\draw [line width=1.2pt,color=ffqqqq] (0.8452994616207482,2.)-- (2.,0.);
\draw [line width=1.2pt,color=ffqqqq] (0.8452994616207482,2.)-- (2.,4.);
\draw [line width=1.2pt,color=ffqqqq] (8.,0.)-- (10.5,1.4433756729740648);
\draw [line width=1.2pt,color=ffqqqq] (10.5,1.4433756729740648)-- (13.,0.);
\draw [line width=1.2pt,color=ffqqqq] (8.,4.)-- (10.5,2.5566243270259372);
\draw [line width=1.6pt,color=ffqqqq] (10.5,2.5566243270259372)-- (10.5,1.4433756729740648);
\draw [line width=1.6pt,color=ffqqqq] (10.5,2.5566243270259372)-- (13.,4.);
\draw [line width=1.2pt,color=ffqqqq] (-5.993144396363616,-4.062242412727284)-- (-3.0877843563636187,-4.983454132727285);
\draw [line width=1.2pt,color=ffqqqq] (-3.0877843563636187,-4.983454132727285)-- (-1.1546710824314497,-4.988335938636718);
\draw [line width=1.2pt,color=ffqqqq] (0.,-7.)-- (-1.1546710824314497,-4.988335938636718);
\draw [line width=1.2pt,color=ffqqqq] (-1.1546710824314497,-4.988335938636718)-- (0.,-3.);
\draw [line width=1.2pt,color=ffqqqq] (6.,-5.)-- (8.92339922363637,-6.046390732727287);
\draw [line width=1.2pt,color=ffqqqq] (8.92339922363637,-6.046390732727287)-- (10.942978763636367,-6.081821952727287);
\draw [line width=1.2pt,color=ffqqqq] (10.942978763636367,-6.081821952727287)-- (14.,-5.);
\begin{scriptsize}
\draw [fill=qqqqff] (-3.,4.) circle (2.5pt);
\draw [fill=qqqqff] (-3.,0.) circle (2.5pt);
\draw [fill=qqqqff] (2.,0.) circle (2.5pt);
\draw [fill=qqqqff] (2.,4.) circle (2.5pt);
\draw [fill=qqqqff] (8.,4.) circle (2.5pt);
\draw [fill=qqqqff] (8.,0.) circle (2.5pt);
\draw [fill=qqqqff] (13.,0.) circle (2.5pt);
\draw [fill=qqqqff] (13.,4.) circle (2.5pt);
\draw [fill=qqqqff] (10.5,2.5566243270259372) circle (2.5pt);
\draw [fill=qqqqff] (-5.993144396363616,-4.062242412727284) circle (2.5pt);
\draw [fill=qqqqff] (-3.0877843563636187,-4.983454132727285) circle (2.5pt);
\draw [fill=qqqqff] (0.,-3.) circle (2.5pt);
\draw [fill=qqqqff] (0.,-7.) circle (2.5pt);
\draw [fill=qqqqff] (6.,-5.) circle (2.5pt);
\draw [fill=qqqqff] (8.92339922363637,-6.046390732727287) circle (2.5pt);
\draw [fill=qqqqff] (10.942978763636367,-6.081821952727287) circle (2.5pt);
\draw [fill=qqqqff] (14.,-5.) circle (2.5pt);
\draw [fill=qqqqff] (-1.8452994616207485,2.) circle (2.5pt);
\draw [fill=qqqqff] (0.8452994616207482,2.) circle (2.5pt);
\draw [fill=qqqqff] (10.5,1.4433756729740648) circle (2.5pt);
\draw [fill=qqqqff] (-1.1546710824314497,-4.988335938636718) circle (2.5pt);

\draw[color=qqwuqq] (-3.7439657320986438,-4.141021438340103) node {$\geq 2\pi/3$};
\draw[color=qqwuqq] (-1.72295442176550093,-4.064176154519576) node {$2\pi/3$};
\draw[color=qqwuqq] (-0.075070594958054284,-4.87994973541694) node {$2\pi/3$};
\draw[color=qqwuqq] (-1.7282223474962921,-5.795723316314306) node {$2\pi/3$};

\draw[color=qqwuqq] (11.181679381770058, -5.00622897148088) node {$\geq 2\pi/3$};
\draw[color=qqwuqq] (9.38176997065509,   -5.074651613391143) node {$\geq 2\pi/3$};

\draw[color=uuuuuu] (-3.3664493341828933, -7.0458463599268976) node {$(c)$};
\draw[color=uuuuuu] (10.7664493341828933, -7.0458463599268976) node {$(d)$};

\draw[color=qqwuqq] (0.5927285702468283,1.1827218400424312) node {$2\pi/3$};
\draw[color=qqwuqq] (0.5927285702468283,2.9458463599268976) node {$2\pi/3$};
\draw[color=qqwuqq] (2.0242757320415533,2.098495420939796) node {$2\pi/3$};

\draw[color=qqwuqq] (-3.0650072387267977, 2.098495420939796) node {$2\pi/3$};
\draw[color=qqwuqq] (-1.7281921512012813, 1.1827218400424312) node {$2\pi/3$};
\draw[color=qqwuqq] (-1.8281921512012813, 2.9458463599268976) node {$2\pi/3$};

\draw[color=qqwuqq] (10.360668071436914,3.451053893067676) node {$2\pi/3$};
\draw[color=qqwuqq] (11.486947307500848,2.572155792285844) node {$2\pi/3$};
\draw[color=qqwuqq] (9.434419031667992,2.5405784341961075) node {$2\pi/3$};

\draw[color=qqwuqq] (9.439686957398782,1.6458765562219042) node {$2\pi/3$};
\draw[color=qqwuqq] (10.423822787616386,0.7774841106066472) node {$2\pi/3$};
\draw[color=qqwuqq] (11.623762395026365,1.6879897657732217) node {$2\pi/3$};

\draw[color=uuuuuu] (-0.3664493341828933,  -1.0458463599268976) node {$(a)$};
\draw[color=uuuuuu] (10.7664493341828933, -1.0458463599268976) node {$(b)$};

\end{scriptsize}
\end{tikzpicture}

\caption{Locally miminal networks for sets of $4$ points.}
\label{lmn4}

\end{center}
\end{figure}
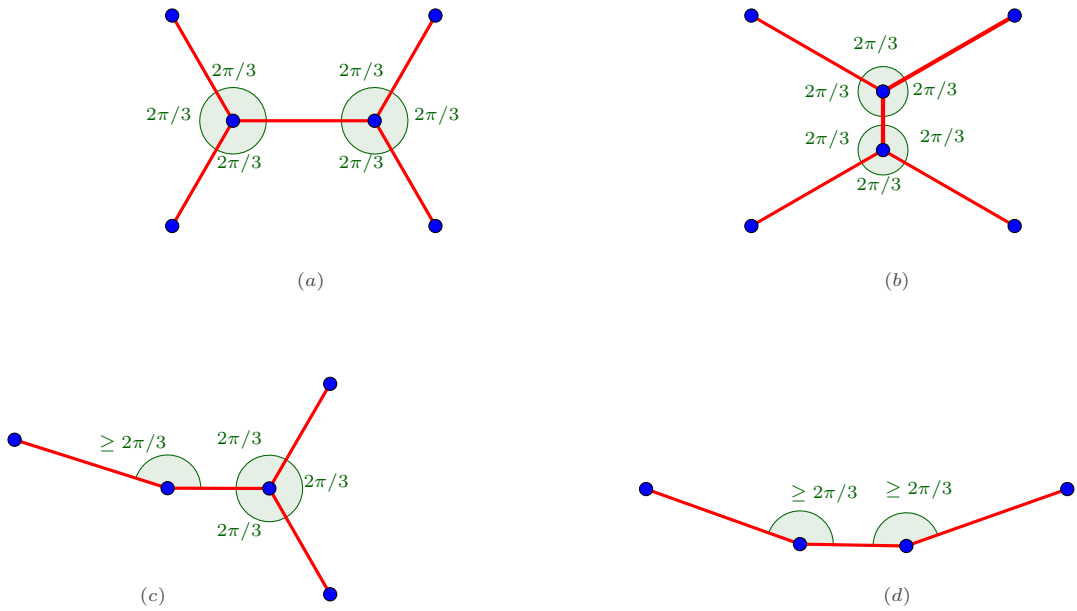


\section*{Acknowledgements} This work is supported by the Russian Science Foundation grant 16-11-10039.
The authors are grateful to Misha Basok, Joseph Gordon, Peter Zograf, Nikita Rastegaev, Fedor Petrov and Eugene Stepanov for numerous useful remarks. Also we are very grateful to Alexander Knop for a help with LaTeX problems.

\bibliographystyle{plain}
\bibliography{main}

\end{document}